\documentclass[10pt,a4paper,subeqn,reqno]{article}

\usepackage{tikz}
\usepackage{array}
\usepackage{tikz-cd}
\usepackage{subcaption}
\usepackage[all]{xy}
\usepackage{amsmath}
\usepackage{amsmath,amsthm}
\usepackage{amssymb}
\usepackage{amssymb,latexsym}
\usepackage{appendix}
\usepackage{authblk}
\usepackage{bm}

\usepackage{caption}
\usepackage{dsfont}
\usepackage{enumerate}
\usepackage{graphicx}
\usepackage{float}
\usepackage{hyperref}
\usepackage[mathscr]{eucal}
\usepackage{mathrsfs}
\usepackage{mathtools}
\usepackage[misc]{ifsym}
\usepackage{url}
\allowdisplaybreaks 
\usepackage[text={165mm,250mm},centering]{geometry}

\newtheorem*{question*}{Question}

\newtheorem{theorem}{Theorem}[section]
\newtheorem{lemma}[theorem]{Lemma}

\newtheorem{corollary}[theorem]{Corollary}
\newtheorem{proposition}[theorem]{Proposition}

\theoremstyle{definition}

\newtheorem{definition}[theorem]{Definition}
\newtheorem{definition-lemma}[theorem]{Definition-Lemma}
\newtheorem{definition-theorem}[theorem]{Definition-Theorem}

\newtheorem{remark}[theorem]{Remark}

\newtheorem{proof*}{Proof}

\numberwithin{equation}{section}

\renewcommand\appendix{\par
\setcounter{section}{}
\setcounter{subsection}{}
\gdef\thesection{Appendix~\Alph{section}}}

\newcommand{\ulin}{\underline}

\newcommand{\wt}{\widetilde}

\newcommand{\wh}{\widehat}

\newcommand{\ov}{\overline}

\newcommand{\mfr}{\mathfrak}
\newcommand{\mc}{\mathcal}
\newcommand{\mbf}{\mathbf}
\newcommand{\mbb}{\mathbb}
\newcommand{\tbf}{\textbf}
\newcommand{\mds}{\mathds}

\newcommand{\C}{{\mathbb C}}

\newcommand{\at}{\mathop{\rm At}\nolimits}
\newcommand{\arccosh}{\mathop{\rm arccosh}\nolimits}
\newcommand{\an}{\mathop{\mathbf{An}}\nolimits}
\newcommand{\set}{\mathop{\mathbf{Sets}}\nolimits}
\newcommand{\rr}{\mathop{\rm R}\nolimits}
\newcommand{\ns}{\mathop{\rm NS}\nolimits}
\newcommand{\kerr}{\mathop{\rm Ker}\nolimits}
\newcommand{\coh}{\mathop{\mathbf{\rm Coh}}\nolimits}

\newcommand{\ccc}{\mathop{\rm c}\nolimits}
\newcommand{\pr}{\mathop{\rm pr}\nolimits}

\newcommand{\num}{\mathop{\rm Num}\nolimits}
\newcommand{\sym}{\mathop{\rm Sym}\nolimits}
\newcommand{\homm}{\mathop{\rm Hom}\nolimits}
\newcommand{\hilb}{\mathop{\rm Hilb}\nolimits}
\newcommand{\ch}{\mathop{\rm ch}\nolimits}
\newcommand{\td}{\mathop{\rm td}\nolimits}
\newcommand{\modd}{\mathop{\rm Mod}\nolimits}
\newcommand{\hh}{\mathop{\mathscr H}\nolimits}
\newcommand{\Hom}{\mathop{\rm{\mathscr H}om}\nolimits}
\newcommand{\Ext}{\mathop{\rm{\mathscr E}xt}\nolimits}
\newcommand{\ext}{\mathop{\rm Ext}\nolimits}
\newcommand{\dimm}{\mathop{\rm dim}\nolimits}
\newcommand{\id}{\mathop{\rm id}\nolimits}

\newcommand{\Tor}{\mathop{\rm\mc Tor}\nolimits}

\newcommand{\tr}{\mathop{\rm tr}\nolimits}
\newcommand{\SL}{\mathop{\tbf{\rm SL}}\nolimits}
\newcommand{\GL}{\mathop{\tbf{\rm GL}}\nolimits}

\newcommand{\dett}{\mathop{\rm det}\nolimits}

\newcommand{\spec}{\mathop{\rm Spec}\nolimits}
\newcommand{\proj}{\mathop{\rm Proj}\nolimits}

\newcommand{\Pic}{\mathop{\rm Pic}\nolimits}

\newcommand{\Sch}{\mathop{\mathbf{Sch}}\nolimits}
\newcommand{\et}{\mathop{\text{\'et}}\nolimits}

\newcommand{\rk}{\mathop{\rm rk}\nolimits}

\newcommand{\supp}{\mathop{\rm Supp}\nolimits}

\renewcommand{\abstractname}{Abstract}

\begin{document}

\title{From Hitchin Systems to Rational Elliptic Surfaces with $\mbb C^*$-actions via Orbifold Hilbert Schemes}

\author[1]{Yonghong Huang\thanks{(\Letter) Email: yonghonghuangmath@gmail.com. }}


\renewcommand\Affilfont{\small}
\affil[1]{Department of Mathematics, Xinjiang University, Urumqi P. R. China}

\date{}

\maketitle
\renewcommand{\abstractname}{Abstract}

\begin{abstract}
Using orbifold Hilbert schemes, we compactify all two-dimensional Hitchin systems corresponding to $\wt A_0$, $\wt D_4$, $\wt E_6$, $\wt E_7$, and $\wt E_8$, thereby obtaining four rational elliptic surfaces with $\mbb{C}^*$-actions. Their singular fibers and relative minimal models are listed in Table~\ref{main table}. A particularly interesting point is that we found they can all be obtained by performing a finite number of blow-ups on the second Hirzebruch surface. To this end, we prove that Hilbert schemes of orbifold surfaces are connected smooth projective schemes, and we use the Hilbert–Chow morphism to construct the minimal resolutions of the coarse moduli spaces.
\end{abstract}


\section{Introduction}
The study of Higgs bundles and Hitchin systems represents a profound and dynamic intersection of algebraic geometry, integrable systems, representation theory, and mathematical physics. Founded on Nigel Hitchin’s seminal work on the two-dimensional reduction of the self-dual Yang--Mills equations (\cite{hi87}), this area has since revealed deep connections between moduli spaces, symplectic geometry, singularity theory, and the geometric Langlands program. A pivotal conjecture put forward by Philip Boalch proposed that the Hilbert scheme of points on a moduli space of meromorphic Higgs bundles should itself carry the structure of a moduli space of such Higgs bundles (\cite{boa11}). This conjecture was verified for cotangent bundles of elliptic curves by Gorsky, Nekrasov, and Rubtsov (\cite{gnr01}), and later established in full generality by Groechenig (\cite{go14}).

Using the Fourier–Mukai transform, Groechenig constructed all two-dimensional Hitchin systems associated with the affine Dynkin diagrams $\widetilde{D}_4$, $\widetilde{E}_6$, $\widetilde{E}_7$, and $\widetilde{E}_8$ through crepant resolutions of the GIT quotients $T^\vee E / \Gamma$, where $E$ is an elliptic curve and $\Gamma$ a finite cyclic group. These systems correspond precisely to moduli spaces of orbifold Higgs bundles over one-dimensional Calabi--Yau orbifolds, including elliptic curves and weighted projective lines with prescribed cyclic stacky structures. Despite this landmark achievement, several fundamental geometric questions remained open: explicit compactifications of these Hitchin systems, their realization as projective algebraic surfaces, the structure of their singular fibers, their relatively minimal models, and their compatibility with $\mbb{C}^*$-actions and Poisson structures.

Orbifold Hilbert schemes provide a natural framework for addressing these gaps. Much of the earlier literature concerns the case of quotients of algebraic surfaces by finite groups, where the orbifold Hilbert scheme coincides with the equivariant Hilbert scheme, allowing many properties to be studied equivariantly. For instance, Bridgeland, King and Miles Reid used $G$-Hilbert schemes to construct minimal resolutions and establish the derived McKay correspondence for finite subgroups of $\SL_2$ (see \cite{bkr01}). Ishii later extended this to general finite small subgroups $G\subset\GL_2$, showing that $\hilb^1([\mbb C^2/G])$ yields the minimal resolution of $\mbb C^2/G$ (see \cite{is02}). For a finite subgroup $G\subset\SL_2$, the Hilbert schemes $\hilb^n([\mbb C^2/G])$ of quotient stack $[\mbb C^2/G]$ can be realized as Nakajima quiver varieties, and their geometry is by now well understood; see, for example, Nakajima’s ICM 2002 report~\cite{na02}. In contrast, when $G \subset \GL_2$, much less is known about the geometry of the corresponding Hilbert schemes. In particular, the quiver-theoretic approach is generally ineffective in this setting, and it is not known whether $\hilb^n([\mbb C^2/G])$ can be realized as a quiver variety for any finite small subgroup $G\subset\GL_2$.

Although many results are already known on moduli spaces of sheaves and Hilbert schemes over orbifold surfaces, much remains to be developed. In this paper, we aim to develop some basic aspects of the theory of moduli spaces of sheaves and Hilbert schemes on orbifold surfaces, and apply this framework to compactify and characterize the aforementioned two-dimensional Hitchin systems. We first investigate wall-crossing phenomena for stability conditions on orbifold surfaces and prove the existence of polarizations ensuring that semistability coincides with stability for generic numerical K-classes. We then show that moduli spaces of stable sheaves are connected, smooth, projective schemes. Crucially, we construct orbifold Hilbert schemes for general projective orbifold surfaces, establish their smoothness and connectedness, and prove that the Hilbert-Chow morphism provides minimal and Poisson resolutions of the coarse moduli spaces, extending classical results beyond the $\SL_2$ case to all finite small subgroups of $\GL_2$.

Applying these general results to the Hitchin systems of types $\widetilde{D}_4$, $\widetilde{E}_6$, $\widetilde{E}_7$, and $\widetilde{E}_8$, we realize their natural compactifications as orbifold Hilbert schemes of projective cotangent bundles. These compactifications are rational elliptic surfaces with $\mathbb{C}^*$-actions whose fibrations possess singular fibers only over $0$ and $\infty$. We explicitly describe these fibers, their dual graphs, and relatively minimal models, demonstrating that all such surfaces arise from iterated blow-ups of the second Hirzebruch surface. In doing so, we complete the geometric classification of these Hitchin systems, extend the scope of orbifold Hilbert scheme theory, and reveal unanticipated links between integrable systems, rational elliptic surfaces, orbifold geometry, and minimal resolutions of singularities.

\subsection{Main results}
For an algebraic normal surface with quotient singularities, there exists an associated smooth Deligne-Mumford stack\textemdash namely, the canonical stack\textemdash whose coarse moduli space is the surface itself, and whose stacky locus has codimension two. Henceforth, in what follows, we only consider the orbifold surface whose stacky locus has codimension two. Let $X$ be an orbifold surface with coarse moduli space $\pi : \mc X\rightarrow X$. In Section~\ref{sec NS gp}, we study the wall-crossing theory on the orbifold surface $\mc X$. The reason for employing the classical wall-crossing theory is that, at present, Bridgeland stability is defined only for orbifold surfaces with rational double point singularities (see \cite{lr22}). We analyze the wall-crossing phenomena for stability conditions on $\mc X$ and establish the existence of a polarization $H$ on $X$ such that, for a generic numerical $K$-class $\upsilon \in K^{\mathrm{num}}(\mc X)$, the moduli space $\mc M_\upsilon$ of semistable sheaves contains no strictly semistable objects (see Theorem~\ref{thm change of plarization} and Lemma~\ref{lemm no str ss}). Assuming additionally that either $K_\mc X \cdot \pi^*H < 0$ or $K_\mc X \cong \mc O_\mc X$, we show that $\mc M_\upsilon$ is a connected, smooth, projective scheme (Proposition~\ref{pro sm of mod}, Corollary~\ref{thm connectedness of moduli}). Moreover, if $\mc X$ is equipped with a Poisson structure, then $\mc M_\upsilon$ naturally inherits a compatible Poisson structure, thereby forming a connected, smooth, projective Poisson variety (Proposition~\ref{pro poi st on moduli}). To achieve these results, we construct the Atiyah class for smooth Deligne–Mumford stacks (Section~\ref{subsec at class}), in a manner distinct from the general construction for algebraic stacks by Kuhn \cite{ku24}. This construction enables the definition of the Kodaira–Spencer map for $\mc M_\upsilon$, yielding explicit descriptions of its tangent and cotangent bundles (Proposition~\ref{pro ks map}, Corollary~\ref{cor cotang of mod sp}). Furthermore, we express the diagonal class of $\mc M_\upsilon$ in terms of Chern classes and prove that its cohomology ring is generated by the Künneth components of the orbifold Chern character associated with a universal sheaf (Theorem~\ref{thm genetators of moduli}). In Section~\ref{sec orb Hilb schemes}, we carry out a systematic study of the Hilbert scheme $\hilb^{\alpha}(\mc X)$ for $\alpha \in \mc N$ (notation as in Corollary~\ref{cor lattice for hilb}). These investigations culminate in the following theorem:

\begin{theorem}
\begin{enumerate}[$(1)$]
\item $\hilb^n(\mc X)$ is a smooth connected projective scheme.
\item  The Hilbert-Chow morphism $h : \hilb^{n}(\mc X)\rightarrow\sym^n(X)$ is a resolution of singularities. 
\item $h : \hilb^{1}(\mc X)\rightarrow X$ is the minimal resolution.
\item If $\mc X$ equipped with a Poisson structure, then $h$ is a Poisson resolution with respect to the induced Poisson structures.
\end{enumerate}
\end{theorem}
\begin{corollary}
Suppose that $W$ is a smooth connected quasiprojective scheme equipped with an action of a finite group $G$. If the fixed locus of the action is zero-dimensional, then
\begin{equation*}
		\hilb^n([W/G])=\{Z\subseteq W|\text{$Z$ is a $G$-invariant closed subscheme with $H^0(\mc O_Z)\cong\rho_{\rm reg}^{\oplus n}$}\}
\end{equation*} 
is a smooth connected quasiprojective scheme. Moreover, all the conclusions of the above theorem remain valid for $[W/G]$.
\end{corollary}

\begin{corollary}
Let $X$ be an irreducible symplectic projective surface with quotient singularities and let $\mc X$ be its associated canonical stack. Then the Hilbert-Chow morphism 
\[
	h :	\hilb^n(\mc X)\longrightarrow\sym^n(X)
\]
is a symplectic resolution.
\end{corollary}

As there exists a natural morphism 
\begin{equation}\label{equ res of hilb}
	\hilb^n(\mc X)\longrightarrow\hilb^{n}(X)
\end{equation}
(see (\ref{equ mor orb to coar}) in Section \ref{subsec HC}), it serves as a resolution whenever $X$ possesses only rational double singularities. In fact, a notable result is due to Zheng \cite{ze23}, who proved that the Hilbert scheme of points on a normal quasi-projective surface with at worst rational double point singularities is irreducible. However, for general quotient singularities, this statement does not hold; in fact, Zheng constructed explicit counterexamples in the same paper.

Using quiver-theoretic methods, Craw, Gammelgaard, Gyenge, and Szendr\H{o}i showed that, in the case $\mc X = [\mathbb C^2/G]$ with $G \subset \SL_2$, the morphism~\eqref{equ res of hilb} provides the unique projective symplectic resolution; see~\cite{cggs21}. This motivates the following question:

\begin{question*}
	Let $\mc X$ be an irreducible symplectic surface with quotient singularities. Its associated canonical stack $\mc X$ is then a symplectic orbifold surface. Does the morphism~\eqref{equ res of hilb} give the unique projective symplectic resolution (up to isomorphism)?
\end{question*}

We also obtain the following proposition.

\begin{proposition}
For any finite small group $G\subset\GL_2$, $\hilb^n([\spec (\mbb C[x,y])/G])$ and $\hilb^n([\spec (\mbb C[\![x,y]\!])/G])$ are connected. 
\end{proposition}
For $G = \{\mathrm{id}\} \subset \GL_2$, the connectedness of $\hilb^n([\spec(\mathbb C[\![x,y]\!])/G])$ was established by Brian\c{c}on, who in fact proved that $\hilb^n(\mathbb C[\![x,y]\!])$ is irreducible (\cite{br77}). Crawley-Boevey showed that $\hilb^n([\mathbb C^2/G])$ is connected for any finite subgroup $G \subset \SL_2$ (\cite{cr01}). More recently, Bejleri and Zaimi established the connectedness of $\hilb^n([\mathbb C^2/G])$ in the case where $G$ is abelian (see \cite[Corollary 2.3]{bz23}).

In Section \ref{sec cpt hitchin systems}, we compactify the two-dimensional Hitchin systems via orbifold Hilbert schemes. Precisely, the two-dimensional Hitchin systems corresponding to the affine Dynkin diagrams $\wt A_0$, $\wt D_4$, $\wt E_6$, $\wt E_7$ and $\wt E_8$ are the moduli spaces of orbifold Higgs bundles on 
\begin{equation*}
E,\quad \mds P_{2,2,2,2}^1,\quad \mds P_{3,3,3}^1,\quad \mds P_{4,4,2}^1\quad \text{and}\quad \mds P_{6,3,2}^1
\end{equation*}
respectively, where $E$ is an elliptic curve and $\mds P^1_{a_1,\cdots,a_s}$ denotes a projective line with $s$ orbifold points of the specified orders. These are precisely the one-dimensional Calabi-Yau orbifolds, and, except for $E$, each can be realized as a nontrivial quotient stack $\mc X_i:=[E_i/\mu_i]$ of a certain elliptic curve $E_i$ by cyclic group $\mu_i$ for $i=2,3,4,6$. In the case of $\wt A_0$, the Higgs moduli space is $T^\vee E$, and its natural compactification is $\mds P(T^\vee E\oplus\mc O_E)$. Therefore we restrict our attention to the remaining cases.

\begin{theorem}\label{thm cpf of two higgs}
The two-dimensional Hitchin systems for $\wt D_4$, $\wt E_6$, $\wt E_7$, and $\wt E_8$ admit natural compactifications 
\begin{equation}\label{equ cpt of hitch}
\hilb^{1}\!\bigl(\mds P(T^\vee\mc X_i\oplus\mc O_{\mc X_i})\bigr)
\end{equation}
with the following properties:
\begin{enumerate}[$(1)$]
\item The natural $\mbb{C}^*$-action and Poisson structure on $\hilb^1(\mds P(T^\vee\mc X_i \oplus \mc O_{\mc X_i}))$ are compatible with, and extend, the $\mbb{C}^*$-action and symplectic structure on $\mc M(i)$.
\item The Hitchin maps extend to the compositions
\begin{equation}\label{equ cpt Hitch map}
\begin{tikzcd}
\hilb^1\!\bigl(\mds P(T^\vee\mc X_i\oplus\mc O_{\mc X_i})\bigr)\arrow[r, "h_i"] \arrow[rr,bend left=15, "\pi_i"] & E\times\mds P^1/\mu_i \arrow[r] & \mds P^1/\mu_i \cong \mds P^1,
\end{tikzcd}
\end{equation}
where $h_i$ are the Hilbert–Chow morphisms. 
\item Each $h_i$ is the minimal resolution of the GIT quotient
\begin{equation}\label{equ cpt of hitch 1}
		 \mds P(T^\vee E_i\oplus\mc O_{E_i})/\mu_i,
\end{equation}
and provides a Poisson resolution.
\item The boundary (with reduced structure) consists of $s+1$ copies of $\mds P^1$, where $s$ is the number of orbifold points of $\mc X_i$.
\end{enumerate}

\end{theorem}

For brevity, we denote the GIT quotient (\ref{equ cpt of hitch 1}) and its compactifications (\ref{equ cpt of hitch}) by $X_i$ and $\wt X_i$, respectively.

\begin{theorem}
Each $\wt X_i$ is a rational elliptic surface with a $\mathbb{C}^*$-action, whose fibration $\pi_i : \wt X_i \to \mathbb{P}^1$ has singular fibers only over $0$ and $\infty$, as summarized in Table \ref{main table}. Each $\wt X_i$ is obtained by blowing up the second Hirzebruch surface (see Propositions \ref{pro stru of ell 2}, \ref{prop stru for ell 3}, \ref{pro ell 4 str}, \ref{pro ell 6 str}), and the Hitchin system $M(i)$ is isomorphic to $\wt X_i$ with the fiber over $\infty$ removed.
\end{theorem}

\begin{table}[H]
\begin{center}
\begin{tabular}{|m{0.5cm}|m{2.5cm}|m{1.8cm}|m{3.9cm}|m{2.8cm}|m{2.5cm}|}
\hline
\textbf{} & \textbf{Fiber over 0} & \textbf{Fiber over $\infty$}&\textbf{Dual graph (over $0$)} & \textbf{Dual graph (over $\infty$)} & \textbf{Relatively minimal model} \\ 

\hline
			
$\wt X_2$ & $2D_0+E_1+E_2+E_3+E_4$ & $2D_\infty+F_1+F_2+F_3+F_4 $ & 

\begin{center}
\begin{tikzpicture}[scale=0.25][font=\tiny]
	
\node at (-5,0) {$\mathrm{I}_0^* (\wt D_4)$};
	
\node[circle, fill=black, inner sep=1.5pt] (A) at (0,0) {};
\node[circle, fill=black, inner sep=1.5pt] (B) at (-2,2) {};
\node[circle, fill=black, inner sep=1.5pt] (C) at (2,2) {};
\node[circle, fill=black, inner sep=1.5pt] (D) at (-2,-2) {};
\node[circle, fill=black, inner sep=1.5pt] (E) at (2,-2) {};

\draw [thick] (A)--(B) (A)--(C) (A)--(D) (A)--(E);

\node at (A)  [yshift=-10pt]  {$-2$};
\node[left] at (B) {$-2$};
\node[right] at (C) {$-2$};
\node[left] at (D) {$-2$};
\node[right] at (E) {$-2$};

\end{tikzpicture}
\end{center} 
&
\begin{center}

\begin{tikzpicture}[scale=0.25][font=\tiny]
	
\node at (-5,0) {$\mathrm{I}_0^* (\wt D_4)$};		
\node[circle, fill=black, inner sep=1.5pt] (A) at (0,0) {};
\node[circle, fill=black, inner sep=1.5pt] (B) at (-2,2) {};
\node[circle, fill=black, inner sep=1.5pt] (C) at (2,2) {};
\node[circle, fill=black, inner sep=1.5pt] (D) at (-2,-2) {};
\node[circle, fill=black, inner sep=1.5pt] (E) at (2,-2) {};
\draw [thick] (A)--(B) (A)--(C) (A)--(D) (A)--(E);

\node at (A) [yshift=-10pt] {$-2$};
\node[left] at (B) {$-2$};
\node[right] at (C) {$-2$};
\node[left] at (D) {$-2$};
\node[right] at (E) {$-2$};
	\end{tikzpicture}
\end{center} 
& Already relatively minimal \\ \hline
			
$\wt X_3$ & $3D_0+2E_{11}+E_{21}+2E_{12}+E_{22}+2E_{13}+E_{23}$ & $3D_\infty+F_1+F_2+F_3$ &

\begin{center}
\begin{tikzpicture}[scale=0.6][font=\tiny]
		
\node at (-2,-1) {$\mathrm{IV}^* (\wt E_6)$};

\node[circle, fill=black, inner sep=1.5pt] (A) at (-2,0) {};
\node[circle, fill=black, inner sep=1.5pt] (B) at (-1,0) {};
\node[circle, fill=black, inner sep=1.5pt] (C) at (0,0) {};
\node[circle, fill=black, inner sep=1.5pt] (D) at (1,0) {};
\node[circle, fill=black, inner sep=1.5pt] (G) at (2,0) {};
\node[circle, fill=black, inner sep=1.5pt] (E) at (0,-1) {};
\node[circle, fill=black, inner sep=1.5pt] (F) at (0,-2) {};
		
\draw [thick] (A) -- (B) -- (C) -- (D)--(G);
\draw [thick] (C) -- (E) -- (F);
	\node [above left] at (A)  {-2};
	\node [above]   at (B)  {-2};
	\node [above]   at (C)  {-2};
	\node  [above]  at (D)  {-2};
	\node  [above]  at (G) {-2};
	\node  [right]  at (E)  {-2};
	\node  [right]  at (F) {-2};
\end{tikzpicture}
\end{center}

&
\begin{center}
\begin{tikzpicture}[scale=0.3][font=\tiny]
\node[circle, fill=black, inner sep=1.5pt] (A) at (0,2)    {};
\node[circle, fill=black, inner sep=1.5pt] (B) at (-2,2)   {};
\node[circle, fill=black, inner sep=1.5pt] (C) at (2,2)    {};
\node[circle, fill=black, inner sep=1.5pt] (D) at (0, 4)   {};
		
\draw [thick] (A)--(B);
\draw [thick](C)--(A);
\draw [thick](A)--(D);
		

\node [below]   at (A)  {$-1$};
\node [left]    at (B)  {$-2$};
\node [right]   at (C)  {$-2$};
\node [above]    at (D)  {$-2$};
		
\end{tikzpicture}
\end{center}
& Blow down $D_\infty$: fiber over $\infty$ becomes $\mathrm{IV} (\wt A_2)$ \\

\hline
			
$\wt X_4$ & $4D_0+3E_{11}+2E_{21}+E_{31}+3E_{12}+2E_{22}+E_{32}+2E$ & $4D_\infty+F_1+F_2+2F$ &

\begin{center}
\begin{tikzpicture}[scale=0.5][font=\tiny]

\node at (-2,-1) {$\mathrm{III}^* (\wt E_7)$};
\node[circle, fill=black, inner sep=1.5pt] (A) at (-2,0) {};
\node[circle, fill=black, inner sep=1.5pt] (B) at (-1,0) {};
\node[circle, fill=black, inner sep=1.5pt] (C) at (0,0) {};
\node[circle, fill=black, inner sep=1.5pt] (D) at (1,0) {};
\node[circle, fill=black, inner sep=1.5pt] (G) at (2,0) {};
\node[circle, fill=black, inner sep=1.5pt] (E) at (3,0) {};
\node[circle, fill=black, inner sep=1.5pt] (F) at (4,0) {};
\node[circle, fill=black, inner sep=1.5pt] (H) at (1,-1) {};

\draw [thick] (A) -- (B) -- (C) -- (D) -- (G) -- (E) -- (F);
\draw  [thick] (D)--(H);
		
\node  [above]   at (A)   {$-2$};
\node  [above]   at (B)   {$-2$};
\node  [above]   at (C)   {$-2$};
\node  [above]   at (D)   {$-2$};
\node  [above]   at (G)   {$-2$};
\node  [above]  at (E)    {$-2$};
\node  [above]  at (F)    {$-2$};
\node  [left]   at (H)    {$-2$};


\end{tikzpicture}
\end{center}

&
\begin{center}
\begin{tikzpicture}[scale=0.25][font=\tiny]
		\node[circle, fill=black, inner sep=1.5pt] (A) at (0,0)  {};
		\node[circle, fill=black, inner sep=1.5pt] (B) at (3,0)  {};
		\node[circle, fill=black, inner sep=1.5pt] (C) at (6,0)  {};
		\node[circle, fill=black, inner sep=1.5pt] (D) at (3,2)  {};
		
		\draw [thick] (A)--(B)--(C);
		\draw [thick] (B)--(D);
	    \node [left]       at (A)  {$-4$};
		\node [below]      at (B)  {$-1$};
		\node [right]      at (C)  {$-4$};
		\node [above]      at (D)   {$-2$};

		
\end{tikzpicture}
\end{center} 		
& Blow down two $(-1)$-curves over $\infty$: fiber becomes $\mathrm{III}(\wt A_1)$\\ 

\hline
			
$\wt X_6$ & $6D_0+E_5+2E_4+3E_3+4E_2+5E_1+4E_6+2E_7+3E_8$ & $6D_\infty+F_1+2F_2+3E_9$ &

\begin{center}
\begin{tikzpicture}[scale=0.5][font=\tiny]
	
\node at (-1,-1) {$\mathrm{II}^* (\wt E_8)$};		
\node[circle, fill=black, inner sep=1.5pt] (A) at (-2,0) {};
\node[circle, fill=black, inner sep=1.5pt] (B) at (-1,0) {};
\node[circle, fill=black, inner sep=1.5pt] (C) at (0,0) {};
\node[circle, fill=black, inner sep=1.5pt] (D) at (1,0) {};
\node[circle, fill=black, inner sep=1.5pt] (G) at (2,0) {};
\node[circle, fill=black, inner sep=1.5pt] (E) at (3,0) {};
\node[circle, fill=black, inner sep=1.5pt] (F) at (4,0) {};
\node[circle, fill=black, inner sep=1.5pt] (H) at (5,0) {};
\node[circle, fill=black, inner sep=1.5pt] (I) at (3,-1) {};
		
\draw [thick] (A) -- (B) -- (C) -- (D) -- (G) -- (E) -- (F)--(H);
\draw  [thick] (E)--(I);
		
		\node [above]    at (A)  {$-2$};
		\node [above]    at (B)  {$-2$};
		\node [above]    at (C)  {$-2$};
		\node [above]    at (D)  {$-2$};
		\node [above]    at (G)  {$-2$};
		\node [above]    at (E)  {$-2$};
		\node [above]    at (F)  {$-2$};
		\node [above]    at (H)  {$-2$};
		\node [left]     at (I)  {$-2$};
\end{tikzpicture}
\end{center}

&

\begin{center}
\begin{tikzpicture}[scale=0.25][font=\tiny]
		\node[circle, fill=black, inner sep=1.5pt] (A) at (0,0)  {};
		\node[circle, fill=black, inner sep=1.5pt] (B) at (3,0)  {};
		\node[circle, fill=black, inner sep=1.5pt] (C) at (6,0)  {};
		\node[circle, fill=black, inner sep=1.5pt] (D) at (3,2)  {};
		
		\draw [thick] (A)--(B)--(C);
		\draw [thick] (B)--(D);
		\node [left]       at (A)  {$-6$};
		\node [below]      at (B)  {$-1$};
		\node [right]      at (C)  {$-2$};
		\node [above]      at (D)   {$-3$};
		
	
\end{tikzpicture}
\end{center}

								

& Blow down three $(-1)$-curves over $\infty$: fiber becomes $\mathrm{II}$ \\
\hline
\end{tabular}
\end{center}
\caption{The singular fibers of $\pi_i$ (the numbers in dual graphs denote self-intersection).}
\label{main table}
\end{table}

\subsection{Notations and conventions}\label{sec: conven}
All schemes and Deligne-Mumford stacks are defined over $\mathbb C$. Deligne-Mumford stacks are assumed to be of finite type over $\mathbb C$, with finite diagonal, and with quasiprojective coarse moduli schemes. An orbifold is a smooth, connected Deligne-Mumford stack with trivial generic stabilizer. Unless otherwise specified, we assume that $\mc X$ is a projective orbifold surface whose stacky locus has codimension two. The set of stacky points of $\mc X$ is denoted by $\{p_k\}_{k\in\mfr I}$, whose images under the coarse map $\pi : \mc X\rightarrow X$ are also written as $p_k$. We write $K_0(\mc X)_\mathbb C:=K_0(\mc X)\otimes\mbb C$ for the Grothendieck group of coherent sheaves, $T\mc X$ for the tangent sheaf, and $K_\mc X$ for the canonical line bundle when $\mc X$ is smooth. The category of quasicoherent sheaves is $\rm Qcoh$, and $\mc X_{\text{\'et}}$ denotes its small \'etale site. For Deligne-Mumford stacks $\mc X$ and $\mc Y$, let $\pr_1$ and $\pr_2$ be the natural projections from $\mc X\times\mc Y$. For a locally free sheaf $E$ on $\mc X$, the projective bundle $\mds P(E)$ is ${\bm\proj}(\rm Sym^\bullet E^\vee)$.

Some basic facts on finite group representations: let $G$ be a finite group and $K(G)$ its representation ring over $\mathbb Z$. Enumerating the irreducible complex representations as $\{\rho_0,\ldots,\rho_t\}$ with $\rho_0$ trivial, we have $K(G)=\bigoplus_{i=0}^t \mathbb Z \rho_i$. The character $\chi_\rho$ of a representation $\rho$ is $\chi_\rho(g)=\tr(\rho(g))$, and the inner product of two characters is
\[
\langle \chi_\rho,\chi_\varrho\rangle = \frac{1}{|G|}\sum_{g\in G} \tr(\rho(g))\overline{\tr(\varrho(g))}.
\]
In particular, the first orthogonality relation reads $\langle\chi_{\rho_i},\chi_{\rho_j}\rangle = \delta_{ij}$.

\section{N\'{e}ron-Severi group and change of polarization}\label{sec NS gp}
In this section we study the change of polarization for $\mc X$ and formulate general results that will be needed later. For a first reading, this section may be skipped and consulted as necessary.

\subsection{N\'{e}ron-Severi group}
We assume that $\mc X$ is a connected normal projective Deligne-Mumford stack in this subsection.
Recall that the Picard functor
\begin{equation*}
	\Pic_\mc X : \Sch\rightarrow\set,\quad T\mapsto\frac{\Pic(\mc X\times T)}{\Pic(T)}.
\end{equation*}
is representable by a scheme $\Pic(\mc X)$ and its connected component of identity $\Pic^0(\mc X)$ is an abelian varity of dimension $\dimm_{\mbb C}H^1(\mc X,\mc O_{\mc X})$ (see \cite[Corollary 2.3.6]{bro12} and \cite[Th\'{e}or\`{e}me 1.3]{bro09}).

\begin{definition}\label{def of ns 1}
The N\'{e}ron-Severi group of $\mc X$ is defined as
\begin{equation*}
\ns(\mc X):=\Pic(\mc X)(\mbb C)/\Pic^0(\mc X)(\mbb C).
\end{equation*}
\end{definition}

\begin{theorem}[Brochard]
The group $\ns(\mc X)$ is of finite type. Moreover, if we regard $\ns(\mc X)$ as a constant group scheme, then there is an exact sequence
	\begin{equation*}
		\xymatrix@C=0.5cm{
			0 \ar[r] & \Pic^0(\mc X) \ar[r]  & \Pic(\mc X) \ar[r]  & \ns(\mc X) \ar[r] & 0 }.
	\end{equation*}
	\qed
\end{theorem}

For $\mc X$, the associated analytic Deligne–Mumford stack $\mc X^{\rm an}$ is defined in \cite[D\'{e}finition 5.2]{toen99(1)}.  
The analytification functor
\begin{equation*}
\coh(\mc X)\to\coh(\mc X^{\rm an}),\quad F\mapsto F^{\rm an}
\end{equation*}
is an equivalence of abelian categories (GAGA for DM stacks, \cite[Th\'{e}or\`{e}me 5.10]{toen99(1)}), hence $\Pic(\mc X)\cong\Pic(\mc X^{\rm an})$.  
Moreover, by \cite[proof of Th\'{e}or\`{e}me 3.3]{gro60}, the exponential sequence
\begin{equation}\label{equ exact seq of exp}
\begin{tikzcd}[column sep=1.2em]
0 \arrow[r] & \underline{\mbb Z} \arrow[r] & \mc O_\mc {X^{\rm an}} \arrow[r] & \mc O_{\mc X^{\rm an}}^* \arrow[r] & 0 	
\end{tikzcd}	
\end{equation}
holds for $\mc X^{\rm an}$. The associated long exact sequence is 
\begin{equation}\label{equ exp exact seq}
\begin{tikzcd}[column sep=1.1em]
0 \arrow [r] & H^1(\mc X^{\rm an},\underline{\mbb Z}) \arrow[r] & H^1(\mc X^{\rm an},\mc O_{\mc X^{\rm an}}) \arrow[r] & \Pic(\mc X^{\rm an})\arrow[r,"\ccc_1"] & H^2(\mc X^{\rm an},\underline{\mbb Z}) \arrow[r] & H^2(\mc X^{\rm an},\mc O_\mc X)
\end{tikzcd}
\end{equation}
where $\ccc_1(L)$ is the first Chern class of $L\in\Pic(\mc X^{\rm an})$.

\begin{definition}\label{def ns group}
The N\'{e}ron-Severi group of $\mc X^{\rm an}$ is
\begin{equation*}
 \ns(\mc X^{\rm an}):=\Pic(\mc X^{\rm an})/\Pic^0(\mc X^{\rm an})
\end{equation*}
where $\Pic^0(\mc X^{\rm an})$ is $\kerr(c_1)\cong H^1(\mc X^{\rm an},\mc O_{\mc X^{\rm an}})/H^1(\mc X^{\rm an},\underline{\mbb Z})$.
\end{definition}
\begin{remark}\label{rm def of ns}
If $\mc X$ is a projective orbifold surface, $\Pic^0(\mc X^{\rm an})$ is a compact complex torus of dimension $b_3/2$ where $b_3$ is the third Betti number of $X$ (see (\romannumeral 1), (\romannumeral 2) and (\romannumeral 5) in Section \ref{subsec hodge indx}).
\qed
\end{remark}

\begin{definition}\label{def an pic sf}
Let $\an$ be the category of analytic spaces. We consider the functor
\begin{equation}
\ulin{\Pic}_{\mc X^{\rm an}} : \an \longrightarrow \set,\quad S\longmapsto H^0(S,R^1p_{S*}\mc O^*_{\mc X^{\rm an}\times S})
\end{equation}
	where $p_S : \mc X^{\rm an}\times S\rightarrow S$ is the projection to $S$.
\end{definition}
\begin{remark}\label{rem pic presf}
If $\an$ is equipped with the usual analytic topology, then $\ulin{\Pic}_{\mc X^{\rm an}}$ is the sheaf associated to the presheaf
\begin{equation*}
\an \longrightarrow \set,\quad S\longmapsto \Pic(\mc X^{\rm an}\times S).
\end{equation*}
\end{remark}

\begin{proposition}\label{pro rep of an pic}
The functor $\ulin{\Pic}_{\mc X^{\rm an}}$ is represented by the abelian analytic group $\Pic(\mc X^{\rm an})$, fitting into the exact sequence
\begin{equation*}
\begin{tikzcd}[column sep=1.5em]
0 \arrow[r] & \Pic^0(\mc X^{\rm an}) \arrow[r] & \Pic(\mc X^{\rm an})
\arrow[r] & \ns(\mc X^{\rm an}) \arrow[r] & 0,
\end{tikzcd}
\end{equation*}
where $\ns(\mc X^{\rm an})$ is discrete.
\end{proposition}

\begin{proof}
The proof of this proposition is the same as the proof of Th\'{e}or\`{e}me 3.3 in \cite{gro60}.
\end{proof}

\begin{corollary}\label{cor eq of an gps}
 As analytic groups, $\Pic(\mc X)^{\rm an}\cong\Pic(\mc X^{\rm an})$ and $\Pic^0(\mc X)^{\rm an}\cong\Pic^0(\mc X^{\rm an})$.
\end{corollary}
\begin{proof}
Let $\mc P$ be the Poincar\'{e} line bundle on $\mc X\times\Pic(\mc X)$.  
Its analytification $\mc P^{\rm an}$ induces a bijective morphism of analytic groups $a : \Pic(\mc X)^{\rm an}\rightarrow \Pic(\mc X^{\rm an})$. Since $\Pic^0(\mc X)$ is an abelian variety \cite[Th\'{e}or\`{e}me 1.3]{bro09}, 
$\Pic^0(\mc X)^{\rm an}$ is a compact complex torus \cite[Proposition 3.2, Expos\'{e} XII]{sga1}, 
and it is the identity component of $\Pic(\mc X)^{\rm an}$ \cite[Corollaire 2.6, Expos\'{e} XII]{sga1}. Restricting $a$ gives $a' : \Pic^0(\mc X)^{\rm an}\rightarrow \Pic^0(\mc X^{\rm an})$, a holomorphic bijection between compact complex tori. Hence $a'$ is an isomorphism \cite[p.\,19]{gh94}, and therefore $a$ is an isomorphism of analytic groups.	
\end{proof}

\begin{corollary}\label{cor eq of ns gp}
	$\ns(\mc X)\cong\ns(\mc X^{\rm an})$.
	\qed
\end{corollary}

\begin{proposition}\label{pro equ of pic gp}
	$\Pic^0(\mc X)$ is isomorphic to $\Pic^0(X)$ as group schemes.
\end{proposition}
\begin{proof}

The coarse moduli space of $\mc X^{\rm an}$ is $X^{\rm an}$ which is the analytification of $X$ (see \cite{toen99(1)}). Since $\pi_*\ulin{\mbb Z}=\ulin{\mbb Z}$, $\pi_*\mc O_{\mc X^{\rm an}}=\mc O_{X^{\rm an}}$ and $\pi_*\mc O_{\mc X^{\rm an}}^*=\mc O_{X^{\rm an}}^*$, we have the following diagram where horizontal and vertical sequences are exact.
\begin{equation}\label{equ diag of exp seq}
\begin{split}
\xymatrix@C=0.3cm{
&            0                &      0                &        0               &     \\
0 \ar[r] & \rr\pi_*\ulin{\mbb Z}/\ulin{\mbb Z} \ar[u]\ar[r]  & \rr\pi_*\mc O_{\mc X^{\rm an}}/\mc O_{X^{\rm an}}  \ar[u]\ar[r] & \rr\pi_*\mc O_{\mc X^{\rm an}}^*/\mc O_{X^{\rm an}}^* \ar[u]\ar[r] & 0   \\
0 \ar[r] & \rr\pi_*\ulin{\mbb Z}  \ar[u]\ar[r] & \rr\pi_*\mc O_{\mc X^{\rm an}} \ar[u]\ar[r] & \rr\pi_*\mc O_{\mc X^{\rm an}}^* \ar[u]\ar[r] & 0 \\
0 \ar[r] & \ulin{\mbb Z} \ar[u]\ar[r]  & \mc O_{X^{\rm an}}  \ar[u]\ar[r] & \mc O_{X^{\rm an}}^* \ar[u]\ar[r] & 0 \\
& 0 \ar[u] & 0 \ar[u] & 0 \ar[u]}
\end{split}
\end{equation}

For every abelian sheaf $J$ on $\mc X^{\rm an}$, $\mbb H^0(X^{\rm an},\rr\pi_*J/J)=0$ and $\mbb H^i(X^{\rm an},\rr\pi_*J)=H^i(\mc X^{\rm an},J)$ for all $i$. Since $\rr^i\pi_*\mc O_{\mc X}=\mc O_{X}$ for $i=0$ and vanishes for $i>1$ \cite[Remark 1.4(3)]{ni08}, the same holds for $\pi : \mc X^{\rm an}\to X^{\rm an}$ \cite[GAGA Theorem]{toen99(1)}.

Taking the long exact hypercohomology sequences of (\ref{equ diag of exp seq}) yields
\begin{equation}\label{equ diag of long ex sq of exp seq}
\begin{split}
\xymatrix@C=0.2cm{
0 \ar[r] & 0 \ar[r] & 0 \ar[r] &\ast \ar[r] & \mbb H^2(X^{\rm an},\rr\pi_*\ulin{\mbb Z}/\ulin{\mbb Z}) \ar[r] & 0  \\
0 \ar[r] & H^1(\mc X^{\rm an},\ulin{\mbb Z})  \ar[u]\ar[r]  & H^1(\mc X^{\rm an},\mc O_{\mc X^{\rm an}}) \ar[u]\ar[r] & H^1(\mc X^{\rm an},\mc O_{\mc X^{\rm an}}^*) \ar[u]\ar[r]^{\quad\ccc_1} & H^2(\mc X^{\rm an},\ulin{\mbb Z})\ar[u]\ar[r] & H^2(\mc X^{\rm an},\mc O_{\mc X^{\rm an}}) \ar[u] \\
0 \ar[r] & H^1(X^{\rm an},\ulin{\mbb Z}) \ar[u]^{\cong}\ar[r] & H^1(X^{\rm an},\mc O_{X^{\rm an}})  \ar[u]^{\cong}\ar[r] & H^1(X^{\rm an},\mc O_{X^{\rm an}}^*)\ar[u]\ar[r]^{\quad\ccc_1} & H^2(X^{\rm an},\ulin{\mbb Z}) \ar[u]\ar[r] & H^2(X^{\rm an},\mc O_{X^{\rm an}}) \ar[u]^{\cong} \\
0\ar[r] &  0 \ar[u]\ar[r] &  0  \ar[u]\ar[r] & 0 \ar[u]\ar[r] & \mbb H^1(X^{\rm an},\rr\pi_*\ulin{\mbb Z}/\ulin{\mbb Z}) \ar[u]\ar[r] & 0 \ar[u]}
\end{split}
\end{equation}

Chasing the diagram then gives $\Pic^0(X^{\rm an})\cong\Pic^0(\mc X^{\rm an})$, hence $\Pic^0(X)^{\rm an}\cong\Pic^0(\mc X)^{\rm an}$ (Proposition \ref{cor eq of an gps}). Since the functor of analytification on proper schemes is fully faithful (see \cite[Corollaire 4.5 in expos\'{e} \uppercase\expandafter{\romannumeral 12}]{sga1}), we conclude $\Pic^0(X)\cong\Pic^0(\mc X)$.

\end{proof}

\begin{corollary}\label{cor ex sq of ns gp}
If $\mc X$ is a projective orbifold surface with codimension two stacky locus, the Néron–Severi groups satisfy:
\begin{enumerate}[$(1)$]
\item There is a short exact sequence
\begin{equation}\label{equ seq of ns gp} 
\xymatrix@=0.5cm{ 0 \ar[r] & \ns(X) \ar[r] & \ns(\mc X) \ar[r] & {\ns(\mc X)}/{\ns(X)} \ar[r] & 0 } 
\end{equation}
where $\ns(\mc X)/\ns(X)$ is finite.
\item The Picard numbers agree: $\rho(\mc X)=\rho(X)$.
\end{enumerate}
\end{corollary}

\begin{proof}
For $j>0$, the support of $\rr^j\pi_*J$ lies in the stacky locus. The stalk at a stacky point $p_k$ is $H^j(G_k,J_{q_k})$, where $J_{q_k}$ is the stalk of $J$ at a lift $q_k$ of $p_k$. Hence $\mbb H^i(X^{\rm an},\rr\pi_*J/J)=\bigoplus_k H^i(G_k,J_{q_k})$ for $i>0$.
In particular,
\begin{equation*}
\mbb H^i(X^{\rm an},\rr\pi_*\ulin{\mbb Z}/\ulin{\mbb Z})=
\begin{cases}
		0 & i=1,\\
\bigoplus_k H^i(G_k,\mbb Z) & i\ge 2,
\end{cases}
\end{equation*}
where we using $H^1(G_k,\mbb Z)=0$. By diagram \ref{equ diag of long ex sq of exp seq}, we obtain
\[\xymatrix@C=0.5cm{
	0 \ar[r] & \ns(X) \ar[r]  & \ns(\mc X) \ar[r] & \bigoplus_kH^2(G_k,\mbb Z) }.\]
Since each $H^2(G_k,\mbb Z)$ is finite, the claim follows.
\end{proof}

\begin{proposition}
Under the assumptions stated above, we obtain the following results.
\begin{enumerate}[$(1)$]
\item The torsion part $\ns(\mc X)_{\rm tor}$ of $\ns(\mc X)$ is isomorphic to the torsion part $H^2(\mc X,\ulin{\mbb Z})_{\rm tor}$ of $H^2(\mc X,\ulin{\mbb Z})$.
\item If we use $\ns_\mbb R(\mc X)$ to indicate $\ns(\mc X)\otimes_{\mbb Z}\mbb R$, then we have
\begin{equation}\label{equ ns group}
\ns_\mbb R(\mc X)=H^{1,1}(\mc X^{\rm an})\cap H^2(\mc X^{\rm an},\ulin{\mbb R}).
\end{equation}
\end{enumerate}
\end{proposition}

\begin{proof}
By the torsion freeness of $H^2(\mc X,\mc O_\mc X)$, we have $\ns(\mc X)_{\rm tor}=H^2(\mc X,\ulin{\mbb Z})_{\rm tor}$, and hence
\begin{equation*}
\xymatrix@C=0.3cm{
0 \ar[r] & {\ns(\mc X)}/{\ns(\mc X)_{\rm tor}} \ar[r] & {H^2(\mc X,\ulin{\mbb Z})}/{H^2(\mc X,\ulin{\mbb Z})}_{\rm tor} \ar[r] & H^2(\mc X,\mc O_\mc X)}.
\end{equation*}
Tensoring with $\mbb R$ gives the exact sequence
\begin{equation}\label{equ exact seq of real ns}
\xymatrix@C=0.3cm{
0 \ar[r] & \ns_\mbb R(\mc X) \ar[r] & H^2(\mc X,\ulin{\mbb R}) \ar[r] & H^2(\mc X,\mc O_\mc X)}
\end{equation}
The last arrow in (\ref{equ exact seq of real ns}) is isomorphic to the composition
	\begin{equation}\label{equ exact seq of real ns 1}
		\begin{split}
			H^2(\mc X,\ulin{\mbb R})\hookrightarrow H^2(\mc X,\ulin{\mbb C})=H^{2,0}(\mc X)\oplus H^{1,1}(\mc X)\oplus H^{0,2}(\mc X)\xrightarrow{\rm pr}H^{0,2}(\mc X)
		\end{split}
	\end{equation}
using Dolbeaut-Kodaira and Hodge decomposition (see (\romannumeral 4), (\romannumeral 5) in Section \ref{subsec hodge indx}). Thus $\ns_\mbb R(\mc X)\cong H^{1,1}(\mc X)\cap H^2(\mc X,\ulin{\mbb R})$.

\end{proof}

\subsection{Change of Polarization}
Following the case of smooth projective surfaces (see \cite[Section 4.C]{hl10}), we will prove some basic facts about the change of polarization for $\mc X$.
\begin{definition}\label{def num gp}
The group of numerical classes of $\mc X$ is  
\begin{equation*}
\num(\mc X):=\ns(\mc X)/\ns(\mc X)_{\rm tor}
\end{equation*}
the torsion free part of $\ns(\mc X)$.
\end{definition}

\begin{lemma}\label{lemm proty of num gp}
\begin{enumerate}[$(1)$]
	\item $\num(\mc X)$ is a free $\mbb Z$-module of rank $\rho(\mc X)$.
	\item $\num_\mbb R(\mc X)\cong\ns_{\mbb R}(\mc X)$, where $\num_{\mbb R}(\mc X):=\num(\mc X)\otimes_{\mbb Z}\mbb R$.
	\item $\num_{\mbb R}(\mc X)\cong\num_{\mbb R}(X)$.
	\item The positive cone is defined as
	\begin{equation*}
		K^{+}(\mc X):=\{x\in\num_\mbb R(\mc X)|x^2>0 \text{ and $x\cdot\pi^*H>0$ for some ample divisor $H$ on $X$} \}.
	\end{equation*}
	Then the cone $A(\mc X)$ spanned by the ample divisors is an open subcone of $K^+(\mc X)$.
\end{enumerate}
\end{lemma}
\begin{proof}
By Corollary \ref{cor ex sq of ns gp}, the statements (1), (2), (3) are immediate. First, note that $K^+(\mc X)\cong K^+(X)$ and $A(\mc X)\cong A(X)$. Recall that $A(X)$ is an open subcone of $K^+(X)$. Then the result is also true for $\mc X$.
\end{proof}

\begin{definition}
If we fix a generating sheaf $\mc E$ on $\mc X$, then a polarization of $\mc X$ is a ray $\mbb R_{>0}\cdot H$, where $H\in A(\mc X)$.
\end{definition}
By the Hodge index Theorem, $\num_\mbb R(\mc X)$ carries a Minkowski metric. For any $u\in\num_\mbb R(\mc X)$, let $|u|=|u^2|^{1/2}$. The set $\mbf H$ of rays in $K^+$ can be identified with the hyperbolic manifold $\{H\in K^+\mid |H|=1\}$. The hyperbolic metric $d_{\mbf H}$ is defined as
\begin{equation*}
d_{\mbf H}(H,H^\prime):=\arccosh\left(\frac{H\cdot H^\prime}{|H|\cdot|H^\prime|}\right)
\end{equation*}
for $H,H^\prime\in\mbf H$.

\begin{definition}
Let $r\geq 2$ be an integer and let $\Delta$ be a positive rational number. Then $\xi\in\num(\mc X)$ is said to be of type $(r,\Delta)$ if
\begin{equation*}
	-\frac{r^2}{4}\Delta\leq\xi^2<0.
\end{equation*}
The wall defined by $\xi$ is the hyperplane $W_\xi:=\{H\in\mbf H\mid\xi\cdot H=0\}$.
\end{definition}

The locally finiteness of walls is also holds for $\mc X$ and the proof does not differ from the usual case. 
\begin{lemma}\label{lemm ft of walls}
The set of walls of type $(r,\Delta)$ is locally finite in $\mbf H$, i.e. for every point $H\in\mbf H$, there exists an open neighborhood of $H$ intersecting only finitely many walls of type $(r,\Delta)$.
\qed
\end{lemma}


Fix an ample divisor $H$ on the coarse moduli space $X$. We can introduce the notion of $\mu_H$-stability for torsion free sheaves on $\mc X$. Suppose that $E$ is a torsion free sheaf on $\mc X$. The $\mu_H$-slope of $E$ is
\begin{equation}\label{def mu-slope}
\mu_H(E)=\frac{\ccc_1(E)\cdot\pi^*H}{\rk(E)}
\end{equation}
The discriminant of $E$ is defined as
\begin{equation}\label{def disciminant}
\Delta(E)=2\rk(E)\ccc_2(E)-(\rk(E)-1)\ccc_1(E).
\end{equation}

\begin{theorem}\label{thm change of plarization}
Suppose that $F$ is a $\mu_H$-semistable sheaf of rank $r$ and discriminant $\Delta$ on $\mc X$.
\begin{enumerate}[$(1)$]
	\item If $F^\prime$ is a rank $r^\prime$ coherent subsheaf of $F$ such that
	$0< r^\prime < r$ and $\mu_H(F^\prime)=\mu_H(F)$, then $\xi=r\cdot\ccc_1(F^\prime)-r^\prime\cdot\ccc_1(F)$ satisfies
	\begin{equation*}
		\xi\cdot\pi^*H=0\quad\text{and}\quad-\frac{r^2}{4}\Delta(F)\leq\xi^2\leq 0
	\end{equation*}
	where $\xi^2=0$ if and only if $\xi=0$.
	\item If $\ccc_1(F)\in\num(\mc X)$ is indivisible, then there exists an ample line bundle $H$ on $X$ such that the $\mu_H$-semistableness implies the $\mu_H$-stableness.
\end{enumerate}
\end{theorem}

\begin{proof}
Without loss of generality, we can assume that $F^\prime$ is saturated, i.e. $F^{\prime\prime}:=F/F^\prime$ is torsion free. Then
\begin{equation*}
\Delta(F)=\frac{r}{r^\prime}\Delta(F^\prime)+\frac{r}{r^{\prime\prime}}\Delta(F)-\frac{\xi^2}{r^\prime r^{\prime\prime}}
\end{equation*}
with $r^{\prime\prime}=\rk(F^{\prime\prime})$. By the stacky Bogomolov inequality (\cite[Proposition 4.2.4]{lie11},
\begin{equation}
	-\xi^2\leq r^\prime r^{\prime\prime}\Delta(F)\leq\frac{r^2}{4}\Delta(F).
\end{equation}
Since $\xi\cdot\pi^*H=0$, the Hodge index Theorem \ref{pro hodge index tm} gives $\xi^2\leq0$ with equality if and only if $\xi=0$. If $\ccc_1(F)\in\num(\mc X)$ is not divisible, then $\xi$ is not zero in $\num(\mc X)$. By Lemma \ref{lemm ft of walls}, we can choose $H$ avoiding walls, so no strictly $\mu_H$-semistable sheaves with rank $r$ and discriminant $\Delta$ exist.
\end{proof}

\section{The Poisson structure on the  moduli spaces of stable sheaves}\label{sec poisson st}

\subsection{The Atiyah class}\label{subsec at class}

We construct the Atiyah class on smooth projective Deligne-Mumford stacks (for the general case, see \cite{ku24}). To begin, recall the construction of Atiyah class for schemes (see \cite{il71} for details). Let $S$ be a separable scheme of finite type over $\mbb C$, and let $E$ be a coherent sheaf on $S$. Consider the exact sequence
\begin{equation}\label{equ at seq 1}
\xymatrix@C=0.5cm{
  0 \ar[r] &  (I_S/I^2_S)\otimes\pr_2^*E \ar[r]  & (\mc O_{S\times S}/I^2_S)\otimes\pr^*_2E \ar[r] & (\mc O_{S\times S}/I_S)\otimes\pr_2^*E \ar[r]  & 0 }
\end{equation}
where $I_S$ is the ideal sheaf of the diagonal in $S\times S$. The pushforward of (\ref{equ at seq 1}) along $\pr_1:S\times S\to S$ yields the exact sequence
\begin{equation}\label{equ prin parts}
	0 \to E\otimes \Omega_S \to P^1_S(E) \to E \to 0,
\end{equation}
where $P^1_S(E)=\pr_{1*}\big((\mc O_{S\times S}/I_S^2)\otimes\pr_2^*E\big)$. Its extension class $\at(E)\in \Ext^1(E,E\otimes\Omega_S)$ is the Atiyah class of $E$.

\begin{lemma}\label{lemm ft of at class}
Suppose that $f : T\rightarrow S$ is an \'{e}tale morphism between separable schemes. Then there is a functorial isomorphism $f^*P^1_S(E)\longrightarrow P^1_T(f^*E)$. For another \'{e}tale morphism $g : W\rightarrow T$ between separable schemes, then
\begin{equation*}
(f\circ g)^*P^1_S(E)\longrightarrow P^1_W((f\circ g)^*E))
\end{equation*}
is isomorphic to the composite
\begin{equation*}
g^*f^*P^1_S(E)\rightarrow g^*P^1_T(f^*E)\rightarrow P^1_W(g^*f^*E)
\end{equation*}
up to canonical isomorphisms.
\end{lemma}

\begin{proof}
Pulling (\ref{equ at seq 1}) back to $T\times T$, we get
\begin{equation}\label{equ at seq 2}
\xymatrix@C=0.5cm{
  0 \ar[r] & J_T/J^2_T\otimes\pr_2^*f^*E \ar[r] & \mc O_{T\times T}/J^2_T\otimes\pr^*_2f^*E \ar[r]  & \mc O_{T\times T}/J_T\otimes\pr_2^*f^*E \ar[r]  & 0 }
\end{equation}
where $J_T=(f\times f)^*I_S$ is the ideal sheaf of $T\times_ST$ in $T\times T$. Since the diagonal of $T\times T$ is an open and closed component of $T\times_ST$,  restricting (\ref{equ at seq 2}) to the diagonal of $T\times T$, we get
\begin{equation}
\xymatrix@C=0.5cm{
  0 \ar[r] & I_T/I^2_T\otimes\pr_2^*f^*E \ar[r] & \mc O_{T\times T}/I^2_T\otimes\pr^*_2f^*E \ar[r]  & \mc O_{T\times T}/I_T\otimes\pr_2^*f^*E \ar[r]  & 0 }
\end{equation}
where $I_T$ is the ideal sheaf of diagonal in $T\times T$. Then we have a functorial isomorphism of short exact sequences
\begin{equation}
\begin{split}
\xymatrix@C=0.5cm{
0 \ar[r] & f^*(E\otimes\Omega_S) \ar[r]\ar[d] &  f^*P^1_S(E) \ar[r]\ar[d] &   f^*E \ar[r]\ar@{=}[d] & 0 \\
0 \ar[r] & f^*E\otimes\Omega_T \ar[r] &  P^1_T(f^*E) \ar[r]  & f^*E \ar[r] & 0 .}
\end{split}
\end{equation}
By the above construction of the functorial isomorphism, the second statement of the lemma is immediate.
\end{proof}

Suppose that $\mc Y$ is a smooth projective Deligne-Mumford stack. Then there exists an \'{e}tale cover $U\rightarrow\mc Y$ with $U$ affine. Consider the cartesian diagram
\begin{equation*}
  \xymatrix{
  U\times_\mc YU \ar[r]^{\quad\pr_1} \ar[d]_{\pr_2} & U \ar[d]\\
    U \ar[r] & \mc Y                         }
\end{equation*}
For a coherent sheaf $F$ on $\mc Y$, let $F[1]$ denote its pullback to $U$. There is an isomorphism $\sigma : {\rm pr}_1^*E[1]\longrightarrow{\rm pr}_2^*E[1]$ on $U\times_\mc YU$ satisfying the cocycle condition ${\pr}_{23}^*\sigma\circ{\pr}_{12}^*\sigma={\pr}_{13}^*\sigma$ on $U\times_\mc YU\times_\mc YU$, where $\pr_{12}$, $\pr_{23}$ and $\pr_{13}$ are the natural projections from $U\times_\mc YU\times_\mc YU$ to $U\times_\mc YU$.  Applying Lemma \ref{lemm ft of at class} to $\pr_1$ and $\pr_2$, we get two canonical isomorphism
\begin{equation}\label{equ pr at class}
\pr_1^*P^1_U(F[1])\rightarrow P^1_{U\times_\mc YU}(\pr_1^*F[1])\quad\text{and}\quad\pr_2^*P^1_U(F[1])\rightarrow P^1_{U\times_\mc YU}(\pr_2^*F[1]).
\end{equation}
In addition, we also have an isomorphism
\begin{equation}\label{equ sigma at}
 P^1_{U\times_\mc YU}(\sigma) : P^1_{U\times_\mc YU}(\pr_1^*F[1])\rightarrow P^1_{U\times_\mc YU}(\pr_2^*F[1]).
\end{equation}
Composing (\ref{equ sigma at}) with (\ref{equ pr at class}), we get an isomorphism $\wt\sigma : \pr_1^*P^1_U(F[1])\rightarrow\pr_2^*P^1_U(F[1])$. By the second statement of Lemma \ref{lemm ft of at class}, $\wt\sigma$ satisfies the cocycle condition $\pr_{23}^*\wt\sigma\circ\pr_{12}^*\wt\sigma=\pr_{13}^*\wt\sigma$ up to canonical isomorphisms. This gives the exact sequence
\begin{equation}\label{equ at class for DM}
  \xymatrix@C=0.5cm{
    0 \ar[r] & F\otimes\Omega_\mc Y \ar[r] & P^1_\mc Y(F) \ar[r] & F \ar[r] & 0 }.
\end{equation}

\begin{definition}\label{def at class for DM}
The Atiyah class $\at(F)\in\ext^1(F,F\otimes\Omega_\mc Y)$ of $F$ is the extension class defined by (\ref{equ at class for DM}).
\end{definition}
Alternatively, consider the trivial square-zero extension of $\mc O_\mc Y$ by $\Omega_\mc Y$
\begin{equation}\label{equ trivial ext of o}
\xymatrix@C=0.5cm{
  0 \ar[r] & \Omega_\mc Y \ar[r] & \mc O_\mc Y\oplus\Omega_\mc Y \ar[r] & \mc O_\mc Y \ar[r] & 0 }.
\end{equation}
Let $d_\mc Y : \mc O_\mc Y\rightarrow\Omega_\mc Y$ be the universal derivation of $\mc Y$. The morphism $(\id, d_\mc Y) : \mc O_\mc Y\rightarrow\mc O_\mc Y\oplus\Omega_\mc Y$ is a section of the projection in (\ref{equ trivial ext of o}), giving $\mc O_\mc Y\oplus\Omega_\mc Y$ a right $\mc O_\mc Y$-module structure. Tensoring with $F$ yields an exact sequence
\begin{equation}\label{equ at class for DM 1}
\xymatrix@C=0.5cm{
  0 \ar[r] & F\otimes\Omega_\mc Y \ar[r] & (\mc O_\mc Y\oplus\Omega_\mc Y)\otimes F \ar[r] & F \ar[r] & 0 },
\end{equation}
which is isomorphic to (\ref{equ at class for DM}.

\subsection{The Smoothness of the moduli space of sheaves}
Henceforth let $\mc X$ be an orbifold surface with stacky locus of codimension two. Fix a numerical K-class $\upsilon\in K^{\rm num}(\mc X)$ with rank one or primitive first Chern class in $\num(\mc X)$. By Theorem \ref{thm change of plarization}, there exists an ample divisor $H$ on $X$ such that no strictly $\mu_H$-semistable sheaves of class $\upsilon$ exist. Fix a generating sheaf $\mc E$ on $\mc X$ (see \cite[Definition 2.6]{ni08}), and let $\mc M_\upsilon$ denote the moduli space of Gieseker semistable torsion free sheaves of class $\upsilon$ (see ibid.).

\begin{lemma}\label{lemm no str ss}
$\mc M_\upsilon$ does not contain strictly semistable objects.
\end{lemma}
\begin{proof}
Let $P_E(m)=\chi(E\otimes\mc E^\vee\otimes\pi^*\mc O_X(nH))=a_2n^2/2+a_1n+a_0$ be the modified Hilbert polynomial of $E$.
By To\"{e}n-Riemann-Roch formula (see \cite[Theorem A.0.6]{ts10}), we have
\begin{equation}
\frac{a_1}{a_2}=\frac{\mu_H(E)-\mu_H(\mc E)-\mu_H(\Omega_\mc X)}{H^2}.
\end{equation}
If $E$ is Gieseker semistable, then $E$ is $\mu_H$-stable. Hence $E$ is Gieseker stable.
\end{proof}

The determinant map $\dett : \mc M_v\rightarrow\Pic(\mc X)$, $E\mapsto\dett(E)$; let $\mc M_\upsilon^L$ be its fiber  over $L\in\Pic(\mc X)$.

\begin{proposition}\label{pro sm of mod}
If $K_\mc X\cdot H<0$ or $K_\mc X\cong\mc O_\mc X$, then $\mc M_\upsilon$ and $\mc M_\upsilon^L$ are smooth projective schemes.
\end{proposition}

\begin{proof}
For a Gieseker stable sheaf $E$, both $E$ and $E\otimes K_\mc X$ are $\mu_H$-stable, so $\dimm_\mbb C\homm(E,E\otimes K_\mc X)\leq1$. By Serre duality, $\ext^2(E,E)_0\cong\homm(E,E\otimes K_\mc X)_0=0$. Then the standard argument in \cite[Theorm 4.5.4]{hl10} shows that $\mc M_\upsilon$ and $\mc M_\upsilon^L$ are smooth.
\end{proof}


\subsection{The Kodaira-Spencer map}
Pick a (twisted) universal sheaf $\mbf E$ on $\mc M_\upsilon\times\mc X$. The decomposition $\Omega_{\mc M_\upsilon\times\mc X}=\pr_1^*\Omega_{\mc M_\upsilon}\oplus\pr_2^*\Omega_\mc X$ induces a decomposition of the Atiyah class $\at(\mbf E)=\at(\mbf E)_1+\at(\mbf E)_2$, where $\at(\mbf E)_1\in\ext^1(\mbf E,\mbf E\otimes\pr_1^*\Omega_{\mc M_\upsilon})$ and $\at(\mbf E)_2\in\ext^1(\mbf E,\mbf E\otimes\pr_2^*\Omega_\mc X)$. Via the Grothendieck spectral sequence
\begin{equation*}
E_2^{i,j}=H^i(\mc M_\upsilon,\Ext^j_{\pr_1}(\mbf E,\mbf E\otimes\pr_1^*\Omega_{\mc M_\upsilon}))\Longrightarrow\ext^{i+j}(\mbf E,\mbf E\otimes\pr_1^*\Omega_{\mc M_\upsilon}),
\end{equation*}
we get the global to local map $\ext^1(\mbf E,\mbf E\otimes\pr^*_1\Omega_{\mc M_\upsilon})\rightarrow H^0(\mc M_\upsilon,\Ext^1_{\pr_1}(\mbf E,\mbf E\otimes\pr_1^*\Omega_{\mc M_\upsilon}))$, and we denote the image of $\at(\mbf E)_1$ by the same symbol. The Kodaira-Spencer map is then
\begin{equation*}
{\rm KS} : T\mc M_\upsilon\xrightarrow{\at(\mbf E)_1}T{\mc M_\upsilon}\otimes\Ext^1_{\pr_1}(\mbf E,\mbf E\otimes\pr^*_1\Omega_{\mc M_\upsilon})\rightarrow\Ext^1_{\pr_1}(\mbf E,\mbf E).
\end{equation*}

\begin{proposition}\label{pro ks map}
$\rm KS $ is an isomorphism.
\end{proposition}

\begin{proof}
For any closed point $p\in\mc M_\upsilon$ and $\mbf v\in T_p\mc M_\upsilon$, we have a morphism of $\mc O_{\mc M_\upsilon}$-modules $\mbf v : \Omega_{\mc M_\upsilon}\rightarrow\mbb C$, and its pullback $\pr_1^*\mbf v : \pr_1^*\Omega_{\mc M_\upsilon}\rightarrow\mc O_\mc X$. Consider the diagram
\begin{equation}\label{equ equ 3 for ks}
\begin{split}
\xymatrix@C=0.3cm{
\ext^1(\mbf E,\mbf E\otimes\pr_1^*\Omega_{\mc M_\upsilon}) \ar[r]\ar[d] & \ext^1(\mbf E,\mbf E\otimes\mc O_\mc X) \ar[d] \\
H^0(\mc M_\upsilon,\Ext^1_{\pr_1}(\mbf E,\mbf E\otimes\pr_1^*\Omega_{\mc M_\upsilon})) \ar[r] & H^0(\mc M_\upsilon,\Ext^1_{\pr_1}(\mbf E,\mbf E\otimes\mc O_\mc X))}
\end{split}
\end{equation}
where the vertical arrows are the global to local maps. Let $i_p : p\rightarrow\mc M_\upsilon$ and $\mc X\rightarrow\mc M_\upsilon\times\mc X$ be the inclusions. Then $\mbf E\otimes\mc O_\mc X=i_*\mbf E_p$, where $\mbf E_p$ is the stable sheaf corresponding to $p$. Hence $\Ext^1_{\pr_1}(\mbf E,\mbf E\otimes\mc O_\mc X)=i_{p*}\ext^1(\mbf E_p,\mbf E_p)$ and $\ext^1(\mbf E,\mbf E\otimes\mc O_\mc O)=\ext^1(\mbf E_p,\mbf E_p)$. The second vertical arrow in diagram (\ref{equ equ 3 for ks}) is an isomorphism. The image of $\mbf v$ under the Kodaira-Spencer map is then the image of $\at(\mbf E)_1$ under the top arrow, which corresponds to the extension
\begin{equation}\label{equ equ 5 for ks}
\xymatrix{
0\ar[r] &  \mbf E_p \ar[r] & \wt{\mbf E}_p \ar[r] & \mbf E_p \ar[r] & 0}.
\end{equation}
defined by $\mbf v\in\Ext^1(E_p,E_p)$. This follows from pulling back the trivial square-zero extensions of $\mc O_{\mc M_\upsilon\times\mc X}$ along $\mbf v$ and tensoring with $E$.
\end{proof}

\begin{corollary}\label{cor cotang of mod sp}
The cotangent bundle of $\mc M_\upsilon$ is isomorphic to $\Ext_{\pr_1}^1(\mbf E,\mbf E\otimes\pr_2^*K_\mc X)$.
\end{corollary}
\begin{proof}
By Grothendieck duality for Deligne-Mumford stacks (see \cite[Corollary 2.10]{ni09}), we have that $
R\pr_{1*}R\Hom(\mbf E,\mbf E\otimes\pr_2^*K_\mc X[2])\cong R\Hom(R\pr_{1*}R\Hom(\mbf E,\mbf E),\mc O_{\mc M_\upsilon})$. Taking cohomology gives $\Ext_{\pr_1}^1(\mbf E,\mbf E\otimes\pr_2^*K_\mc X)\cong\Hom(\Ext_{\pr_1}^1(\mbf E,\mbf E),\mc O_{\mc O_{\mc M_\upsilon}})$.

\end{proof}

\subsection{The Poisson structure on $\mc M_\upsilon$}\label{subsec poisson st}
\begin{definition}
	Let $\mc X$ be a smooth Deligne-Mumford stack. A Poisson structure on $\mc X$ is a bilinear operation $\{\cdot,\cdot\}$ on $\mc O_\mc X$ such that for any object $U\rightarrow\mc X$ of $\mc X_{\et}$ and any $f,g,h\in\mc O_\mc X(U\rightarrow\mc X)$, the three axioms hold: (1) Skew-symmetry: $\{f,g\}=-\{g,f\}$; (2) Leibniz rule: $\{f,gh\}=\{f,g\}h+g\{f,h\}$; (3) Jacobi identity: $\{f,\{g,h\}\}+\{g,\{h,f\}\}+\{h,\{f,g\}\}=0$.
\end{definition}
$\{\cdot,\cdot\}$ determines an antisymmetric contravariant $2$-tensor $\theta\in H^0(\mc X,\wedge^2T_\mc X)$ by
\begin{equation}\label{equ theta tensor}
	\langle\theta,df\wedge dg\rangle:=\{f,g\}.
\end{equation}
Conversely, any $\theta\in H^0(\mc X,\wedge^2T_\mc X)$ defines a Poisson structure by (\ref{equ theta tensor}) iff $\wt d\theta=0$ (see \cite[Proposition 1.1]{bo95}), where $\wt d : H^0(\mc X,\wedge^2T_\mc X)\rightarrow H^0(\mc X,\wedge^3T_\mc X)$ is defined as follows: for any $1$-forms $\alpha_1,\alpha_2,\alpha_3$,
\begin{equation}
	\begin{split}
		\wt d\theta(\alpha_1,\alpha_2,\alpha_3):=B_\theta(\alpha_1)\theta(\alpha_2,\alpha_3)+B_\theta(\alpha_2)\theta(\alpha_3,\alpha_1)+B_\theta(\alpha_3)\theta(\alpha_1,\alpha_2)\\
		-\langle[B_\theta(\alpha_1),B_\theta(\alpha_2)],\alpha_3\rangle-\langle[B_\theta(\alpha_2),B_\theta(\alpha_3)],\alpha_1\rangle-\langle[B_\theta(\alpha_3),B_\theta(\alpha_1)],
		\alpha_2\rangle,
	\end{split}
\end{equation}
where $B_\theta : \Omega_\mc X\rightarrow T_\mc X$ is defined by $\langle\theta,\alpha_1\wedge\alpha_2\rangle=\langle B_\theta(\alpha_1),\alpha_2\rangle$ and $[\cdot,\cdot]$ is the usual commutator of the vector fields.

\begin{lemma}\label{cor poisson st 2 ts}
	If $\dimm \mc X=2$, then Poisson structures on $\mc X$ are global sections of $K_\mc X^{-1}$.
	\qed
\end{lemma}
\begin{definition}
	A symplectic structure on $\mc X$ is a closed nondegenerate $2$-form on $\mc X$.
\end{definition}

For a smooth Deligne-Mumford curve $\mc C$, the total space of $K_\mc C$ is a symplectic, with natural compactification $\mds P(K_\mc C\oplus\mc O_\mc C)$.

\begin{proposition}\label{pro pois st on cp ct bd}
	$\mds P(K_\mc C\oplus\mc O_\mc C)$ carries a Poisson structure extending the symplectic form on $K_\mc C$.
\end{proposition}

\begin{proof}
	From the Euler sequence one computes $\wedge^2T_{\mds P(K_\mc C\oplus\mc O_\mc C)}\cong \mc O(2)$.
	The section $(0,1)\in H^0(\mc C,T_\mc C\oplus\mc O_\mc C)$ induces $s\in H^0(\mc O(1))$, and $s^{\otimes2}$ defines the desired Poisson structure.
\end{proof}

\begin{proposition}\label{pro poi st on moduli}
$\mc M_\upsilon$ has a natural Poisson structure, i.e. there exists a bilinear map
\begin{equation*}
\begin{split}
\wt\theta : \Ext^1_{\pr_1}(\mbf E,\mbf E\otimes\pr_2^*K_\mc X)\otimes\Ext^1_{\pr_1}(\mbf E,\mbf E\otimes\pr_2^*K_\mc X)
\longrightarrow\mc O_{\mc M_\upsilon}
\end{split}
\end{equation*}
which defines a Poisson structure on $\mc M_\upsilon$. Moreover, the restriction of $\wt\theta$ to $\mc M_\upsilon^L$ is also a Poisson structure.
\end{proposition}

\begin{proof}
Let $\theta\in H^0(\mc X,\wedge^2T\mc X)$ be the antisymmetric two tensor defined by the Poisson structure on $\mc X$. Consider the commutative diagram
\begin{equation}
\begin{split}
\xymatrix@=0.5cm{
R\pr_{1*}R\Hom(\mbf E,\mbf E\otimes\pr_2^*K_\mc X)\otimes^LR\pr_{1*}R\Hom(\mbf E,\mbf E\otimes\pr_2^*K_\mc X) \ar[d]^\theta \ar[r]& \mc O_{\mc M_\upsilon}\otimes R\Gamma(K_\mc X^2) \ar[d]^\theta \\
R\pr_{1*}R\Hom(\mbf E,\mbf E\otimes\pr_2^*K_\mc X)\otimes^LR\pr_{1*}R\Hom(\mbf E,\mbf E)\ar[d]^{\cong}\ar[r] & \mc O_{\mc M_\upsilon}\otimes R\Gamma(K_\mc X) \ar@{=}[d] \\
R\pr_{1*}R\Hom(R\Hom(\mbf E,\mbf E),\pr_2^*K_\mc X)\otimes^LR\pr_{1*}R\Hom(\mbf E,\mbf E) \ar[r] & \mc O_{\mc M_\upsilon}\otimes R\Gamma(K_\mc X)
}
\end{split}
\end{equation}
where the horizontal arrows come from cup product and trace map, and the last from the evaluation map. Passing to cohomology yields
\begin{equation}\label{equ diag poisson st}
\begin{split}
\xymatrix@=0.5cm{
\Ext^1_{\pr_1}(\mbf E,\mbf E\otimes\pr_2^*K_\mc X)\otimes\Ext^1_{\pr_1}(\mbf E,\mbf E\otimes\pr_2^*K_\mc X) \ar[d]^\theta \ar[r] &  \mc O_{\mc M_\upsilon} \otimes H^2(\mc X,K^2_\mc X) \ar[d]^\theta \\
\Ext^1_{\pr_1}(\mbf E,\mbf E\otimes\pr_2^*K_\mc X)\otimes \Ext^1_{\pr_1}(\mbf E,\mbf E) \ar[d]^{\cong} \ar[r] & \mc O_{\mc M_\upsilon}\otimes
H^2(\mc X,K_\mc X) \ar@{=}[d] \\
\Ext^1_{\pr_1}(R\Hom(\mbf E,\mbf E),\pr_2^*K_\mc X)\otimes \Ext^1_{\pr_1}(\mbf E,\mbf E) \ar[r] & \mc O_{\mc M_\upsilon}\otimes H^2(\mc X,K_\mc X)
}
\end{split}
\end{equation}
where the last arrow is a perfect pairing (see \cite[Corollary 2.10]{ni09}). Thus the induced pairing 
\begin{equation}\label{equ poi st on mod}
\begin{split}
\Ext^1_{\pr_1}(\mbf E,\mbf E\otimes\pr_2^*K_\mc X)\otimes\Ext^1_{\pr_1}(\mbf E,\mbf E\otimes\pr_2^*K_\mc X)\\
\rightarrow \mc O_{\mc M_\upsilon}\otimes H^2(\mc X,K^2_\mc X)\xrightarrow{\theta}\mc O_{\mc M_\upsilon}\otimes H^2(\mc X,K_\mc X)\cong\mc O_{\mc M_\upsilon}.
\end{split}
\end{equation}
defines a Poisson structure $\wt\theta$ on $\mc M_\upsilon$. Equivalently, it corresponds to the map
\[
\wt B:\ T^\vee\mc M_\upsilon\;\cong\;\Ext^1_{\pr_1}(\mbf E,\mbf E\otimes\pr_2^*K_\mc X)
\;\xrightarrow{\ \theta\ }\;\Ext^1_{\pr_1}(\mbf E,\mbf E)\;\cong\;T\mc M_\upsilon,
\]
see \cite{bo95}. Finally, the trace map yields a decomposition
\[
\Ext^1_{\pr_1}(\mbf E,\mbf E\otimes\pr_2^*K_\mc X)
=\Ext^1_{\pr_1}(\mbf E,\mbf E\otimes\pr_2^*K_\mc X)_0
\oplus \bigl(\mc O_{\mc M_\upsilon}\otimes H^1(\mc X,K_\mc X)\bigr),
\]
orthogonal with respect to the pairing above. Since the trace-free part is identified with $T^\vee\mc M_\upsilon^L$, the restriction of $\wt\theta$ to $\mc M_\upsilon^L$ is again Poisson.

\end{proof}

\section{The Connectedness of the moduli spaces of sheaves}
The aim of this section is to prove that $\mc M_\upsilon$ is connected. We first express the diagonal class of $\mc M_\upsilon$ in terms of Chern classes, then show that $\mc M_\upsilon$ is generated by the K\"{u}nneth factors of the orbifold Chern character of a universal sheaf, and finally deduce the connectedness. Our main references are \cite{mk87, bea92, yo01, klc06, ma07}.

\begin{theorem}\label{thm genetators of moduli}
Suppose that $K_\mc X\cdot H<0$ or $K_\mc X\cong\mc O_\mc X$. Let $\mbf E$ be a (twisted) universal sheaf on $\mc M_\upsilon\times\mc X$.
\begin{enumerate}[$(1)$]
\item The Poincar\'{e} dual of the diagonal $\delta$ of $\mc M_\upsilon\times\mc M_\upsilon$ is
\begin{equation*}
\ccc_m(-[R\pr_{12*}(\pr_{23}^*\mbf E^\vee\otimes^L\pr_{13}^*\mbf E)])
\end{equation*}
where $m=\dimm\mc M_\upsilon$.
\item If $\mbf E$ is a universal sheaf, then the K\"{u}nneth factors of the orbifold Chern character of $\mbf E$ generate the cohomology ring of $H^*(\mc M_\upsilon,\mbb C)$.
\end{enumerate}
\end{theorem}

\begin{proof}[Proof of the case $K_\mc X\cdot H<0$]
\tbf{Case I.} $\mc M_\upsilon$ is fine and $\mbf E$ is a universal sheaf. Consider the diagram
\begin{equation}\label{equ connect diag 1}
\begin{split}
\xymatrix{
& \mc M_\upsilon\times\mc M_\upsilon\times\mc X  \ar[dr]^{\pr_{23}}\ar[dl]_{\pr_{13}}\ar[d]^{\pr_{12}} \\
 \mc M_\upsilon\times\mc X   &   \mc M_\upsilon\times\mc M_\upsilon   & \mc M_\upsilon\times\mc X .}
\end{split}
\end{equation}
and the complex $R\pr_{12*}R\Hom(\pr_{23}^*\mbf E,\pr_{13}^*\mbf E)$. Since $\ext^2(E,F)=0$ for any stable sheaves $E,F$, this complex is represented by a two-term complex $[W^0\xrightarrow{u} W^1]$, with $\rk(W^1)=\rk(W^0)+m-1$ where $m=\dim \mc M_\upsilon$. For $z=(E_1,E_2)\in \mc M_\upsilon\times\mc M_\upsilon$, 
the exact sequence
\begin{equation*}
\xymatrix@=0.5cm{
0 \ar[r] & \homm(E_2,E_1) \ar[r] &  W^0\otimes k(z)   \ar[r]^{u(z)} &   W^1\otimes k(z) \ar[r] & \ext^1(E_2,E_1) \ar[r] & 0}
\end{equation*}
shows that the $(r_0-1)$-st degeneracy locus $D_{r_0-1}$ of $u$ coincides with the diagonal $\delta$. By Thom–Porteous, $[\delta]=c_m(W^1-W^0)\cap[\mc M_\upsilon\times\mc M_\upsilon]$, i.e. its Poincaré dual is
$c_m\!\left(-[R\pr_{12*}(\pr_{23}^*\mbf E^\vee\otimes^L\pr_{13}^*\mbf E)]\right)$. Using Proposition \ref{pro comp ext}, the K-theoretic pushforward satisfies
\begin{equation*}
\pr_{12!}([\pr_{23}^*\mbf E]^\vee\cdot[\pr_{13}^*\mbf E])=[R\pr_{12*}(\pr_{23}^*\mbf E^\vee\otimes^L\pr_{13}^*\mbf E)].
\end{equation*}
By T\"{o}en-Riemann-Roch formula,
\begin{equation*}
\ch([R\pr_{12*}(\pr_{23}^*\mbf E^\vee\otimes^L\pr_{13}^*\mbf E)])=I\pr_{12*}\left(\wt\ch([\pr_{23}^*\mbf E]^\vee)\cdot\wt\ch([\pr_{13}^*\mbf E])\cdot\wt\td( \pr_3^*T\mc X)\right)
\end{equation*}

By K\"{u}nneth decomposition, $\wt\ch(\mbf E)=\sum_i \pr_1^*\alpha_i\cup I\pr_2^*\beta_i$ for some $\alpha_i\in H^*(\mc M_\upsilon,\mbb C)$ and $\beta_i\in H^*(I\mc X,\mbb C)$. Then
\begin{equation*}
\ccc_m(-[R\pr_{12*}(\pr_{23}^*\mbf E^\vee\otimes^L\pr_{13}^*\mbf E)])=\sum_i\pr_1^*\alpha_i\cup\pr_2^*\gamma_i
\end{equation*}
with $\gamma_i\in H^*(\mc M_\upsilon,\mbb C)$. Thus $\{\alpha_i\}$ generates $H^*(\mc M_\upsilon,\mbb C)$.

\tbf{Case II.} $\mc M_\upsilon$ is not fine and $\mbf E$ a twisted universal sheaf. We will use some facts about twisted sheaves (see \cite[Subsection 2.1]{ca20} for details). By the proof of Proposition \ref{pro comp ext}, we obtain a twisted locally free resolution $L^\bullet$ of $\pr_{23}^*\mbf E$ such that $R\pr_{12*}R\Hom(\pr_{23}^*\mbf E,\pr_{13}^*\mbf E)=\pr_{12*}\Hom(L^\bullet,\pr_{13}^*\mbf E)$, a complex of locally free sheaves satisfying the same universal property. Hence, as in Case I, $R\pr_{12*}(\pr_{23}^*\mbf E^\vee\otimes^L\pr_{13}^*\mbf E)$ is represented by a two-term complex, and the Poincar\'{e} dual of $\delta$ is again $\ccc_m(-[R\pr_{12*}(\pr_{23}^*\mbf E^\vee\otimes^L\pr_{13}^*\mbf E)])$.

\end{proof}

\begin{proof}[Proof of the case: $K_\mc X\cong\mc O_\mc X$]

Using Lemma \ref{lemm main thm}, we obtain a complex of locally free sheaves
\[
\xymatrix@C=0.5cm{
	V_{-1} \ar[r] & V_0 \ar[r] & V_1}
\]
on $\mc M_\upsilon\times\mc M_\upsilon$ with cohomologies $\hh_{-1}=0$, $\hh_0=\Ext^1_{\pr_{12}}(\pr_{23}^*\mbf E,\pr_{13}^*\mbf E)$, $\hh_1=\Ext^2_{\pr_{12}}(\pr_{23}^*\mbf E,\pr_{13}^*\mbf E)$, and $-r_{-1}+r_0-r_1=m-2$, where $r_{-1}, r_0, r_1$ are the ranks of $V_{-1},V_0,V_1$, respectively. Both $\Ext^2_{\pr_{12}}(\pr_{23}^*\mbf E,\pr_{13}^*\mbf E)$ and $\Ext^2_{\pr_{12}}(\pr_{13}^*\mbf E,\pr_{23}^*\mbf E)$ are line bundles on the diagonal $\delta$. By \cite[Lemma~4]{ma02}, the Poincaré dual of $\delta$ is $c_m\!\left(-[R\pr_{12*}(\pr_{23}^*\mbf E^\vee\otimes^L\pr_{13}^*\mbf E)]\right)$. Assuming $\mc M_\upsilon$ is fine with universal sheaf $\mbf E$, we have
\[
\pr_{12!}([\pr_{23}^*\mbf E]^\vee\cdot[\pr_{13}^*\mbf E])
=[R\pr_{12*}(\pr_{23}^*\mbf E^\vee\otimes^L\pr_{13}^*\mbf E)].
\]
Applying the Toen–Riemann–Roch formula shows that the Künneth factors of the orbifold Chern character of $\mbf E$ generate $H^*(\mc M_\upsilon,\C)$.
\end{proof}

\begin{lemma}\label{lemm main thm}
There exists a locally free sheaf $\mbf A_1$ and a short exact sequence of sheaves
\begin{equation}\label{equ main thm 3}
\xymatrix@C=0.5cm{
0 \ar[r] & \mbf A_0 \ar[r] &  \mbf A_1 \ar[r] & \mbf E \ar[r] & 0}
\end{equation}
satisfying
\begin{enumerate}[$(1)$]
\item $\Ext^0_{\pr_{12}}(\pr^*_{23}\mbf E,\pr^*_{13}\mbf E)$=$\Ext^2_{\pr_{12}}(\pr_{23}^*\mbf A_0,\pr_{13}^*\mbf E)$=$\Ext^1_{\pr_{12}}(\pr^*_{23}\mbf A_1,\pr^*_{13}\mbf E)$=$\Ext^2_{\pr_{12}}(\pr^*_{23}\mbf A_1,\pr^*_{13}\mbf E)=0$;
\item  $\Ext^1_{\pr_{12}}(\pr_{23}^*\mbf A_0,\pr_{13}^*\mbf E)\cong\Ext^2_{\pr_{12}}(\pr_{23}^*\mbf E,\pr_{13}^*\mbf E)$;
\item $\Ext^0_{\pr_{12}}(\pr^*_{23}\mbf A_1,\pr^*_{13}\mbf E)$ is a locally free sheaf;
\item $\xymatrix@=0.5cm{0 \ar[r]  & \Ext^0_{\pr_{12}}(\pr_{23}^*\mbf A_1,\pr_{13}^*\mbf E) \ar[r]  & \Ext^0_{\pr_{12}}(\pr_{23}^*\mbf A_0,\pr_{13}^*\mbf E) \ar[r] & \Ext^1_{\pr_{12}}(\pr_{23}^*\mbf E,\pr^*_{13}\mbf E) \ar[r] &0  }$.
\end{enumerate}
\end{lemma}

\begin{proof}
Fix a generating sheaf $\mc E$ on $\mc X$, and consider
\begin{equation*}
\xymatrix@=0.5cm{
&  \mc M_v\times\mc X \ar[dl]_{\pr_1} \ar[dr]^{\pr_2}  \\
\mc M_v &&  \mc X.}
\end{equation*}
By the boundedness of semistable sheaves on Deligne-Mumford stacks (\cite[Theorem 4.27 (2)]{ni08}), there is an integer $n$ such that the Mumford-Castelnuovo regularity ${\rm reg}(E)\leq n$ for any semistable sheaf $E$. The natural surjection $\pr_1^*\pr_{1*}(\mbf E\otimes\pr_2^*\mc E^\vee(n))\otimes\pr_2^*\mc E(-n)\rightarrow\mbf E$ defines $\mbf A_1$, giving exact sequence (\ref{equ main thm 3}). Since $\Ext^3_{\pr_{12}}(\pr_{23}^*\mbf E,\pr_{13}^*\mbf E)=0$ (see Theorem \ref{thm base ch for ext} in Appendix \ref{app rela ext}), the long exact sequence associated to (\ref{equ main thm 3}) is
\begin{equation}\label{equ main thm 4}
\begin{split}
\xymatrix@=0.5cm{
  0 \ar[r] & \Ext^0_{\pr_{12}}(\pr_{23}^*\mbf E,\pr^*_{13}\mbf E) \ar[r]  & \Ext^0_{\pr_{12}}(\pr_{23}^*\mbf A_1,\pr_{13}^*\mbf E) \ar[r]  & \Ext^0_{\pr_{12}}(\pr_{23}^*\mbf A_0,\pr_{13}^*\mbf E) }\\
\xymatrix@=0.5cm{
  \ar[r] & \Ext^1_{\pr_{12}}(\pr_{23}^*\mbf E,\pr^*_{13}\mbf E) \ar[r]  & \Ext^1_{\pr_{12}}(\pr^*_{23}\mbf A_1,\pr^*_{13}\mbf E) \ar[r] & \Ext^1_{\pr_{12}}(\pr^*_{23}\mbf  A_0,\pr_{13}^*\mbf E)}\\
\xymatrix@=0.5cm{
  \ar[r] & \Ext^2_{\pr_{12}}(\pr_{23}^*\mbf E,\pr^*_{13}\mbf E) \ar[r]  & \Ext^2_{\pr_{12}}(\pr^*_{23}\mbf A_1,\pr^*_{13}\mbf E) \ar[r] & \Ext^2_{\pr_{12}}(\pr^*_{23}\mbf A_0,\pr_{13}^*\mbf E) \ar[r] & 0}\end{split}\end{equation}
(see Proposition \ref{lemm exact se ext} in Appendix \ref{app rela ext}).
Conditions (1)-(4) then follow directly.
\end{proof}

\begin{corollary}\label{thm connectedness of moduli}
$\mc M_\upsilon$ is connected.
\end{corollary}

\begin{proof}
Using Theorem \ref{thm genetators of moduli} and the base change theorem for relative Ext-sheaves on DM stacks (see Appendix \ref{app rela ext}), the proof of this corollary can be completed by following the proof of Corollary 10 in \cite{ma07}.
\end{proof}

\section{Orbifold Hilbert schemes}\label{sec orb Hilb schemes}

\subsection{Hilbert schemes}

Let $\mc Z$ be a zero-dimensional closed substack of $\mc X$ with coarse moduli space $Z$.

\begin{lemma}\label{lemm local str of DM}
For a stacky point $p$ with stabilizer group $G$, there exists a cartesian diagram
\begin{equation}\label{diag local str of DM}
\begin{tikzcd}[scale=0.6]
{[\spec(\mbb C[\![x,y]\!])/G]} \arrow[r]\arrow[d] & \mc X \arrow[d,"\pi"]\\
{\spec(\mbb C[\![x,y]\!]^G)} \ar[r] & X
\end{tikzcd}	
\end{equation}
where the $G$-action on $\mbb C[\![x,y]\!]$ is induced by the $G$-action on the cotangent space of $\mc X$ at $p$. Moreover, $\mbb C[\![x,y]\!]^G$ is isomorphic to the complete local ring $\wh{\mc O}_{X,p}$ of $X$ at $p$, and $\mbb C[\![x,y]\!]^G\rightarrow X$ coincides with $\spec(\wh{\mc O}_{X,p})\rightarrow X$.
\end{lemma}

\begin{proof}
By \cite[Theorem 11.3.1]{ol16}, there exists an \'etale neighborhood $\spec(B)\to X$ of $p$ and a cartesian diagram
\begin{equation*}
\begin{tikzcd}[scale=0.6]
{[\spec(A)/G]} \arrow[d]\arrow[r] & \mc X \arrow[d,"\pi"] \\
\spec(B) \arrow[r] & X	
\end{tikzcd}
\end{equation*}
with $B=A^G$. Let $\mfr n\subseteq B$ be the maximal ideal corresponding to a lift $p'\in\spec(B)$ of $p$, and let $\mfr m\subseteq A$ be the unique maximal ideal over $\mfr n$. Then $\wh A_\mfr m \cong\; A\otimes_B \wh B_\mfr n$. By the Cohen structure theorem, $\wh A_\mfr m\cong\mbb C[\![x,y]\!]$, hence $\wh B_\mfr n=\mbb C[\![x,y]\!]^G$. This yields the cartesian diagram \eqref{diag local str of DM}, where $\spec(\mbb C[\![x,y]\!]^G)\to X$ identifies with the formal neighborhood of $p$.
\end{proof}

\begin{lemma}\label{lemm local str subst}
If the coarse moduli space $Z$ of $\mc Z$ satisfies $\supp(Z)=\{p\}$, then $\mc Z \;\cong\; [\spec(\mbb C[\![x,y]\!]/I_\mc Z)/G]$ for some $G$-invariant ideal $I_\mc Z$ of finite colength.
\end{lemma}

\begin{proof}
We may write $Z=\spec(\mc O_{X,p}/I_Z)$ for some ideal $I_Z$ of finite colength. Using Lemma \ref{lemm local str of DM}, we obtain a cartesian diagram
\begin{equation*}
\begin{tikzcd}[scale=0.6]
Z\times_X\mc X \arrow[d] \arrow[r] & {[\spec(\mbb C[\![x,y]\!])/G]} \arrow[d] \arrow[r] & \mc X \arrow[d,"\pi"] \\
Z  \arrow[r] & {\spec(\mbb C[\![x,y]\!]^G)} \ar[r] & X.
\end{tikzcd}
\end{equation*}
Since $\mc Z$ is a closed substack of $Z\times_X\mc X$, it is also a closed substack of $[\spec(\mbb C[\![x,y]\!])/G]$.
\end{proof}

\begin{lemma}
If the decomposition of $A:=\mbb C[\![x,y]\!]/I_\mc Z$ into irreducible representations is
\begin{equation*}
\textstyle{A\cong\bigoplus_{0\leq i\leq t}\rho_i^{\oplus v_i}}\quad\text{for some integers $v_i$},
\end{equation*}
then $[\mc O_\mc Z]=\sum_{0\leq i\leq t}v_i[\mc O_p\otimes\rho_i]$ in $K_0(\mc X)$ (see Section \ref{sec orb chern char} for notation).
\qed
\end{lemma}

\begin{lemma}
$\big\{[\mc O_q], [\mc O_{p_k}\otimes\rho_{k,1}],[\mc O_{p_k}\otimes\rho_{k,2}],\ldots,[\mc O_{p_k}\otimes\rho_{k,r_k}]\big\}_{k\in\mfr I}$ is linearly independent in $K^{\rm num}(\mc X)$ (see Definition \ref{def num K gp}), where $q$ is any non stacky point of $\mc X$ and $\{\rho_{k,0},\cdots,\rho_{k,r_k}\}$ are the sets of irreducible representations of $G_k$ whose trivial representations are $\rho_{k,0}$.
\end{lemma}

\begin{proof}
Recall the character of the regular representation of $G_k$ is
\begin{equation*}
\chi_{\rm reg}(g)=
\begin{cases}
|G_k| \quad\text{if $g$ is the identity element},\\
0 \quad \text{otherwise.}
\end{cases}
\end{equation*}
Using Proposition \ref{pro ch of rho} and Corollary \ref{cor chern char of num k gp}, we obtain the lemma by a direct computation of the orbifold Chern characters.
\end{proof}

\begin{corollary}\label{cor lattice for hilb}
$\mc N:=\mbb Z[\mc O_q]+\sum_{k\in\mfr I}\sum_{1\leq i\leq r_k}\mbb Z[\mc O_{p_k}\otimes\rho_{k,i}]$ is a lattice whose rank is $1+\sum_{k\in\mfr I}r_k$.
\qed
\end{corollary}

Fix a K-class $\alpha=n_0[\mc O_q]+\sum_{k\in\mfr I}\sum_{1\leq i\leq r_k}n_{k,i}[\mc O_{p_k}\otimes\rho_{k,i}]\in \mc{N}$ with $n_0\geq 0$ and $n_{k,i}\geq 0$. Consider the Hilbert scheme
\begin{equation}\label{equ orb hilb sch}
\hilb^\alpha(\mc X):=\{\text{closed substacks $\mc Z\subseteq\mc X$}|\quad[\mc O_\mc Z]=\alpha\}.
\end{equation}

\begin{proposition}\label{pro hilb sch}
$\hilb^\alpha(\mc X)$ is a smooth projective scheme with dimension
\begin{equation}\label{equ dim of hilb}
\begin{split}
2n_0+\sum_{k\in\mfr I}\bigg( \sum_{1\leq i\leq r_k}n_{k,i}\langle\chi_{\rho_{k,0}},\chi_{\rho_{k,i}\otimes\rho_{K_\mc X}}\rangle-\sum_{1\leq i,j\leq r_{k}}n_{k,i}n_{k,j}\big(\langle\chi_{\rho_{k,j}},\chi_{\rho_{k,i}}\rangle\\
+\langle\chi_{\rho_{k,j}},\chi_{\rho_{k,i}\otimes\rho_{K_\mc X}}\rangle-\langle\chi_{\rho_{k,j}},\chi_{\rho_{k,i}\otimes\rho_{\Omega_\mc X}}\rangle\big)\bigg).
\end{split}
\end{equation}
\end{proposition}

\begin{proof}
First, $\hilb^\alpha(\mc X)\rightarrow\mc M_\upsilon^{\mc O_\mc X}$, $\mc Z\mapsto I_\mc Z$
is an isomorphism, where $I_\mc Z$ is the ideal sheaf of $\mc Z$ and $\mc M_\upsilon^{\mc O_\mc X}$ is the moduli space of stable sheaves with trivial determinant and K-class $\upsilon=[\mc O_\mc X]-\alpha\in K^{\rm num}(\mc X)$. Hence $\hilb^\alpha(\mc X)$ is a smooth projective scheme (Proposition \ref{pro sm of mod}). The tangent space at $I_\mc Z$ is $\ext^1(I_\mc Z,I_\mc Z)_0$. Since $I_{\mc Z}^{\vee\vee}$ is a line bundle, we have $\homm(I_{\mc Z},I_\mc Z\otimes K_\mc X)=\homm(\mc O_\mc X,K_\mc X)$. By Serre duality, $\ext^2(I_\mc Z,I_\mc Z)_0=\homm(\mc O_\mc X,K_\mc X)_0^\vee=0$. Using Euler characteristic (Appendix \ref{subsec euler form}), $\dimm\hilb^{\alpha}(\mc X)=\chi(\mc O_\mc Z)-\chi(I_\mc Z,I_\mc Z)$, where
\begin{equation*}
\begin{split}
\chi(I_\mc Z,I_\mc Z)&=\chi(\mc O_\mc X)-2n_0-\sum_{k\in\mfr I}\bigg(\sum_{1\leq i\leq r_k}n_{k,i}\langle\chi_{\rho_{k,0}},\chi_{\rho_{k,i}\otimes\rho_{K_\mc X}}\rangle\\
&+\sum_{1\leq i,j\leq r_k}n_{k,i}n_{k,j}(\langle\chi_{\rho_{k,j}},\chi_{\rho_{k,i}}\rangle+\langle\chi_{\rho_{k,j}},\chi_{\rho_{k,i}\otimes\rho_{K_\mc X}}\rangle-\langle\chi_{\rho_{k,j}},\chi_{\rho_{k,i}\otimes\rho_{\Omega_\mc X}}\rangle)\bigg)
\end{split}
\end{equation*}
(see Proposition \ref{pro comp of euler form}). 
\end{proof}

\begin{proposition}\label{thm hilbert-chow poi}
$\hilb^\alpha(\mc X)$ is a smooth connected projective scheme provided that $K_\mc X\cdot H<0$ or $K_\mc X\cong\mc O_\mc X$. Furthermore, if $\mc X$ admits a Poisson structure, then $\hilb^\alpha(\mc X)$ admits one as well.
\end{proposition}

\begin{proof}
Under our assumption, Corollary \ref{thm connectedness of moduli} implies that $\mc M_\upsilon$ is a connected smooth projective scheme with $\upsilon=[\mc O_\mc X]-\alpha\in K^{\rm num}(\mc X)$. If nonempty, there is an isomorphism 
\begin{equation*}
\mc M_\upsilon\rightarrow\Pic^0(\mc X)\times\hilb^{\alpha}(\mc X),\quad F\mapsto (\dett(F), (\dett(F)/F)\otimes\dett(F)^{-1}).
\end{equation*}
By Proposition \ref{pro rep of an pic}, $\Pic^0(\mc X^{\rm an})$ is a complex torus. Hence $\Pic^0(\mc X)$ is connected (via GAGA theorem for DM stacks) and $\hilb^{\alpha}(\mc X)$ is connected as well. If $\mc X$ carries a Poisson structure, so does $\hilb^\alpha(\mc X)$ (Proposition \ref{pro poi st on moduli}). 
\end{proof}
\begin{remark}\label{rem poi st on hilb}
Alternatively, the tangent and cotangent spaces at an ideal sheaf $I_\mc Z$ are $T_{I_\mc Z} \hilb^\alpha(\mc X) = \homm(I_\mc Z, \mc O_\mc Z)$, and $T^\vee_{I_\mc Z} \hilb^\alpha(\mc X) = \homm(I_\mc Z, \mc O_\mc Z \otimes K_\mc X)$. If $\theta\in H^0(\mc X, \wedge^2 T\mc X)$ is a Poisson structure on $\mc X$, it induces a morphism $B_\theta(I_\mc Z) : \homm(I_\mc Z, \mc O_\mc Z \otimes K_\mc X) \rightarrow \homm(I_\mc Z, \mc O_\mc Z)$, defining the Poisson structure on $\hilb^\alpha(\mc X)$. Moreover, by \cite[Lemma 3.2]{bo98}, the kernel of $B_\theta(I_\mc Z)$ is $\homm(I_\mc Z, \Tor_1^{\mc O_\mc X}(\mc O_\mc Z, \mc O_\mc D))$, where $\mc D$ is the Cartier divisor determined by $\theta$.
\end{remark}

\subsection{Hilbert-Chow morphism}\label{subsec HC}
While the previous subsection treated the projective case, the Hilbert scheme for a quasiprojective orbifold surface can be defined via Quot functor on Deligne-Mumford stacks (\cite{os03}). Below we define the Hilbert-Chow morphism only in the projective case; the quasiprojective case can be treated in the same way. Let $\mc Z_{\hilb^\alpha}\subset\hilb^\alpha(\mc X)\times\mc X$ be the universal closed substack with ideal sheaf $\mc I_{\hilb^\alpha}$, flat over $\hilb^\alpha(\mc X)$. By the exactness of $(\pi\times\id_{\hilb^\alpha})_*$ ( \cite[Theorem 3.2]{aov08}), $(\pi\times\id_{\hilb^\alpha})_*\mc I_{\hilb^\alpha}$ is a sheaf of ideals on $\hilb^\alpha(\mc X)\times X$, flat over $\hilb^\alpha(\mc X)$ (\cite[Corollary 1.3]{ni08}). Indeed, $(\pi\times\id_{\hilb^\alpha})_*\mc I_{\hilb^\alpha}$ is the ideal sheaf of the coarse moduli space of $\mc Z_{\hilb^\alpha}$(see \cite[the proof of Lemma 3.7]{hu24} ), defining a morphism 

\begin{equation}\label{equ mor orb to coar}
\hilb^\alpha(\mc X)\longrightarrow\hilb^{n_0}(X). 
\end{equation}
Composing with the classical Hilbert-Chow morphism
\begin{equation}\label{equ class HC}
\hilb^{\alpha}(X)\rightarrow\sym^{n_0}(X),\quad Z\mapsto\sum_{z\in X(\mbb C)}\text{length}(\mc O_{Z,z})\cdot z,
\end{equation}
we obtain the Hilbert-Chow morphism $h : \hilb^\alpha(\mc X)\rightarrow\sym^{n_0}(X)$. We remark that $h$ is, in general, not surjective; however, it is surjective for $\alpha=n[\mc O_q]$. For brevity, we denote $\hilb^{n[\mc O_q]}(\mc X)$ by $\hilb^n(\mc X)$. At this point, we state one of our main results.

\begin{theorem}\label{thm connectedness of hilb n}
\begin{enumerate}[$(1)$]
\item $\hilb^n(\mc X)$ is a connected smooth projective scheme.
\item  The Hilbert-Chow morphism
\begin{equation}\label{equ HC n}
 h : \hilb^{n}(\mc X)\longrightarrow\sym^n(X)
\end{equation}
is a resolution of singularities. 
\item If $\mc X$ equipped with a Poisson structure, then $h$ is a Poisson resolution with respect to the induced Poisson structures.
	\end{enumerate}
\end{theorem}

Before giving the proof, we need to establish a technique lemma. Let $G\subseteq\GL_2(\mbb C)$ be a small finite subgroup, that is a finite subgroup acting freely on $\mbb C^2\setminus\{0\}$. Let $\hilb^n([\mbb C^2/G])$ (resp. $\hilb^n([\spec (\mbb C[\![x,y]\!])/G])$ ) be the scheme parametrising $G$-invariant ideals $I$ in $\mbb C[x,y]$ (resp. $\mbb C[\![x,y]\!]$) such that $\mbb C[x,y]/I$ (resp. $\mbb C[\![x,y]\!]/I$) is isomorphic to the direct sum of n copies of regular representation $\rho_{\rm reg}$ of $G$.

\begin{lemma}\label{pro conn gl2}
	$\hilb^{n}([\mbb C^2/G])$ and $\hilb^n([\spec (\mbb C[\![x,y]\!])/G])$ are connected.
\end{lemma}

\begin{proof}
Define a $G$-representation by 
\begin{equation*}
G\rightarrow\GL_3(\mbb C)\quad g\mapsto {\rm diag}(g,1)
\end{equation*}
which induce a $G$-action on $\mds P^2$. The quotient stack $[\mds P^2/G]$ compactifies $[\mbb C^2/G]$. The coarse space map factors as
\begin{equation}\label{equ diag can decom of sec G-action}
\begin{tikzcd}[row sep=1.5em, column sep=1.5em]
{[\mds P^2/G]} \arrow[dr] \arrow[rr] && {(\mds P^2/G)}^{\rm can} \arrow[dl] \\
			& \mds P^2/G &
\end{tikzcd}
\end{equation}
where $(\mds P^2/G)^{\rm can}$ is the canonical stack. The top arrow is an isomorphism over $[\mbb C^2/G]$ ( \cite[Theorem 1]{gs17}). Hence $(\mds P^2/G)^{\rm can}$ is also a compactification of $[\mbb C^2/G]$, and then $\hilb^{n}([\mbb C^2/G])$ is an open subscheme of $\hilb^n((\mds P^2/G)^{\rm can})$. For some sufficient large integer $N$, $\mc O_{\mds P^2}(N)$ descends to a very ample line bundle $\mc O_{\mds P^2/G}(H)$ on $\mds P^2/G$. Note that $K_{(\mds P^2/G)^{\rm can}}\cdot H<0$. By Corollary \ref{thm connectedness of moduli}, $\hilb^{n}([\mbb C^2/G])$ is a connected smooth quasiprojective scheme of dimension $2n$. Let $o$ be the singular point of $\mbb C^2/G$. The fiber of the Hilbert-Chow morphism $ \hilb^n([\mbb C^2/G])\rightarrow\sym^n(\mbb C^2/G)$ over $o$ with reduced scheme structure is isomorphic to $\hilb^n([\spec (\mbb C[\![x,y]\!])/G])_{\rm red}$, which is a projective scheme.
Since $\sym^n(\mbb C^2/G)\cong\mbb C^{2n}/G_n$ where $G_n$ is the wreath product of $G$ with the $n$-th symmetric group $S_n$, $\sym^n(\mbb C^2/G)$ is normal. By the Zariski's Main Theorem, $\hilb^n([\spec (\mbb C[\![x,y]\!])/G])$ is connected.
\end{proof}

\begin{proof}[Proof of Theorem \ref{thm connectedness of hilb n}]
Without loss of generality, assumed that $\mc X$ has exactly one orbifold point $p$. We have a natural stratification of the Cartesian product $X^n$:
\begin{equation*}
	X^n=\coprod_{k=0}^n X^n[k],
\end{equation*}
where $X^n[k]$ denote the locus of points in $X^n$ with exactly $k$ components equal to the singular point $p$. This induces a stratification of $\sym^n(X)$: 
\begin{equation*}
	\sym^n(X)=\coprod_{k=0}^n\sym^n(X)[k],
\end{equation*}
where $\sym^n(X)[k]=X^n[k]/S_n$ which are isomorphic to $\sym^{n-k}(X\setminus\{p\})$. Moreover, one has
\begin{equation*}
 h^{-1}(\sym^n(X)[k])=\hilb^k([\spec (\mbb C[\![x,y]\!])/G])\times\hilb^{n-k}(\mc X^*)
\end{equation*}
with $\mc X^*=\pi^{-1}(X\setminus\{p\})$. By the connectedness of $\hilb^{n-k}(\mc X^*)$ (\cite[Lemma 7.2.1]{fgiknv05}),  $h^{-1}(\sym^n(X)[k])$ is connected as well. The closure of $\hilb^n(\mc X^*)\subset\hilb^n(\mc X)$ is connected, and its image under $h$ covers $\sym^n(X)$. Hence $\hilb^n(\mc X)$ is connected. Since $h$ restricts to a Poisson resolution  $\hilb^n(\mc X^*)\rightarrow\sym^n(X^*)$, it follows that $h$ is a Poisson resolution with respect to the induced Poisson structures (\cite[Corollary 5.2]{fu05}).
\end{proof}

\begin{corollary}
Suppose that $W$ is a smooth connected quasiprojective scheme with an action of a finite group $G$. If the fixed locus is dimension zero, then
\begin{equation*}
\hilb^n([W/G])=\{Z\subseteq W|\text{$Z$ is a $G$-invariant closed subscheme with $H^0(\mc O_Z)\cong\rho_{\rm reg}^{\oplus n}$}\}
\end{equation*} 
is a smooth connected quasiprojective scheme.
\end{corollary}
\begin{proof}
First we can $G$-equivariantly embed $W$ into a projective space. The closure of $W$ with reduced scheme structure is denoted by $\wt W$, which is $G$-invariant. Let $\overline W$ be a $G$-equivariant resolution of $\wt W$ (see \cite{ko07}). Then $\overline W$ gives a $G$-equivariant compactification of $W$, and consequently, $[\overline W/G]$ is a compactification of $[W/G]$. If the stacky locus of $[\overline W/G]$ is codimension one, we consider the canonical stack $[\overline W/G]^{\rm can}$. By Theorem \ref{thm connectedness of hilb n}, the conclusion is immediate.
\end{proof}

\begin{corollary}
Let $X$ be an irreducible symplectic projective surface with quotient singularities and let $\mc X$ be its associated canonical stack. Then the Hilbert-Chow morphism 
\[
h :	\hilb^n(\mc X)\longrightarrow\sym^n(X)
\]
is a symplectic resolution.
\end{corollary}

\begin{proof}
It follows from Proposition 2.4 in \cite{bea00} that $\sym^n(X)$ has symplectic singularities.  Since $h$ is a projective resolution, the symplectic form on the smooth locus of $\sym^n(X)$ extends to a holomorphic two-form on $\hilb^n(\mc X)$.

On the other hand, under our assumption, $\mc X$ is a symplectic orbifold surface with only finitely many orbifold points; in particular, $\hilb^n(\mc X)$ carries a natrual symplectic structure. These two forms coincide, and hence the Hilbert-Chow morphism is a symplectic resolution.
\end{proof}

\begin{theorem}\label{thm minimal resl}
$h : \hilb^{1}(\mc X)\rightarrow X$ is the minimal resolution.
\end{theorem}

\begin{proof}
By \cite[Proposition 2.18]{ko07}, it suffices to check \'{e}tale-locally around the orbifold points of $X$. For an orbifold point, there exists an \'{e}tale neighborhood $\spec(A^G)\to X$ such that $\spec(A^G)\times_X \mc X \cong [\spec(A)/G]$, where $A$ is a smooth connected affine surface with $G$-action. By \cite[Proposition 2.3]{ct08}, $\spec(A^G)\times_X \hilb^{[\mc O_q]}(\mc X)$ is the component of $G$-$\hilb(\spec(A))$ containing free $G$-orbits. Recall a $G$-cluster $Z$ is a $G$-invariant finite subscheme of $\spec(A)$ with $H^0(\spec(A),\mc O_Z) \cong \mbb C[G]$. Since $\hilb^{|G|}(\spec(A))$ is smooth \cite[Theorem 7.2.3 (2)]{fgiknv05} and the $G$-action is linearizable, $G$-$\hilb(\spec(A))$ is smooth. Denote $\spec(A^G)\times_X \hilb^{[\mc O_q]}(\mc X)$ by $Y$.  

\textbf{Claim:} $Y$ is the minimal resolution of $\spec(A^G)$. Let $f : \wt Y\to \spec(A^G)$ be the minimal resolution and $g : \spec(A)\to \spec(A^G)$ the quotient map. The graph $\Gamma_g : \spec(A)\hookrightarrow \spec(A)\times \spec(A^G)$ defines a $G$-equivariant surjection $\mc O_{\spec(A)\times \spec(A^G)} \to \Gamma_{g*}\mc O_{\spec(A)}$. Pushing forward to $\spec(A^G)$ gives $\mc O_{\spec(A^G)}\otimes_\mbb C A \to g_* \mc O_{\spec(A)}$, and pulling back along $f$ yields $\mc O_{\wt Y}\otimes_\mbb C A \to f^* g_* \mc O_{\spec(A)}$. Modding out torsion, $f^* g_* \mc O_{\spec(A)}/{\rm torsion}$ is locally free of rank $|G|$ \cite[Lemma 2.2]{es85}, defining a $G$-cluster family over $\wt Y$. The induced morphism $\wt Y \to Y$ is an isomorphism outside the singular fiber, so minimality of $\wt Y$ implies $Y \cong \wt Y$.
\end{proof}

\section{Compactification of the Hitchin systems}\label{sec cpt hitchin systems}
In this section we restrict to the two-dimensional Hitchin systems corresponding to the affine Dynkin diagrams $\wt A_0$, $\wt D_4$, $\wt E_6$, $\wt E_7$ and $\wt E_8$, constructed by Groechenig (\cite{go14}). Concretely, they are moduli spaces of orbifold Higgs bundles on $E$, $\mds P_{2,2,2,2}^1$, $\mds P_{3,3,3}^1$, $\mds P_{4,4,2}^1$, $\mds P_{6,3,2}^1$, where $E$ is an elliptic curve and $\mds P^1_{a_1,\cdots,a_s}$ denotes an orbifold curve with coarse moduli space $\mds P^1$ and $s$ orbifold points of the specified orders. These are exactly all the one-dimensional Calabi-Yau orbifolds. Except $E$, each arises as a nontrivial quotient of an elliptic curve by a cyclic group:
\begin{equation*}
	[E_2/\mu_2],\quad[E_3/\mu_3],\quad [E_4/\mu_4],\quad [E_6/\mu_6].
\end{equation*}

For brevity, we denote the $\mu_i$-Hilbert scheme $\mu_i$-$\hilb(T^\vee E_i)$ by $\hilb^1(T^\vee\mc X_i)$.

\begin{theorem}[\cite{go14}]
$\hilb^1(T^\vee\mc X_i)$ is isomorphic to a two-dimensional moduli space $\mc M(i)$ of stable orbifold Higgs bundles on $\mc X_i:=[E_i/\mu_i]$. In particular these moduli spaces are crepant resolutions of the GIT quotients $T^\vee E_i/\mu_i$.
\qed
\end{theorem}

\begin{remark}\label{rmk equiv isom}
Indeed, Groechenig proved that the $\mbb C^*$-action on $\mu_i$-$\hilb(T^\vee E_i)$ induced by the natural $\mbb C^*$ action on $T^\vee E_i$ coincides with the natural $\mbb C^*$-action on $\mc M(i)$, although this is not stated explicitly in his paper. In addition, this isomorphism is a symplectomorphism with respect to their natural symplectic structures (see \cite{jia25}).
\end{remark}



\begin{proposition}[\cite{go14}]\label{pro hitchin map}
Composing the Hilbert-Chow morphism with the coarse map of $\psi_i$ yields
\begin{equation}\label{equ hitchin map}
	\hilb^1(T^\vee\mc X_i)\rightarrow E_i\times\mbb C/\mu_i\rightarrow\mbb C/\mu_i\cong\mbb C
\end{equation}
which is isomorphic the corresponding Hitchin map.
\qed
\end{proposition}



\begin{lemma}\label{lemm pun hilb}
If the formal power series ring $\mbb C[\![x,y]\!]$ is equipped with the $\mu_r$-action $\zeta_r(x)=\zeta_r\cdot x$ and $\zeta_r(y)=\zeta_r\cdot y$, where $\zeta_r$ is a primitive $r$-th root of unit, then $\hilb^1([\spec(\mbb C[\![x,y]\!])/\mu_r])\cong\mds P^1$.
\end{lemma}
\begin{proof}
Since any $I\in\hilb^1([\spec(\mbb C[\![x,y]\!])/\mu_r])$ can be represented as $I=(ax+by)+(x,y)^r$
for some $[a,b]\in\mds P^1$, we complete the proof.
\end{proof}

\begin{theorem}\label{thm cpf of two higgs}
The two-dimensional Hitchin systems for $\wt D_4$, $\wt E_6$, $\wt E_7$, and $\wt E_8$ admit natural compactifications 
\[
\hilb^{1}\!\bigl(\mds P(T^\vee\mc X_i\oplus\mc O_{\mc X_i})\bigr)
\]
with the following properties:
\begin{enumerate}[$(1)$]
\item The natural $\mbb{C}^*$-action and Poisson structure on $\hilb^1(\mds P(T^\vee\mc X_i \oplus \mc O_{\mc X_i}))$ are compatible with, and extend, the $\mbb{C}^*$-action and symplectic structure on $\mc M(i)$.
\item The Hitchin maps extend to the compositions
\[
\hilb^1\!\bigl(\mds P(T^\vee\mc X_i\oplus\mc O_{\mc X_i})\bigr)
\xrightarrow{\,h_i\,} E\times\mds P^1/\mu_i
\longrightarrow \mds P^1/\mu_i \cong \mds P^1,
\]
where $h_i$ are the Hilbert–Chow morphisms.
\item Each $h_i$ is the minimal resolution of the GIT quotient $\mds P(T^\vee E_i\oplus\mc O_{E_i})/\mu_i$, and provides a Poisson resolution.
\item The boundary (with reduced structure) consists of $s+1$ copies of $\mds P^1$, where $s$ is the number of orbifold points of $\mc X_i$.
\end{enumerate}
\end{theorem}

\begin{proof}
Note that $\hilb^1(T^\vee\mc X_i)$ is an open subscheme of $\hilb^{1}(\mds P(T^\vee\mc X_i\oplus\mc O_{\mc X_i}))$. Since $\mds P(T^\vee\mc X_i\oplus\mc O_{\mc X_i})$ carries a Poisson structure $\theta$ extending the symplectic form on $T^\vee \mc X_i$ (Proposition \ref{pro pois st on cp ct bd}), $\hilb^{1}(\mds P(T^\vee \mc X_i\oplus\mc O_{\mc X_i}))$ is a connected smooth Poisson projective surface, and the Hilbert-Chow morphism is a minimal (Poisson) resolution (Theorem \ref{thm hilbert-chow poi}, Theorem \ref{thm minimal resl} and \cite[Proposition 3.3]{fu05}). By Remark \ref{rmk equiv isom}, the natural $\mbb{C}^*$-action and Poisson structure on $\hilb^1(\mds P(T^\vee\mc X_i \oplus \mc O_{\mc X_i}))$ are compatible with, and extend, the $\mbb{C}^*$-action and symplectic structure on $\mc M(i)$. By Proposition \ref{pro hitchin map}, the condition (2) is immediate. The degenerate locus of $\theta$ is the divisor $\mc D=2\cdot\mds P(T^\vee \mc X_i)$, so the degenerate locus of the natural Poisson structure $B_\theta$ on $\hilb^{1}(\mds P(T^\vee\mc X_i\oplus\mc O_{\mc X_i}))$ is
\begin{equation*}
\{\mc Z\in\hilb^{1}(\mds P(T^\vee\mc X_i\oplus\mc O_{\mc X_i}))|\quad\mc Z\cap\mc D\neq\emptyset\}
\end{equation*}
(Remark \ref{rem poi st on hilb}), which is the complement of $\mu_r$-$\hilb(T^\vee E)$ in
$\hilb^{1}(\mds P(T^\vee \mc X_i\oplus\mc O_{\mc X_i}))$. With reduced structure, it is
\begin{equation*}
\textstyle{\hilb^{1}(\mds P(T^\vee\mc X_i))\bigcup_{i=1}^s\text{$\mu_{a_i}$-$\hilb(\mbb C[\![x,y]\!])$}}.
\end{equation*}
Since the corase moduli space of $\mds P(T^\vee\mc X_i)$ is $\mds P^1$, then $\hilb^{1}(\mds P(T^\vee\mc X_i))$ is isomorphic to $\mds P^1$. We complete the proof by applying Lemma \ref{lemm pun hilb}.

\end{proof}

In what follows, we will discuss each case individually.
\subsection{$\wt D_4$-case}\label{subsec ell 2}
Recall that any elliptic curve can be written in Weierstrass form $E_{(a,b)} : y^2=x^3+ax+b$. The involution given by negation in the group law is $\tau : E_{(a,b)}\rightarrow E_{(a,b)}$, $(x,y)\mapsto (x,-y)$, and induces an $\mu_2$-action on $E_{(a,b)}$ with four fixed points. Consider the orbifold curve $\mc X_2=[E_{(a,b)}/\mu_2]$ whose coarse moduli space is $\mds P^1$ with four orbifold points $p_1$, $p_2$, $p_3$, $p_4$. The differential $\omega=dx/2y=dy/(3x^2+a)$ is a globally defined one-form on $E_{(a,b)}$ satisfying $\tau^*\omega=-\omega$. Hence $\mds P(T^\vee\mc X_2\oplus\mc O_{\mc X_2})=[E_{(a,b)}\times\mds P^1/\mu_2]$, where the $\mu_2$-action on $E_{(a,b)}\times\mds P^1$ is $\tau :  (p,[z_0,z_1])\mapsto(\tau(p),[z_0,-z_1])$.
The projections $E\times\mds P^1\rightarrow E_{(a,b)}$ and $E_{(a,b)}\times\mds P^1\rightarrow\mds P^1$ are  $\mu_2$-equivariant, hence they descent to two morphisms

\begin{equation}\label{diag coarse of ell 2}
\begin{tikzcd}[row sep=1.6em, column sep=0.3em]
& X_2 \arrow[dl,swap,"\wt\pr_1"] \arrow[dr,"\wt\pr_2"] &  \\
{\mds P^1\cong E_{(a,b)}/\mu_2} & & {\mds P^1/\mu_2\cong\mds P^1}
\end{tikzcd}
\end{equation}
where $X_2=E_{(a,b)}\times\mds P^1/\mu_2$. Composing these with the minimal resolution $\wt\pi_2 : \wt X_2\rightarrow X_2$ yields two fibrations
\begin{equation}\label{equ fibr for ell 2}
	\begin{tikzcd}[row sep=2em, column sep=2em]
		& \wt X_2 \arrow[dl,swap,"\pi^\prime_2"] \arrow[dr,"\pi_2"] &  \\
\mds P^1 & & \mds P^1
	\end{tikzcd}
\end{equation}

\begin{lemma}\label{lemm H fib for ell 2}
$\pi_2^\prime : \wt X_2\rightarrow\mds P^1$ is a generically $\mds P^1$-fibration, which has exactly four singular fibers over the four orbifold points. More precisely, there exist smooth rational curves $\{D_i\}_{1\leq i\leq 4}$, $\{E_i\}_{1\leq i\leq 4}$, $\{F_i\}_{1\leq i\leq 4}$ such that $\pi_2^{\prime-1}(p_i)=2D_i+E_i+F_i$, whose dual graphs are
\begin{center}
\begin{tikzpicture}[scale=0.5]
		\node[circle, fill=black, inner sep=1.5pt] (A) at (0,2)    {};
		\node[circle, fill=black, inner sep=1.5pt] (C) at (2,2)    {};
		\node[circle, fill=black, inner sep=1.5pt] (D) at (4, 2)   {};
		
		\draw (A)--(C);
		\draw (C)--(D);
		
		\node [below]   at (A)  {$-2$};
		\node [below]   at (C)  {$-1$};
		\node [below]    at (D)  {$-2$};

		\node [above]   at (A)  {$E_i$};
		\node [above]   at (C)  {$D_i$};
		\node [above]    at (D)  {$F_i$};
\end{tikzpicture}
\end{center}
\end{lemma}

\begin{proof}
Suppose that $q_1$, $q_2$, $q_3$, $q_4$ are the fixed points of $\tau$, corresponding respectively to $p_1$, $p_2$, $p_3$, $p_4$. For each $q_i$, we obtain a smooth rational curve $\wt D_i=\{q_i\}\times\mds P^1/\mu_2$ on $X_2$, and 
\begin{equation}\label{equ Divisor for ell 2}
	\wt\pr_1^{-1}(p_i)=2\wt D_i. 
\end{equation}		
The only singular points of $X_2$ lying on $\wt D_i$ are $0$ and $\infty$. Let $E_i$ and $F_i$ be the exceptional curves over the singular points $0$, $\infty$, respectively. Since all the singular points of $X_2$ are of type $\frac{1}{2}(1,1)$, we have $E_i^2=F_i^2=-2$ and $K_{\wt X_2}\cdot E_i=K_{\wt X_2}\cdot F_i=0$. From (\ref{equ Divisor for ell 2}), it follows that $\pi_2^{\prime-1}(p_i)=2D_i+n_iE_i+m_iF_i$ for some natural numbers $n_i$ and $m_i$, where $D_i$ denotes the strict transform of $\wt D_i$ on $\wt X_2$. Intersecting with $E_i$ gives $(2D_i+n_iE_i+m_iF_i)\cdot E_i=0$, and hence $n_i=D_i\cdot E_i$. Analogously, we have $m_i=D_i\cdot F_i$. Hence, we obtain $2D_i^2+n_i^2+m_i^2=0$, so in particular $ D_i^2<0$. By the adjunction formula, $(2D_i+n_iE_i+m_iF_i)\cdot K_{\wt X_2}=-2$, which implies $D_i\cdot K_{\wt X_2}=-1$. Thus $D_i$ is an exceptional curve of first kind, and consequently, $n_i=m_i=1$. 
\end{proof}

\begin{lemma}\label{lemm E fib for ell 2}
The elliptic fibration $\pi_2 : \wt X_2\rightarrow\mds P^1$ is relatively minimal and 	
the singular fibers are the fibers over $0$ and $\infty$, which are of type $\uppercase\expandafter{\romannumeral 1}_0^\ast(\wt D_4)$. 

\end{lemma}

\begin{proof}
Note that $\wt\pr_2^{-1}(0)=2\wt D_0$, where $\wt D_0=E_{(a,b)}\times\{0\}/\mu_2$ is a smooth rational curve on $X_2$ containing four singular points of type $\frac{1}{2}(1,1)$. Hence $\pi_2^{-1}(0)=2D_0+\sum_{i=1}^4\wt n_iE_i$ with $D_0$ the strict transformation of $\wt D_0$ in $\wt X_2$ and $n_i\in\mbb N$. For each $E_i$, $(2D_0+\sum_{j=1}^{4}\wt n_jE_j)\cdot E_i=0$ gives $\wt n_i=D_0\cdot E_i$. Consequently, $2D_0^2+\sum_{i=1}^4\wt n_i^2=0$ and $D_0^2\leq -2$. By the adjunction formula, $K_{\wt X_2}\cdot (2D_0+\sum_{j=1}^4\wt n_jE_j)=0$, and then  $K_{\wt X_2}\cdot D_0=0$. Since $\frac{1}{2}(D_0\cdot D_0+K_{\wt X_2}\cdot D_0)+1$ is a nonnegative integer, we obtain $D_0^2=-2$ and $\wt n_1=\wt n_2=\wt n_3=\wt n_4=-1$. Analogously, we can show that $D_\infty$ satisfies similar equations.
\end{proof}

For the second Hirzebruch surface $h_2 : H_2=\mds P(\mc O(2)\oplus\mc O)\rightarrow\mds P^1$, the zero-section $E_0$ of $H_2$ is the image of the section $(0,1)$ of $\mc O(2)\oplus\mc O$. Moreover, $\mc O(2)\subset\mc O(2)\oplus\mc O$ determines a smooth rational curve $\mds P(\mc O(2))$ on $H_2$, which we denoted by $E_\infty$. Let $C_1$, $C_2$, $C_3$, $C_4$ be the fibers of $h_2$ over the points $p_1$, $p_2$, $p_3$, $p_4$, respectively. First, we blow up successively the points $E_0\cap C_1$, $E_0\cap C_2$, $E_0\cap C_3$, $E_0\cap C_4$, which yields a new surface and a morphism $h_2^{(1)} : H_2^{(1)}\rightarrow\mds P^1$. The resulting exceptional divisors are denoted by $E_1^{(1)}$, $E_2^{(1)}$, $E_3^{(1)}$, $E_4^{(1)}$. Next, let $\wt C_1$, $\wt C_2$, $\wt C_3$, $\wt C_4$ denote the strict transforms of $C_1$, $C_2$, $C_3$, $C_4$. Blowing up the four points $\wt C_1\cap E_1^{(1)}$, $\wt C_2\cap E_2^{(1)}$, $\wt C_3\cap E_3^{(1)}$, $\wt C_4\cap E_4^{(1)}$, we obtain a new surface $h_2^{(2)} : H_2^{(2)}\rightarrow\mds P^1$.

\begin{proposition}\label{pro stru of ell 2} 
$\pi_2^\prime : \wt X_2\rightarrow\mds P^1$ is isomorphic to $h_2^{(2)} : H_2^{(2)}\rightarrow\mds P^1$.
\end{proposition}

\begin{proof} 
According to Lemmas \ref{lemm H fib for ell 2} and \ref{lemm E fib for ell 2}, the surface $\wt X_2$ is illustrate in Figure \ref{fig ell 2 1}. Note that $D_1,D_2,D_3,D_4$ are exceptional curves of the first kind. 
\begin{figure}[H]
	\centering
	\begin{tikzpicture}[scale=0.85]
		\coordinate (A) at  (0,0);
		\coordinate (B) at  (6,0);
		\coordinate (C) at  (0,3);
		\coordinate (D) at  (6,3);
		\coordinate (E) at  (2.2,-0.3);
		\coordinate (E1) at (0.8,1.2);
		\coordinate (F) at  (1,0.5);
		\coordinate (F1) at (1,2.5);
		\coordinate (G) at  (0.8,1.8);
		\coordinate (G1) at  (2.3,3.3);
		\draw [thick] (A)--(B)
		node[pos=0,left]{$D_0$};
		\draw [thick] (C)--(D)
		node[pos=0,left]{$D_\infty$};
		\draw [thick] (E)--(E1)
		node[pos=0,left]{$E_1$};
		\draw [thick] (F)--(F1)
		node[midway,left]{$D_1$};
		\draw [thick] (G)--(G1)
		node[pos=1,left]{$F_1$};
		\coordinate (E2) at  (3.2,-0.3);
		\coordinate (E3) at (1.8,1.2);
		\coordinate (F2) at  (2,0.5);
		\coordinate (F3) at (2,2.5);
		\coordinate (G2) at  (1.8,1.8);
		\coordinate (G3) at  (3.3,3.3);
		\draw [thick] (E2)--(E3)
		node[pos=0,left]{$E_2$};
		\draw [thick] (F2)--(F3)
		node[midway,left]{$D_2$};
		\draw [thick] (G2)--(G3)
		node[pos=1,left]{$F_2$};
		\coordinate (E4) at  (4.2,-0.3);
		\coordinate (E5) at (2.8,1.2);
		\coordinate (F4) at  (3,0.5);
		\coordinate (F5) at (3,2.5);
		\coordinate (G4) at  (2.8,1.8);
		\coordinate (G5) at  (4.3,3.3);
		\draw [thick] (E4)--(E5)
		node[pos=0,left]{$E_3$};
		\draw [thick] (F4)--(F5)
		node[midway,left]{$D_3$};
		\draw [thick] (G4)--(G5)
		node[pos=1,left]{$F_3$};
		
		\coordinate (E6) at  (5.2,-0.3);
		\coordinate (E7) at (3.8,1.2);
		\coordinate (F6) at  (4,0.5);
		\coordinate (F7) at (4,2.5);
		\coordinate (G6) at  (3.8,1.8);
		\coordinate (G7) at  (5.3,3.3);
		\draw [thick] (E6)--(E7)
		node[pos=0,left]{$E_4$};
		\draw [thick] (F6)--(F7)
		node[midway,left]{$D_4$};
		\draw [thick] (G6)--(G7)
		node[pos=1,left]{$F_4$};
	\end{tikzpicture}
	\caption{Configuration of curves on $\wt X_2$}
	\label{fig ell 2 1}
\end{figure}
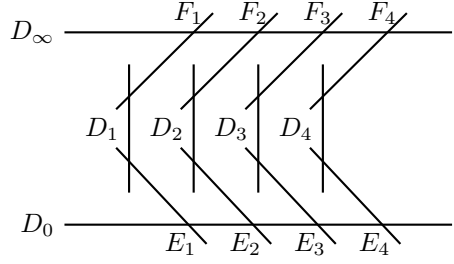
Successively blowing down $D_1$, $D_2$, $D_3$, $D_4$ yields a generically $\mds P^1$-fibration $\pi_2^{\prime(1)} : \wt X_2^{(1)}\rightarrow\mds P^1$ with singular fibers $\pi_2^{\prime(1)-1}(p_i)=E_i^{(1)}+F_i^{(1)}$, where $E_i^{(1)}$ (resp. $F_i^{(1)}$) denote the birational transforms of $E_i$ (resp. $F_i$). Both $E_i^{(1)}$ and $F_i^{(1)}$ are exceptional curves of the first kind. By further blowing down $ F_1^{(1)}$, $F_2^{(1)}$, $F_3^{(1)}$, $F_4^{(1)}$, we obtain a $\mds P^1$-fibration $\pi_2^{\prime(2)} : \wt X_2^{(2)}\rightarrow\mds P^1$, which is a Hirzebruch surface $H_n$ for some $n$. Since $D_0^2=-2$, we deduce that $\wt X_2^{(2)}\cong H_2$ (see \cite[Page.519]{gh94}).
\end{proof}

\subsection{$\wt E_6$-case}\label{subsec ell 3}
For the elliptic curve $E_{(0,1)} : y^2=x^3+1$, the automorphism $\sigma_3 : (x,y)\mapsto (e^{2\pi\sqrt{-1}/3}x,y)$ defines a $\mu_2$-action. For the orbifold curve $\mc X_3=[E_{(0,1)}/\mu_3]$, since $\sigma_3^*\omega=e^{2\pi\sqrt{-1}/3}\omega$, we have $\mds P(T^\vee\mc X_3\oplus\mc O_{\mc X_3})=[E_{(0,1)}\times\mds P^1/\mu_3]$, where the $\mu_3$-action is $(p,[z_0,z_1])\mapsto (\sigma_3(p),[z_0,e^{4\pi\sqrt{-1}/3}z_1])$. As in the $\wt D_4$-case, we have two diagrams:
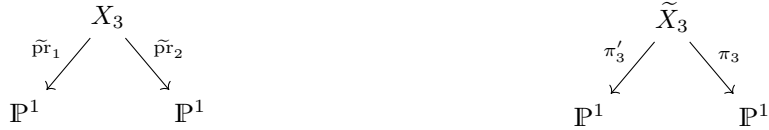
\begin{equation*}
\begin{minipage}{0.45\textwidth}
\centering
\begin{tikzcd}[row sep=2em, column sep=1em]
& X_3 \arrow[dl, "\wt\pr_1"'] \arrow[dr, "\wt\pr_2 "]&  \\
\mds P^1 & & \mds P^1
\end{tikzcd}
\end{minipage}
\hfill
\begin{minipage}{0.45\textwidth}
\centering
\begin{tikzcd}[row sep=2em, column sep=1em]
	& \wt X_3 \arrow[dl, "\pi^\prime_3"'] \arrow[dr, "\pi_3"]&  \\
	\mds P^1 & & \mds P^1
\end{tikzcd}
\end{minipage}
\end{equation*}
where $X_3=E_{(0,1)}\times\mds P^1/\mu_3$ and $\wt X_3$ is its minimal resolution.

\begin{lemma}\label{lemm H fib for ell 3}
$\pi_3^\prime : \wt X_3\rightarrow\mds P^1$ is a generically $\mds P^1$-fibration, whose singular fibers are the ones over the orbifold points $p_1$, $p_2$, $p_3$ of $\mc X_3$. More precisely, there exist smooth rational curves $\{D_i\}_{1\leq i\leq 3}$, $\{E_{ij}\}_{1\leq i\leq 2, 1\leq j\leq 3}$, $\{F_k\}_{1\leq k\leq 3}$ such that $\pi_3^{\prime-1}(p_l)=3D_l+E_{1l}+2E_{2l}+F_l$, whose dual graphs are
\begin{center}
\begin{tikzpicture}[scale=0.5]
		\node[circle, fill=black, inner sep=1.5pt] (A) at (0,0)  {};
		\node[circle, fill=black, inner sep=1.5pt] (B) at (2,0)  {};
		\node[circle, fill=black, inner sep=1.5pt] (C) at (4,0)  {};
		\node[circle, fill=black, inner sep=1.5pt] (D) at (6,0)  {};
		
		
		\draw[thick] (A)--(B)--(C)--(D);
		
		\node [below]    at (A)  {$-2$};
		\node [below]    at (B)  {$-2$};
		\node [below]    at (C)  {$-1$};
		\node [below]    at (D)  {$-3$};
		
		\node [above]    at (A)  {$E_{1l}$};
		\node [above]    at (B)  {$E_{2l}$};
		\node [above]    at (C)  {$D_l$};
	    \node [above]    at (D)   {$F_l$};
\end{tikzpicture}
\end{center} 
\end{lemma}

\begin{proof}
Following the proof of Lemma \ref{lemm H fib for ell 2}, we can show that there exists a smooth rational curve $\wt D_i$ on $X_3$ such that $\wt\pr_1^{-1}(p_i)=3\wt D_i$. The singularities of $X_3$ on $\wt D_i$ are of type $\frac{1}{3}(1,2)$ at $0$ and of type $\frac{1}{3}(1,1)$ at $\infty$. The corresponding exceptional curves are $E_{1i}\cup E_{2i}$ and $F_i$, where $E_{1i}$, $E_{2i}$ and $F_i$ are smooth rational curves satisfying $E_{1i}^2=E_{2i}^2=-2$, $E_{1i}\cdot E_{2i}=1$ and $F_i^2=-3$ (see \cite[Satz 8]{rie77}). Thus $\pi_3^{-1}(p_i)=3D_i+n_{1i}E_{1i}+n_{2i}E_{2i}+m_iF_i$ with integers $n_{1i}$, $n_{2i}$, $m_i$. Intersecting with $E_{1i}$, $E_{2i}$, $F_i$ yields $3D_i\cdot E_{1i}-2n_{1i}+n_{2i}=0$, $3D_i\cdot E_{2i}-2n_{2i}+n_{1i}=0$ and $3D_i\cdot F_i-3m_i=0$. Hence, $9D_i^2+n_{1i}(2n_{1i}-n_{2i})+n_{2i}(2n_{2i}-n_{1i})+3m_i^2=0$, so $D_i^2<0$. By the adjunction formula, $K_{\wt X_3}\cdot(3D_i+n_{1i}E_{1i}+n_{2i}E_{2i}+m_iF_i)=-2$, which implies $3K_{\wt X_3}\cdot D_i=-2-m_i$. Thus, $D_i$ is an exceptional curve of the first kind, with $D_i\cdot D_i=K_{\wt X_3}\cdot D_i=-1$ and  $m_i=1$. Finally, $n_{1i}^2+n_{2i}^2-n_{1i}n_{2i}=3$, so $(n_{1i},n_{2i})$ is $(1,2)$ or $(2,1)$. Up to order, we take $n_{1i}=1$ and $n_{2i}=2$.
\end{proof}

\begin{lemma}\label{lemm E fib for ell 3}
The elliptic fibration $\pi_3 : \wt X_3\rightarrow\mds P^1$ has only two singular fibers, which are the ones over $0$ and $\infty$. Specifically, there exist smooth rational curves $D_0$, $D_\infty $ such that
\begin{enumerate}[$(1)$]
\item $\pi_3^{-1}(0)=3D_0+2E_{11}+E_{21}+2E_{12}+E_{22}+2E_{13}+E_{23}$ which is of type $\uppercase\expandafter{\romannumeral 4}^*$ $(\wt E_6)$.
\item $\pi_3^{-1}(\infty)=3D_\infty+F_1+F_2+F_3$ whose dual graph is

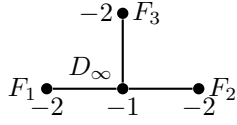
\begin{center}
\begin{tikzpicture}[scale=0.5]
\node[circle, fill=black, inner sep=1.5pt] (A) at (0,2)    {};
\node[circle, fill=black, inner sep=1.5pt] (B) at (-2,2)   {};
\node[circle, fill=black, inner sep=1.5pt] (C) at (2,2)    {};
\node[circle, fill=black, inner sep=1.5pt] (D) at (0, 4)   {};
		
\draw [thick] (A)--(B);
\draw [thick](C)--(A);
\draw [thick](A)--(D);
		
\node [below]   at (A)  {$-1$};
\node [below]   at (B)  {$-2$};
\node [below]   at (C)  {$-2$};
\node [left]    at (D)  {$-2$};

\node [above left]  at (A)  {$D_\infty$};
\node [left]        at (B)  {$F_1$};
\node [right]       at (C)  {$F_2$};
\node [right]       at (D)  {$F_3$};
\end{tikzpicture}
\end{center}
\end{enumerate} 

Furthermore, by blowing down $D_\infty$, we obtain a relatively minimal elliptic surface, the singular fiber of which at $\infty$ is of type $\uppercase\expandafter{\romannumeral 4}$ $(\wt A_2)$.
\end{lemma}

\begin{proof}
As before, there exists a smooth rational curve $\wt D_0\subset X_3$ containing three singular points of type $\frac{1}{3}(1,2)$ with $\wt\pr_2^{-1}(0)=3\wt D_0$. Thus $\pi_3^{-1}(0) = 3D_0 + \sum_{1\leq i\leq 2,\,1\leq j\leq 3} \wt n_{ij} E_{ij}$, where $D_0$ is the strict transform of $\wt D_0$, and $\{\wt n_{ij}\}_{1\leq i\leq 2,\,1\leq j\leq 3}$ is a set of natural numbers. Intersections give $9 D_0^2 + 2\sum_{i,j} \wt n_{ij}^2 - 2 \sum_{k=1}^{3} \wt n_{1k}\wt n_{2k} = 0$, and so $D_0^2 < 0$. By the adjunction formula, $K_{\wt X_3} \cdot \bigl(3 D_0 + \sum_{i,j} \wt n_{ij} E_{ij}\bigr) = 0$, so $K_{\wt X_3}\cdot D_0 = 0$, $D_0^2 = -2$ and $\sum_{i,j}\wt n_{ij}^2-\sum_{k=1}^{3} \wt n_{1k}\wt n_{2k} = 9$. Following Lemma 2.12 (ii) in \cite{bhpv04}, the intersection matrix of $\{D_0, E_{ij}\}$ with $\{\wt n_{kl}\}$ corresponds to the affine Dynkin diagram $\wt E_6$. The projection formula gives $\wt D_l \cdot \wt D_0 = 1/3$, hence $(3D_l + E_{1l} + 2E_{2l} + F_l)\cdot D_0 = 1$, so $D_0 \cdot E_{1l} = 1$ and $D_0 \cdot E_{2l} = 0$. This completes the proof of case (1) by Zariski's Lemma (see ibid). Analogously, let $\wt D_\infty\subset X_2$ be the smooth rational curve through three singular points of type $\frac{1}{3}(1,1)$ satisfying $\wt\pr_2^{-1}(\infty) = 3\wt D_\infty$. Then $\pi_2^{-1}(\infty) = 3 D_\infty + \wt m_1 F_1 + \wt m_2 F_2 + \wt m_3 F_3$ with $\wt m_i = D_\infty \cdot F_i$, which gives $3 D_\infty^2 + \wt m_1^2 + \wt m_2^2 + \wt m_3^2 = 0$ and $K_{\wt X_3} \cdot D_\infty < 0$. Hence, $D_\infty$ is an exceptional curve of the first kind with 
$\wt m_1 = \wt m_2 = \wt m_3 = 1$, as illustrated in Figure~\ref{fig ell 3 1}.

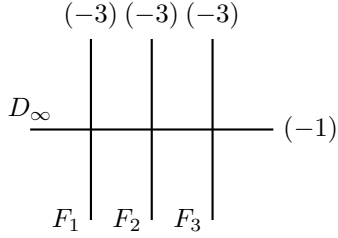
\begin{figure}[H]
\centering
\begin{tikzpicture}[scale=0.8]
\coordinate (A) at (0,0);
\coordinate (B) at (4,0);
\draw[thick] (A)--(B) 
node[pos=0,above]{$D_\infty$} 
node[pos=1,right]{$(-1)$};
\coordinate (D) at (1,1.5);
\coordinate (E) at (1,-1.5);
\draw[thick] (D)--(E) 
node[left]{$F_1$}
node[pos=0,above]{$(-3)$};
\coordinate (F) at (2,1.5);
\coordinate (G) at (2,-1.5);
\draw[thick] (F)--(G) 
node[left]{$F_2$}
node[pos=0,above]{$(-3)$};
\coordinate (H) at (3,1.5);
\coordinate (I) at (3,-1.5);
\draw[thick] (H)--(I) 
node[left]{$F_3$}
node[pos=0,above]{$(-3)$};

\end{tikzpicture}
\caption{Configuration before blowing down $D_\infty$}
\label{fig ell 3 1}
\end{figure}

Blowing down $D_\infty$ yields a relatively minimal elliptic surface with singular fiber of type $\uppercase\expandafter{\romannumeral 4}$ over $\infty$ (see Figure \ref{fig ell 3 2}).

\begin{figure}[H]
\centering
\begin{tikzpicture}[scale=0.7]
\coordinate (A) at (-1,1);
\coordinate (B) at (1,-1);
\coordinate (C) at (1,1);
\coordinate (D) at (-1,-1);
\coordinate (E) at (-1.5,0);
\coordinate (F) at  (1.5,0);
\draw [thick] (A)--(B)
node[pos=0,left]{$\wt F_1$}
node[pos=1,right]{$(-2)$};
\draw [thick] (E)--(F)
node[pos=0,left]{$\wt F_2$}
node[pos=1,right]{$(-2)$};
\draw [thick] (D)--(C)
node[pos=0,left]{$\wt F_3$}
node[pos=1,right]{$(-2)$};
\end{tikzpicture}
\caption{Fiber of type IV after blowing down $D_\infty$}
\label{fig ell 3 2}
\end{figure}
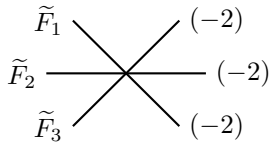

\end{proof}

Let $C_1,C_2,C_3$ be the fibers of $h_2:H_2\to\mathds P^1$ over the orbifold points $p_1,p_2,p_3$. First, blowing up $E_0\cap C_i$ $(i=1,2,3)$ yields $h_2^{(1)}:H_2^{(1)}\to\mathds P^1$ with exceptional curves $E_i^{(1)}$. Next, blowing up the three points $\wt C_i\cap E_i^{(1)}$, where $\wt C_i$ are the strict transforms of $C_i$, gives $h_2^{(2)}:H_2^{(2)}\to\mathds P^1$. Finally, blowing up three points $\wt E_i^{(1)}\cap E_i^{(2)}$, where $\wt E_i^{(1)}$ are the strict transforms of $E_i^{(1)}$ and $E_i^{(2)}$ the new exceptional curves, gives $h_2^{(3)}:H_2^{(3)}\to\mathds P^1$.

\begin{proposition}\label{prop stru for ell 3}
$\pi_3^\prime : \wt X_3\rightarrow\mds P^1$ is isomorphic to $h_2^{(3)} : H_2^{(3)}\rightarrow\mds P^1$.
\end{proposition}

\begin{proof}
By Lemma \ref{lemm H fib for ell 3} and \ref{lemm E fib for ell 3}, $\wt X_4$ is illustrated in Figure \ref{fig ell 3 3}.	

\begin{figure}[H]
\centering
\begin{tikzpicture}[xscale=0.8,yscale=0.5]
\coordinate (A) at (0,0);
\coordinate (B) at (6,0);
\coordinate (C) at (0,6);
\coordinate (D) at (6,6);

\draw [thick] (A)--(B)
node[pos=0,left]{$D_0$};
\draw [thick] (C)--(D)
node[pos=0,left]{$D_\infty$};

\coordinate (A2) at (1,-0.5);
\coordinate (A1) at (2,1.5);
\coordinate (B1) at (2,1);
\coordinate (B2) at (1,3);
\coordinate (C1) at (1,2.5);
\coordinate (C2) at (2,4.5);
\coordinate (D1) at (2,4);
\coordinate (D2) at (1,6.5);

\draw [thick] (A2)--(A1)
node[pos=0,left]{$E_{11}$};
\draw [thick] (B1)--(B2)
node[pos=0,right]{$E_{21}$};
\draw [thick] (C1)--(C2)
node[pos=0,left]{$D_1$};
\draw [thick] (D1)--(D2)
node[pos=1,left]{$F_1$};
\coordinate (A3) at (2.5,-0.5);
\coordinate (A4) at (3.5,1.5);
\coordinate (B3) at (3.5,1);
\coordinate (B4) at (2.5,3);
\coordinate (C3) at (2.5,2.5);
\coordinate (C4) at (3.5,4.5);
\coordinate (D3) at (3.5,4);
\coordinate (D4) at (2.5,6.5);

\draw [thick] (A3)--(A4)
node[pos=0,left]{$E_{12}$};
\draw [thick] (B3)--(B4)
node[pos=0,right]{$E_{22}$};
\draw [thick] (C3)--(C4)
node[pos=0,left]{$D_2$};
\draw [thick] (D3)--(D4)
node[pos=1,left]{$F_2$};


\coordinate (A5) at (4,-0.5);
\coordinate (A6) at (5,1.5);
\coordinate (B5) at (5,1);
\coordinate (B6) at (4,3);
\coordinate (C5) at (4,2.5);
\coordinate (C6) at (5,4.5);
\coordinate (D5) at (5,4);
\coordinate (D6) at (4,6.5);

\draw [thick] (A5)--(A6)
node[pos=0,left]{$E_{13}$};
\draw [thick] (B5)--(B6)
node[pos=0,right]{$E_{23}$};
\draw [thick] (C5)--(C6)
node[pos=0,left]{$D_3$};
\draw [thick] (D5)--(D6)
node[pos=1,left]{$F_3$};
\end{tikzpicture}
\caption{Configuration of curves on $\wt X_3$}
\label{fig ell 3 3}
\end{figure}

Blowing down $D_1$, $D_2$, $D_3$ successively gives $\pi_3^{\prime(1)} : \wt X_3^{(1)}\rightarrow\mds P^1$ with three singular fibers $\pi_3^{\prime(1)-1}(p_i)=E_{1i}^{(1)}+2E_{2i}^{(1)}+F_i^{(1)}$ (see dual graphs below)
\begin{equation*}
\begin{tikzpicture}[scale=0.5]
\node[circle, fill=black, inner sep=1.5pt] (A) at (0,0) {};
\node[circle, fill=black, inner sep=1.5pt] (B) at (2,0) {};
\node[circle, fill=black, inner sep=1.5pt] (C) at (4,0) {};	

\draw[thick] (A)--(B)--(C);

\node [below] at (A) {$-2$};
\node [below] at (B) {$-1$};
\node [below] at (C) {$-2$};

\node [above] at (A) {$E_{1i}^{(1)}$};
\node [above] at (B) {$E_{2i}^{(1)}$};
\node [above] at (C) {$F_i^{(1)}$};
\end{tikzpicture}
\end{equation*}
where $E_{1i}^{(1)}$, $E_{2i}^{(1)}$ and $F_i^{(1)}$ are the birational transformations of $E_{1i}$, $E_{2i}$ and $F_i$ respectively. Blowing down all $E_{2i}^{(1)}$ gives $\pi_3^{\prime(2)} : \wt X_3^{(2)}\rightarrow\mds P^1$ with singular fibers $\pi_3^{\prime(2)-1}(p_i)=E_{1i}^{(2)}+F_i^{(2)}$, where $E_{1i}^{(2)}$ and $F_i^{(2)}$ are the birational transformations of $E_{1i}^{(1)}$ and $F_i^{(1)}$ respectively. $E_{1i}^{(2)}$ and $F_i^{(2)}$ are exceptional curves of first kind. Blowing down all $F_{i}^{(2)}$ produces a $\mds P^1$-fibration $\pi_3^{\prime(3)} : \wt X_3^{(3)}\rightarrow\mds P^1$, which is a Hirzebruch surface $H_n$ for some $n$. Since $D_0^2=-2$, we conclude $n=2$. Reversing this process recovers $\wt X_3$, which coincides with $H_2^{(3)}$.
\end{proof}

=

\subsection{$\wt E_7$-case}

The elliptic curve $E_{(1,0)} : y^2 = x^3 + x$ admits a $\mu_4$-action defined by $\sigma_4 : E_{(1,0)} \to E_{(1,0)}$, $(x,y) \mapsto (-x,\sqrt{-1}\,y)$. The quotient stack $\mc X_4 = [E_{(1,0)}/\mu_4]$ has three orbifold points $p_1$, $p_2$, and $p_3$ with stabilizer groups $\mu_4$, $\mu_4$, and $\mu_2$, respectively.  
Since $\sigma_4^*\omega = \sqrt{-1}\,\omega$, we obtain $\mds P(T^\vee\mc X_4 \oplus \mc O_{\mc X_4}) = [E_{(1,0)} \times \mds P^1 / \mu_4]$, where the $\mu_4$-action is given by
\[
\sigma_4 : E_{(1,0)} \times \mds P^1 \to E_{(1,0)} \times \mds P^1, \quad 
(p,[z_0,z_1]) \mapsto (\sigma_4(p),[z_0,-\sqrt{-1}\,z_1]).
\]
In addition, the orbifold surface $[E_{(0,1)} \times \mds P^1 / \mu_4]$ has six orbifold points; equivalently, the coarse space $X_4 = (E_{(0,1)} \times \mds P^1)/\mu_4$ has exactly six singular points. Analogously, we obtain the following two diagrams:
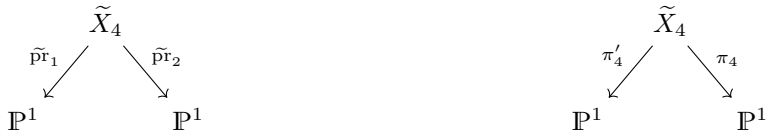
\[
\begin{minipage}{0.45\textwidth}
	\centering
	\begin{tikzcd}[row sep=2em, column sep=1em]
		& \wt X_4 \arrow[dl,"\wt\pr_1"'] \arrow[dr,"\wt\pr_2"] & \\
		\mds P^1 && \mds P^1
	\end{tikzcd}
\end{minipage}
\hfill
\begin{minipage}{0.45\textwidth}
	\centering
	\begin{tikzcd}[row sep=2em, column sep=1em]
		& \wt X_4 \arrow[dl,"\pi_4^\prime"'] \arrow[dr,"\pi_4"] & \\
		\mds P^1 && \mds P^1
	\end{tikzcd}
\end{minipage}
\]
where $\wt X_4$ denotes the minimal resolution of $X_4$.

\begin{lemma}\label{lemm ell 4 H}
$\pi_4^\prime : \wt X_4\rightarrow\mds P^1$ is a generically $\mds P^1$-bundle whose singular fibers are exactly the ones over the orbifold points of $\mc X_4$. More precisely, there exist smooth rational curves $\{D_i\}_{1\leq i\leq 3}$,  $\{E_{ij}\}_{1\leq i\leq 3, 1\leq j\leq 2}$, $\{F_k\}_{1\leq k\leq 2}$, $\{E,F\}$ such that 
\begin{enumerate}[$(1)$]
\item for $l=1,2$, we have $\pi_4^{\prime-1}(p_l)=4D_l+E_{1l}+2E_{2l}+3E_{3l}+F_l$, whose dual graphs are 
\begin{center}
\begin{tikzpicture}[scale=0.5]
	\node[circle, fill=black, inner sep=1.5pt] (A) at (0,0)  {};
	\node[circle, fill=black, inner sep=1.5pt] (B) at (2,0)  {};
	\node[circle, fill=black, inner sep=1.5pt] (C) at (4,0)  {};
	\node[circle, fill=black, inner sep=1.5pt] (D) at (6,0)  {};
	\node[circle, fill=black, inner sep=1.5pt] (E) at (8,0)  {};
		
		
	\draw [thick] (A)--(B)--(C)--(D)--(E);
		
	\node [below]   at (A)  {$-2$};
	\node [below]    at (B)  {$-2$};
	\node [below]   at (C)  {$-2$};
	\node [below]   at  (D)  {$-1$};
	\node [below]    at  (E) {$-4$};
		
	\node [above]   at (A)  {$E_{1l}$};
	\node [above]    at (B)  {$E_{2l}$};
	\node [above]   at (C)  {$E_{3l}$};
	\node [above]   at  (D)  {$D_l$};
	\node [above]    at  (E) {$F_l$};
\end{tikzpicture}
\end{center}

\item $\pi_4^{\prime-1}(p_3)=2D_3+E+F$ with dual graph
\begin{center}
	\begin{tikzpicture}[scale=0.5]
		\node[circle, fill=black, inner sep=1.5pt] (A) at (0,0)  {};
		\node[circle, fill=black, inner sep=1.5pt] (B) at (2,0)  {};
		\node[circle, fill=black, inner sep=1.5pt] (C) at (4,0)  {};
		
		
		\draw [thick] (A)--(B)--(C);
		
		\node [below]    at (A)  {$-2$};
		\node [below]    at (B)  {$-1$};
		\node [below]    at (C)  {$-2$};
		\node [above]    at (A)  {$E$};
		\node [above]    at (B)  {$D_3$};
        \node [above]    at (C)  {$F$};
\end{tikzpicture}
\end{center}
\end{enumerate}
\end{lemma}

\begin{proof}
	
It suffices to prove case (1) when $l=1$. As before, there exists a smooth rational curve $\wt D_1$ on $X_4$ such that $\wt\pr_1^{-1}(p_1) = 4\wt D_1$, on which lie two singular points of $X_4$ of types $\tfrac{1}{4}(1,3)$ and $\tfrac{1}{4}(1,1)$. The corresponding exceptional curves are denoted by $E_{11}\cup E_{21}\cup E_{31}$ and $F_1$, respectively, and they satisfy $E_{11}^2=E_{21}^2=E_{31}^2=-2$, $E_{11}\cdot E_{21}=E_{21}\cdot E_{31}=1$, $E_{11}\cdot E_{31}=0$ and $F_1^2=-4$. Let $D_1$ be the strict transform of $\wt D_1$ on $\wt X_4$. We may write $\pi_4^{\prime -1}(p_1) = 4D_1 + n_1E_{11} + n_2E_{21} + n_3E_{31} + mF_1$ for some natural numbers $n_1,n_2,n_3,m$. Intersecting with itself yields $16D_1^2+(2n_1^2+2n_2^2+2n_3^2-2n_1n_2-2n_2n_3)+4m^2 = 0$, so $D_1^2<0$. By the adjunction formula, we obtain
$(4D_1+n_1E_{11}+n_2E_{21}+n_3E_{31}+mF_1)\cdot K_{\wt X_4} = -2$, which implies $2D_1\cdot K_{\wt X_4} = -1-m$. Hence $D_1$ is an exceptional curve and necessarily $m=1$. Up to reordering, we find $n_1=1$, $n_2=2$ and $n_3=3$.

\end{proof}

\begin{lemma}\label{lemm ell 4 E}
The elliptic fibration $\pi_4 : \wt X_4\rightarrow\mds P^1$ has exactly two singular fibers:
\begin{enumerate}[$(1)$]
\item $\pi_4^{-1}(0)=4D_0+3E_{11}+2E_{21}+E_{31}+3E_{12}+2E_{22}+E_{32}+2E$ is of type $\uppercase\expandafter{\romannumeral 3}^*$ $(\wt E_7)$.
\item $\pi_4^{-1}(\infty)=4D_\infty+F_1+F_2+2F$ with dual graph
\begin{center}
\begin{tikzpicture}[scale=0.5]
		\node[circle, fill=black, inner sep=1.5pt] (A) at (0,0)  {};
		\node[circle, fill=black, inner sep=1.5pt] (B) at (3,0)  {};
		\node[circle, fill=black, inner sep=1.5pt] (C) at (6,0)  {};
		\node[circle, fill=black, inner sep=1.5pt] (D) at (3,2)  {};
		
		\draw [thick] (A)--(B)--(C);
		\draw [thick] (B)--(D);
		\node [below]    at (A)  {$-4$};
		\node [below]    at (B)  {$-1$};
		\node [below]    at (C)  {$-4$};
		\node [right]     at (D)   {$-2$};

		\node [above]    at (A)  {$F_1$};
		\node [above left]    at (B)  {$D_\infty$};
		\node [above]    at (C)  {$F_2$};
		\node [left]  at (D)   {$F$};
		
\end{tikzpicture}
\end{center} 
\end{enumerate}
After two successive blow-downs of the exceptional curves over $\infty$, we obtain a relatively minimal elliptic surface whose singular fiber over $\infty$ is of type $\uppercase\expandafter{\romannumeral 3}$ $(\wt A_1)$.
\end{lemma}

\begin{proof}
As before, we can show there exists a smooth rational curve $\wt D_0$ on $X_4$ such that $\wt\pr_2^{-1}(0)=4\wt D_0$, on which there are three singular points: two of type $\frac{1}{4}(1,3)$ and one of type $\frac{1}{2}(1,1)$. Hence we obtain
\begin{equation}\label{equ ell 4 div 0}
\textstyle{\pi_4^{-1}(0)=4D_0+\sum_{1\leq j\leq 2}\sum_{1\leq i\leq 3}n_{ij}E_{ij}+m_0E},
\end{equation}
where $D_0$ is the strict transformation of $\wt D_0$. From the relation $4D_0^2+\sum_{1\leq j\leq 2}\sum_{1\leq i\leq 3}n_{ij}E_{ij}\cdot D_0+m_0E\cdot D_0=0$ together with the adjunction formula, we deduce that $D_0^2=-2$. For the real vector space with basis ${D_0,E_{ij},E}_{1\leq i\leq 3,1\leq j\leq 2}$, the intersection form defines a quadratic form whose annihilator is one-dimensional with basis given by \eqref{equ ell 4 div 0} (see Zariski’s Lemma in \cite{bhpv04}). This quadratic form with annihilator \eqref{equ ell 4 div 0} corresponds to the affine Dynkin diagram $\wt E_7$ (see Lemma 2.12 in ibid.). In analogy with the proof of Lemma \ref{lemm E fib for ell 2}, we obtain $\wt D_l\cdot \wt D_0=1/4$, and hence by the projection formula, $(4D_l+E_{1l}+2E_{2l}+3E_{3l}+F_l)\cdot D_0=1$. Therefore,
\begin{equation*}
D_0\cdot E_{il}=
\begin{cases}
1 & \text{if } i=1 \\
0 & \text{otherwise}.
\end{cases}
\end{equation*}
This completes the proof of case (1). By Kodaira’s classification of singular elliptic fibers (see ibid., p.~201), the singular fiber over $0$ is of type $\uppercase\expandafter{\romannumeral 3}^*(\wt E_7)$. In a way analogous to case (2) of Lemma \ref{lemm E fib for ell 2}, we assume that $\pi_4^{-1}(\infty)= 4D_\infty+m_1F_1+m_2F_2+m_3F$ for some natural numbers $m_1$, $m_2$, $m_3$ (see Figure \ref{fig ell 4 1}), from which we obtain $8D_\infty^2+2m_1^2+2m_2^2+m_3^2=0$.

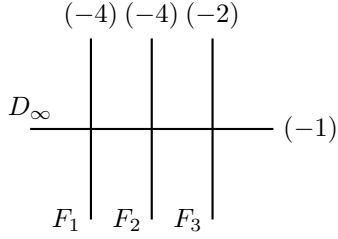
\begin{figure}[H]
	\centering
	\begin{tikzpicture}[scale=0.8]
		\coordinate (A) at (0,0);
		\coordinate (B) at (4,0);
		\draw [thick](A)--(B) 
		node[pos=0,above]{$D_\infty$} 
		node[pos=1,right]{$(-1)$};
		\coordinate (D) at (1,1.5);
		\coordinate (E) at (1,-1.5);
		\draw [thick](D)--(E) 
		node[left]{$F_1$}
		node[pos=0,above]{$(-4)$};
		\coordinate (F) at (2,1.5);
		\coordinate (G) at (2,-1.5);
		\draw [thick](F)--(G) 
		node[left]{$F_2$}
		node[pos=0,above]{$(-4)$};
		\coordinate (H) at (3,1.5);
		\coordinate (I) at (3,-1.5);
		\draw [thick](H)--(I) 
		node[left]{$F_3$}
		node[pos=0,above]{$(-2)$};
	\end{tikzpicture}
	\caption{Configuration of the fiber $\pi_4^{-1}(\infty)$ before contractions.}
	\label{fig ell 4 1}
\end{figure}

In addition, by the adjunction formula we obtain $4K_{\wt X_4}\cdot D_\infty=-2m_1-2m_2$. Hence, $D_\infty^2<0$ and $K_{\wt X_4}\cdot D_\infty<0$, which implies that $D_\infty$ is an exceptional curve of the first kind. A direct computation shows that $m_1=m_2=1$ and $m_3=2$. By blowing down $D_\infty$, we get $\pi_4^{(1)} : \wt X_4^{(1)}\rightarrow\mds P^1$ whose singular fiber at $\infty$ is $\pi_4^{(1)-1}(\infty)=F_1^{(1)}+F_2^{(1)}+2F^{(1)}$, where $F_1^{(1)}$, $F_2^{(1)}$ and $F^{(1)}$ are birational transformations of $F_1$, $F_2$ and $F$ respectively (see Figure \ref{fig ell 4 2}).

\begin{figure}[H]
	\centering
	\begin{tikzpicture}[scale=0.7]
		\coordinate (A) at (-1,1);
		\coordinate (B) at (1,-1);
		\coordinate (C) at (-1,0);
		\coordinate (D) at (1,0);
		\coordinate (E) at (-1,-1);
		\coordinate (F) at (1,1);
		\draw[thick] (A)--(B) node[pos=0,left]{$F_1^{(1)}$} node[pos=1,right]{$(-3)$};
		\draw[thick] (C)--(D) node[pos=0,left]{$F_2^{(1)}$} node[pos=1,right]{$(-3)$};
		\draw[thick] (E)--(F) node[pos=0,left]{$F^{(1)}$} node[pos=1,right]{$(-1)$};
	\end{tikzpicture}
	\caption{Singular fiber $\pi_4^{(1)-1}(\infty)$ after blowing down $D_\infty$.}
	\label{fig ell 4 2}
\end{figure}
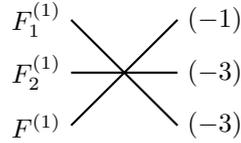

Note that $F_1^{(1)}$, $F_2^{(1)}$ and $F^{(1)}$ intersect at a single point, and that $F^{(1)}$ is an exceptional curve of the first kind. By blowing down $F^{(1)}$, we get a relatively minimal elliptic surface $\pi_4^{(2)} : \wt X_4^{(2)}\rightarrow\mds P^1$ whose singular fiber over $\infty$ is $\pi_4^{(2)-1}(\infty)=F_1^{(2)}+F_2^{(2)}$. Here $F_1^{(2)}$ and $F_2^{(2)}$ denote the birational transformations of $F_1^{(1)}$ and $F_2^{(1)}$ respectively,  satisfying $F_1^{(2)2}=F_2^{(2)2}=-2$ and $F_1^{(2)}\cdot F_2^{(2)}=2$ (see Figure \ref{fig ell 4 3}). The type of $\pi_4^{(2)-1}(\infty)$ is $\uppercase\expandafter{\romannumeral 3}$ and the intersection matrix is given by the affine Dynkin diagram $\wt A_1$.

\begin{figure}[H]
	\centering
	\begin{tikzpicture}[scale=0.7]
		\coordinate (A) at (0,1);
		\coordinate (B) at (0,-1);
		\draw[thick] (A) arc (0:-180:1) node[pos=1,left]{$F_1^{(2)}$} node[pos=0,right]{$(-2)$};
		\draw[thick] (B) arc (0:180:1) node[pos=1,left]{$F_2^{(2)}$} node[pos=0,right]{$(-2)$};
	\end{tikzpicture}
	\caption{Singular fiber $\pi_4^{(2)-1}(\infty)$ after blowing down $F^{(1)}$.}
	\label{fig ell 4 3}
\end{figure}
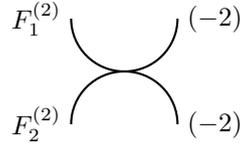

\end{proof}

In analogy with Proposition \ref{prop stru for ell 3}, blowing up $E_0\cap C_i$ ($i=1,2,3$) yields $h_2^{(1)}:H_2^{(1)}\rightarrow\mds P^1$ with exceptional divisors $E_i^{(1)}$. Blowing up the three points $E_i^{(1)}\cap \wt C_i$, where $\wt C_i$ denote the strict transformations of $C_i$. This yields $h^{(2)}_2 : H_2^{(2)}\rightarrow\mds P^1$ with exceptional divisors $E_i^{(2)}$. Next, blowing up $E_j^{(2)}\cap \wt E_j^{(1)}$ ($j=1,2$), where $\wt E_j^{(1)}$ are the strict transforms of $E_j^{(1)}$. This produces $h_2^{(3)} : H_2^{(3)}\rightarrow\mds P^1$ with new exceptional divisors $E_1^{(3)}$ and $E_2^{(3)}$. Finally, blowing up $E_j^{(3)}\cap\overline E_j^{(1)}$ ($j=1,2$), where $\overline E_j^{(1)}$ are the strict transforms of $\wt E_j^{(1)}$, gives $h_2^{(4)}:H_2^{(4)}\to\mathds P^1$.

\begin{proposition}\label{pro ell 4 str}
$\pi_4^\prime : \wt X_4\rightarrow\mds P^1$ is isomorphic to $h_2^{(4)} : H_2^{(4)}\rightarrow\mds P^1$. 
\end{proposition}

\begin{proof}
Based on Lemmas \ref{lemm ell 4 H} and \ref{lemm ell 4 E}, the surface $\wt X_4$ is illustrated in Figure \ref{fig ell 4 4}. Beginning with $D_1$, $D_2$ and $D_3$, We may carry out a sequence of blow-downs, and then obtain a $\mds P^1$-fibration which is a Hirzebruch surface $H_n$ for some $n$. Since $D_0$ remains unchanged throughout this process, it follows from $D_0^2=-2$ that the resulting Hirzebruch surface is the second one. Reversing this process, we can recover the original surface $\wt X_4$. Then $\wt X_4$ is precisely $H_2^{(4)} $, as deduced from the construction of $H_2^{(4)}$. The proof is complete.
\begin{figure}[H]
\centering
\begin{tikzpicture}[xscale=0.8, yscale=0.8]
		\coordinate (A) at  (0,0);
		\coordinate (B) at  (8,0);
		\coordinate (C)  at (0,5);
		\coordinate (D) at  (8,5);
		
		\draw[thick] (A)--(B)
		node [pos=0,left]{$D_0$};
		\draw[thick] (C)--(D)
		node [pos=0,left]{$D_\infty$};
		
		\coordinate (A1) at  (0.5,-0.5);
		\coordinate (A2) at  (2.5,1.5);
		\coordinate (C2) at  (2.5,0.5);
		\coordinate (C1) at  (0.5,2.5);
		\coordinate (D1) at  (2.5,3.5);
		\coordinate (D2) at  (0.5,1.5);
		\coordinate (E2) at  (2.5,2.5);
		\coordinate (E1) at  (0.5,4.5);
		\coordinate (F2) at  (0.5,3.5);
		\coordinate (F1) at  (2.5,5.5);
		
		\draw[thick]  (A1)--(A2)
		node [pos=0,left] {$E_{11}$};
		\draw[thick]  (C2)--(C1)
		node [pos=0,left] {$E_{21}$};
		\draw[thick]  (D2)--(D1)
		node [pos=1,left] {$E_{31}$};
		\draw[thick]  (E2)--(E1)
		node [pos=1,left] {$D_1$};
		\draw[thick]  (F2)--(F1)
		node [pos=1,left] {$F_1$};
		
		\coordinate (A3) at  (3,-0.5);
		\coordinate (A4) at  (5,1.5);
		\coordinate (C4) at  (5,0.5);
		\coordinate (C3) at  (3,2.5);
		\coordinate (D3) at  (5,3.5);
		\coordinate (D4) at  (3,1.5);
		\coordinate (E4) at  (5,2.5);
		\coordinate (E3) at  (3,4.5);
		\coordinate (F4) at  (3,3.5) ;
		\coordinate (F3) at  (5,5.5);
		\draw[thick]  (A3)--(A4)
		node [pos=0,left] {$E_{12}$};
		\draw[thick] (C4)--(C3)
		node [pos=0,left] {$E_{22}$};
		\draw[thick]  (D4)--(D3)
		node [pos=1,left] {$E_{32}$};
		\draw[thick]  (E4)--(E3)
		node [pos=1,left] {$D_2$};
		\draw[thick]  (F4)--(F3)
		node [pos=1,left] {$F_2$};
		
		\coordinate (G)  at  (5.5,-0.5);
		\coordinate (G1) at (7.5,1.5);
		\coordinate (H)  at  (7,0.5);
		\coordinate (I)  at  (7,4);
		\coordinate (J1) at (7.5,2.5);
		\coordinate (J)  at  (5.5,5.5);
		
		\draw[thick] (G)--(G1)
		node[pos=0,left]{$E$};
		\draw[thick] (H)--(I)
		node[pos=1,right]{$D_3$};
		\draw[thick] (J1)--(J)
		node[pos=1,right]{$F$};
\end{tikzpicture}
\caption{Configuration of curves on $\wt X_4$.}
\label{fig ell 4 4}
\end{figure}
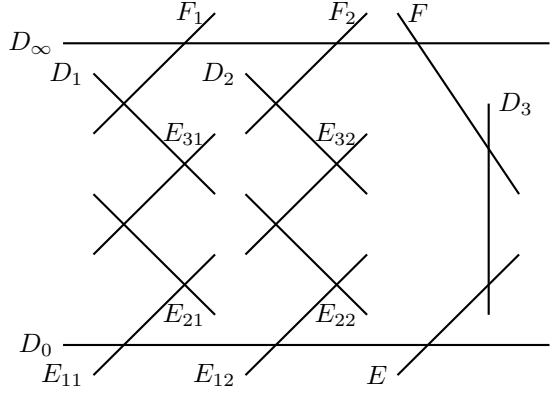
\end{proof}

\subsection{$\wt E_8$-case}

The morphism $\sigma_6=\sigma_3\circ\tau : E_{(0,1)} \to E_{(0,1)}$ defines a $\mu_6$-action on $E_{(0,1)}$.  
The quotient stack $\mc X_6 = [E_{(0,1)}/\mu_6]$ has three orbifold points $p_1$, $p_2$, and $p_3$ with stabilizer groups $\mu_6$, $\mu_3$, and $\mu_2$, respectively. Moreover, $\mds P(T^\vee\mc X_6 \oplus \mc O_{\mc X_6}) = [E_{(0,1)} \times \mds P^1 / \mu_6]$, whose coarse moduli space is $X_6 = (E_{(0,1)} \times \mds P^1)/\mu_6$.  
As before, we obtain two diagrams:
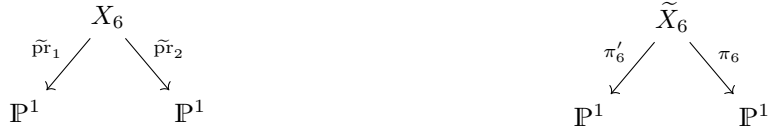
\begin{equation*}
	\begin{minipage}{0.45\textwidth}
		\centering
		\begin{tikzcd}[row sep=2em,column sep=1em]
			& X_6 \arrow[dl,"\wt\pr_1"'] \arrow[dr,"\wt\pr_2"] & \\
			\mds P^1 && \mds P^1
		\end{tikzcd}
	\end{minipage}
	\hfill
	\begin{minipage}{0.45\textwidth}
		\centering
		\begin{tikzcd}[row sep=2em,column sep=1em]
			& \wt X_6 \arrow[dl,"\pi_6^\prime"'] \arrow[dr,"\pi_6"] & \\
			\mds P^1 && \mds P^1
		\end{tikzcd}
	\end{minipage}
\end{equation*}
where $\wt X_6$ denotes the minimal resolution of $X_6$.

\begin{lemma}\label{lemm ell 6 H}
The singular fibers of the generically $\mds P^1$-fibration $\pi_6^\prime : \wt X_6\rightarrow\mds P^1$ are the ones over $p_1$, $p_2$, $p_3$. More specifically, there exist smooth rational curves $\{D_i\}_{1\leq i\leq 3}$, $\{E_j\}_{1\leq j\leq 9}$, $\{F_k\}_{1\leq k\leq 2}$ such that
\begin{enumerate}[$(1)$]
\item $\pi_6^{\prime-1}(p_1)=6D_1+E_1+2E_2+3E_3+4E_4+5E_5+F_1$ with dual graph

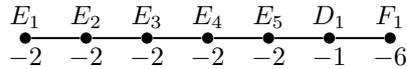
\begin{center}
\begin{tikzpicture}[scale=0.8]
\node[circle, fill=black, inner sep=1.5pt] (A) at (-2,0) {};
\node[circle, fill=black, inner sep=1.5pt] (B) at (-1,0) {};
\node[circle, fill=black, inner sep=1.5pt] (C) at (0,0) {};
\node[circle, fill=black, inner sep=1.5pt] (D) at (1,0) {};
\node[circle, fill=black, inner sep=1.5pt] (G) at (2,0) {};
\node[circle, fill=black, inner sep=1.5pt] (E) at (3,0) {};
\node[circle, fill=black, inner sep=1.5pt] (F) at (4,0) {};
		
\draw [thick] (A) -- (B) -- (C) -- (D) -- (G) -- (E) -- (F);
		
\node  [below]   at (A)  {$-2$};
\node  [below]   at (B)  {$-2$};
\node  [below]   at (C)  {$-2$};
\node  [below]   at (D)  {$-2$};
\node  [below]   at (G)  {$-2$};
\node  [below]  at (E)   {$-1$};
\node  [below]  at (F)   {$-6$};
\node  [above]   at (A)  {$E_1$};
\node [above]   at (B)  {$E_2$};
\node [above]   at (C)  {$E_3$};
\node [above]   at (D)  {$E_4$};
\node [above]   at (G)  {$E_5$};
\node  [above]  at (E)   {$D_1$};
\node  [above]  at (F)    {$F_1$};
\end{tikzpicture}
\end{center}

\item $\pi_6^{\prime-1}(p_2)=3D_2+E_6+2E_7+F_2$ with dual graph
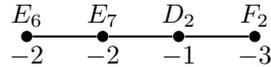
\begin{center}
\begin{tikzpicture}[scale=0.5]
		\node[circle, fill=black, inner sep=1.5pt] (A) at (0,0)  {};
		\node[circle, fill=black, inner sep=1.5pt] (B) at (2,0)  {};
		\node[circle, fill=black, inner sep=1.5pt] (C) at (4,0)  {};
		\node[circle, fill=black, inner sep=1.5pt] (D) at (6,0)  {};
		
		
		\draw [thick] (A)--(B)--(C)--(D);
		
		\node [below]    at (A)  {$-2$};
		\node [below]    at (B)  {$-2$};
		\node [below]    at (C)  {$-1$};
		\node [below]    at (D)  {$-3$};
		
		\node [above]    at (A)  {$E_{6}$};
		\node [above]    at (B)  {$E_{7}$};
		\node [above]    at (C)  {$D_2$};
		\node [above]    at (D)   {$F_2$};
\end{tikzpicture}
\end{center} 
\item $\pi_6^{\prime-1}(p_3)=2D_3+E_8+E_9$ with dual graph
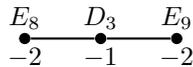
\begin{center}
\begin{tikzpicture}[scale=0.5]
		\node[circle, fill=black, inner sep=1.5pt] (A) at (0,0)  {};
		\node[circle, fill=black, inner sep=1.5pt] (B) at (2,0)  {};
		\node[circle, fill=black, inner sep=1.5pt] (C) at (4,0)  {};
		
		
		\draw [thick] (A)--(B)--(C);
		
		\node [below]    at (A)  {$-2$};
		\node [below]    at (B)  {$-1$};
		\node [below]    at (C)  {$-2$};
		
		\node [above]    at (A)  {$E_8$};
		\node [above]    at (B)  {$D_3$};
		\node [above]    at (C)  {$E_9$};
\end{tikzpicture}
\end{center} 
\end{enumerate}
\end{lemma}

\begin{proof}

We only need to prove case (1). Observe that $\wt\pr_1^{-1}(p_1) = 6\wt D_1$, where $\wt D_1$ is a smooth rational curve containing two singular points of types $\frac{1}{6}(1,5)$ and $\frac{1}{6}(1,1)$. Let $\cup_{i=1}^{5} E_i$ and $F_1$ denote the corresponding exceptional curves, and let $D_1$ be the strict transform of $\wt D_1$. Then $\pi_6^{\prime-1}(p_1) = 6D_1+\sum_{1\leq i\leq 5} n_iE_i + nF_1$ with $n_1,\dots,n_5,n\in\mbb N$ satisfying $(6D_1 + \sum_{1\leq i\leq 5} n_i E_i + n F_1)\cdot D_1 = 0$. By the adjunction formula, $(6D_1 + \sum_{i=1}^{5} n_i E_i + n F_1)\cdot K_{\wt X_6} = -2$, so that $6D_1\cdot K_{\wt X_6} = -2 - 4n$. Hence, $D_1^2 < 0$ and $D_1 \cdot K_{\wt X_6} < 0$, i.e., $D_1$ is an exceptional curve. It follows that $n = 1$ and $\sum_{1\leq i\leq 5}n_i^2-(n_1n_2 + n_2n_3 + n_3n_4 + n_4n_5) = 15$. Up to reordering, we can take $n_1 = 1$, $n_2 = 2$, $n_3 = 3$, $n_4 = 4$, $n_5 = 5$, which gives $D_1\cdot E_1 = D_1\cdot E_2 = D_1\cdot E_3 = D_1\cdot E_4 = 0$, $D_1\cdot E_5 = 1$ and $D_1\cdot F_1 = 1$.

\end{proof}

\begin{lemma}\label{lemm ell 6 E}
	There exist two smooth rational curves $D_0$ and $D_\infty$ on $\wt X_6$ such that the singular fibers of the elliptic fibration $\pi_6$ are as follows:
	\begin{enumerate}[$(1)$]
		\item $\pi_6^{-1}(0) = 6D_0 + E_5 + 2E_4 + 3E_3 + 4E_2 + 5E_1 + 4E_6 + 2E_7 + 3E_8$,  
		which is of type $\uppercase\expandafter{\romannumeral 2}^*$ $(\wt E_8)$;
		\item $\pi_6^{-1}(\infty) = 6D_\infty + F_1 + 2F_2 + 3E_9$,  
		whose dual graph is
		
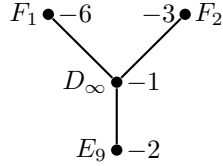
\begin{center}
			\begin{tikzpicture}[scale=0.45]
				\node[circle, fill=black, inner sep=1.5pt] (A) at (0,2)  {};
				\node[circle, fill=black, inner sep=1.5pt] (B) at (2,0)  {};
				\node[circle, fill=black, inner sep=1.5pt] (C) at (4,2)  {};
				\node[circle, fill=black, inner sep=1.5pt] (D) at (2,-2)  {};
				
				\draw [thick] (A)--(B)--(C);
				\draw [thick] (B)--(D);
				
				\node [right]    at (A)    {$-6$};
				\node [right]    at (B)    {$-1$};
				\node [left]     at (C)    {$-3$};
				\node [right]    at (D)    {$-2$};
				
				\node [left]     at (A)    {$F_1$};
				\node [left]     at (B)    {$D_\infty$};
				\node [right]    at (C)    {$F_2$};
				\node [left]     at (D)    {$E_9$};
			\end{tikzpicture}
		\end{center} 
	\end{enumerate}
Moreover, after performing three successive blow-downs of the exceptional curves over $\infty$, one obtains a relatively minimal elliptic fibration whose singular fiber over $\infty$ is of type $\uppercase\expandafter{\romannumeral 2}$.
\end{lemma}

\begin{proof}
As before, we have $\wt\pr_2^{-1}(0)=6\wt D_0$ for a smooth rational curve $\wt D_0$ containing three singular points of types $\frac{1}{6}(1,5)$, $\frac{1}{3}(1,2)$, and $\frac{1}{2}(1,1)$. Let $D_0$ denote its strict transform, so that
\begin{equation}\label{equ ell 6 div 0}
\textstyle{\pi_6^{-1}(0) = 6D_0 + \sum_{1\leq i\leq 8} \wt n_i E_i}
\end{equation}
for some natural numbers $\{\wt n_i\}$. From $6 D_0^2 + \sum_{i=1}^{8} \wt n_i E_i \cdot D_0 = 0$, we deduce $D_0^2 < 0$. The adjunction formula gives $K_{\wt X_6}\cdot(6 D_0 + \sum_{1\leq i\leq 8} \wt n_i E_i) = 0$, and hence $K_{\wt X_6}\cdot D_0 = 0$. Since $\frac{1}{2}(D_0^2 + K_{\wt X_6} \cdot D_0 + 2)$ is a nonnegative integer, 
we conclude $D_0^2 = -2$. Consider the real vector space with basis $\{D_0, E_i\}_{1\le i \le 8}$, 
on which the intersection form with annihilator (\ref{equ ell 6 div 0}) corresponds to the affine Dynkin diagram $\wt E_8$ (see Zariski's Lemma and Lemma 2.12 in \cite{bhpv04}).  
Using the projection formula and the fact that $\wt D_0\cdot \wt D_1 = 1/6$, we obtain $D_0 \cdot \bigl(6 D_1 + \sum_{i=1}^{5} i E_i + F_1\bigr) = 1$, so that
\[
D_0 \cdot E_i = 
\begin{cases}
	1, & i=1,\\
	0, & \text{otherwise}.
\end{cases}
\] 
Hence, we can write $\pi_6^{-1}(0) = 6 D_0 + E_5 + 2E_4 + 3E_3 + 4E_2 + 5E_1 + 4E_6 + 2E_7 + 3E_8$, which is of type $\uppercase\expandafter{\romannumeral 2}^*$ $(\wt E_8)$. For the fiber over $\infty$, we have $\pi_6^{-1}(\infty) = 6 D_\infty + m_1 F_1 + m_2 F_2 + m_3 E_9$, for some natural numbers $m_1, m_2, m_3$ (see Figure~\ref{fig ell 6 1}). Note that $D_\infty^2 < 0$ and $K_{\wt X_6}\cdot D_\infty < 0$. So $D_\infty$ is an exceptional curve of the first kind.  

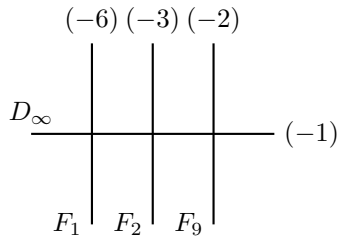
\begin{figure}[H]
	\centering
	\begin{tikzpicture}[scale=0.8]
		\coordinate (A) at (0,0); \coordinate (B) at (4,0);
		\draw [thick] (A)--(B) node[pos=0,above]{$D_\infty$} node[pos=1,right]{$(-1)$};
		\coordinate (D) at (1,1.5); \coordinate (E) at (1,-1.5);
		\draw [thick] (D)--(E) node[left]{$F_1$} node[pos=0,above]{$(-6)$};
		\coordinate (F) at (2,1.5); \coordinate (G) at (2,-1.5);
		\draw [thick](F)--(G) node[left]{$F_2$} node[pos=0,above]{$(-3)$};
		\coordinate (H) at (3,1.5); \coordinate (I) at (3,-1.5);
		\draw [thick](H)--(I) node[left]{$F_9$} node[pos=0,above]{$(-2)$};
	\end{tikzpicture}
	\caption{The singular fiber over $\infty$ in $\wt X_6$}
	\label{fig ell 6 1}
\end{figure}

Blowing down $D_\infty$ yields a new elliptic fibration $\pi_6^{(1)} : \wt X_6^{(1)} \to \mds {P}^1$, with singular fiber $\pi_6^{(1)-1}(\infty) = F_1^{(1)} + 2 F_2^{(1)} + 3 E_9^{(1)}$, where $F_1^{(1)}, F_2^{(1)}, E_9^{(1)}$ are the birational transforms of $F_1, F_2, E_9$ (Figure~\ref{fig ell 6 2}).  

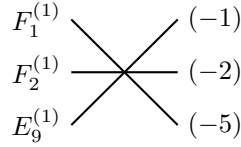
\begin{figure}[H]
	\centering
	\begin{tikzpicture}[scale=0.7]
		\draw [thick] (-1,1)--(1,-1) node[pos=0,left]{$F_1^{(1)}$} node[pos=1,right]{$(-5)$};
		\draw [thick] (-1,0)--(1,0) node[pos=0,left]{$F_2^{(1)}$} node[pos=1,right]{$(-2)$};
		\draw [thick] (-1,-1)--(1,1) node[pos=0,left]{$E_9^{(1)}$} node[pos=1,right]{$(-1)$};
	\end{tikzpicture}
	\caption{After blowing down $D_\infty$}
	\label{fig ell 6 2}
\end{figure}

Next, $E_9^{(1)}$ is an exceptional curve of the first kind. Blowing it down gives $\pi_6^{(2)} : \wt X_6^{(2)} \to \mds {P}^1$, with singular fiber $\pi_6^{(2)-1}(\infty) = F_1^{(2)} + 2 F_2^{(2)}$, where $F_1^{(2)}$ and $F_2^{(2)}$ are birational transforms of $F_1^{(1)}$ and $F_2^{(1)}$ (Figure~\ref{fig ell 6 3}).  

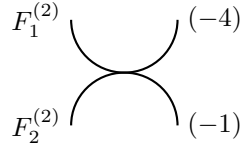
\begin{figure}[H]
	\centering
	\begin{tikzpicture}[scale=0.7]
		\draw [thick] (0,1) arc (0:-180:1) node[pos=0,right]{$(-4)$} node[pos=1,left]{$F_1^{(2)}$};
		\draw [thick] (0,-1) arc (0:180:1) node[pos=0,right]{$(-1)$} node[pos=1,left]{$F_2^{(2)}$};
	\end{tikzpicture}
	\caption{After blowing down $E_9^{(1)}$}
	\label{fig ell 6 3}
\end{figure}

Here, $F_1^{(2)}$ and $F_2^{(2)}$ intersect at one point, with $F_1^{(2)2} = -4$, $F_2^{(2)2} = -1$ and $F_1^{(2)}\cdot F_2^{(2)} = 2$. Finally, blowing down $F_2^{(2)}$ gives $\pi_6^{(3)} : \wt X_6^{(3)} \to \mds P^1$, whose singular fiber over $\infty$ is a cuspidal rational curve $F_1^{(3)}$ (Figure~\ref{fig ell 6 4}).  

\begin{figure}[H]
	\centering
	\begin{tikzpicture}[scale=0.7]
		\draw [thick] (0,1.5) arc (0:-90:1.5) node[pos=0,left]{$F_1^{(3)}$};
		\draw [thick] (0,-1.5) arc (0:90:1.5);
	\end{tikzpicture}
	\caption{Cuspidal fiber over $\infty$ after final blow-down}
	\label{fig ell 6 4}
\end{figure}
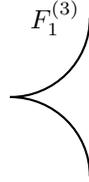

\end{proof}

We now show that $\wt X_6$ arises from the second Hirzebruch surface via a sequence of blow-ups. Let $C_1$, $C_2$, and $C_3$ be the fibers of $h_2 : H_2 \to \mds {P}^1$ over the orbifold points $p_1$, $p_2$, $p_3$. Blowing up $E_0 \cap C_i$ ($i=1,2,3$) yields a new fibration $h_2^{(1)} : H_2^{(1)} \to \mds {P}^1$ with exceptional curves $E_i^{(1)}$. Blowing up the three points $E_i^{(1)}\cap\wt C_i$, where $\wt C_i$ denote the strict transforms of $C_i$, gives $h_2^{(2)} : H_2^{(2)} \to \mds{P}^1$ with exceptional curves $E_i^{(2)}$. Next, blowing up $E_j^{(2)} \cap E_j^{(1,1)}$ ($j=1,2$), where $E_j^{(1,1)}$ denotes the strict transform of $E_j^{(1)}$, produces $h_2^{(3)} : H_2^{(3)} \to \mds{P}^1$ with exceptional curves $E_1^{(3)}$ and $E_2^{(3)}$. Let $E_1^{(1,2)}$ denote the strict transform of $E_1^{(1,1)}$. Blowing up the point $E_1^{(3)} \cap E_1^{(1,2)}$ yields $h_2^{(4)} : H_2^{(4)} \to \mds {P}^1$. Proceeding in this manner, we perform two further blow-ups and finally obtain $h_2^{(6)} : H_2^{(6)} \to \mds {P}^1$.

\begin{proposition}\label{pro ell 6 str}
$\pi_6^\prime : \wt X_6\rightarrow\mds P^1$ is isomorphic to $h_2^{(6)} : H_2^{(6)}\rightarrow\mds P^1$.
\end{proposition}

\begin{proof}
According to Lemmas \ref{lemm ell 6 H} and \ref{lemm ell 6 E}, the surface $\wt X_6$ can be depicted as in Figure~\ref{fig ell 6 5}. As in the proof of Proposition \ref{pro ell 4 str}, this figure illustrates the configuration of the curves, from which the proposition follows.
	
\begin{figure}[H]
\centering
\begin{tikzpicture}[xscale=0.8, yscale=0.5]
		    \coordinate (A) at (0,0);
			\coordinate (B) at (8,0);
			\coordinate (C) at (0,7);
			\coordinate (D) at (8,7);
			\draw [thick] (A)--(B)
			node[pos=0,left]{$D_0$};
			\draw [thick] (C)--(D)
			node[pos=0,left]{$D_\infty$};
			
			\coordinate (A1) at (0.5,-0.5);
			\coordinate (A2) at (2.5,1.5);
			\coordinate (B1) at (2.5,0.5);
			\coordinate (B2) at (0.5,2.5);
			\coordinate (C1) at (0.5,1.5);
			\coordinate (C2) at (2.5,3.5);
			\coordinate (D1) at (2.5,2.5);
			\coordinate (D2) at (0.5,4.5);
			\coordinate (E1) at (0.5,3.5);
			\coordinate (E2) at (2.5,5.5);
			\coordinate (F1) at (2.5,4.5);
			\coordinate (F2) at (0.5,6.5);
			\coordinate (G1) at  (0.5,5.5);
			\coordinate (G2) at (2.5,7.5);
			\draw [thick] (A1)--(A2)
			node[pos=0,left]{$E_1$};
			\draw [thick] (B1)--(B2)
			node[pos=0,left]{$E_2$};
			\draw [thick] (C1)--(C2)
			node[pos=0,left]{$E_3$};
			\draw [thick] (D1)--(D2)
			node[pos=0,left]{$E_4$};
			\draw [thick] (E1)--(E2)
			node[pos=0,left]{$E_5$};
			\draw [thick] (F1)--(F2)
			node[pos=0,left]{$D_1$};
			\draw [thick] (G1)--(G2)
			node[pos=0,left]{$F_1$};
			
			
			\coordinate (A3) at (3.5,-0.5);
			\coordinate (A4) at (4.5,2.5);
			\coordinate (B3) at (4.5,1.5);
			\coordinate (B4) at (3.5,4.5);
			\coordinate (C3) at (3.5,3.5);
			\coordinate (C4) at (4.5,6.5);
			\coordinate (D3) at (4.5,5.5);
			\coordinate (D4) at (3.5,7.5);
			
			\draw [thick] (A3)--(A4)
			node[pos=0,left]{$E_6$};
			\draw [thick] (B3)--(B4)
			node[pos=0,right]{$E_7$};
			\draw [thick] (C3)--(C4)
			node[pos=0,left]{$D_2$};
			\draw [thick] (D3)--(D4)
			node[pos=0,right]{$F_2$};
			
			\coordinate (A5) at (5,-0.5);
			\coordinate (A6) at (6,2.5);
			\coordinate (B5) at (6,2);
			\coordinate (B6) at (5,5);
			\coordinate (C5) at (5,4.5);
			\coordinate (C6) at (6,7.5);
			
			\draw [thick] (A5)--(A6)
			node[pos=0,left]{$E_8$};
			\draw[thick] (B5)--(B6)
			node[pos=0,right]{$D_3$};
			\draw[thick] (C5)--(C6)
			node[pos=0,left]{$E_9$};
\end{tikzpicture}
\caption{Configuration of the curves on $\wt X_6$.}
\label{fig ell 6 5}
\end{figure}
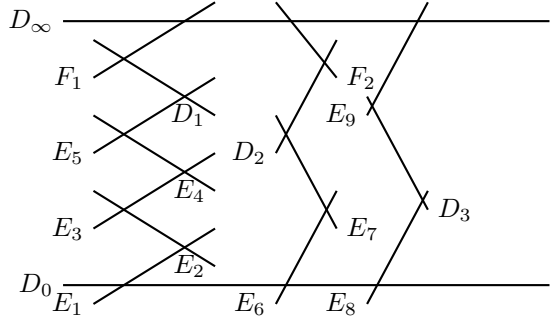

\end{proof}

\section{Appendix}
Throughout this appendix, we always assume that $\mc X$ is a Deligne-Mumford stack of finite type over $k$ with finite diagonal, whose coarse moduli space is $\pi : \mc X\rightarrow X$.

\begin{appendices}

\section{Orbifold Chern character and Euler form}\label{sec orb chern char}
Without loss of generality, assume $\mc X$ is a projective orbifold surface with a single stacky point $p$, whose stabilizer group is $G$. Let $\iota: BG\hookrightarrow \mc X$ denote the residue gerbe at $p$. Define $\mc O_p:=\iota_*\mc O_{BG}$ and, for any $G$-representation $\rho$, set $\mc O_p\otimes\rho:=\iota_*V_\rho$, where $V_\rho$ is the locally free sheaf on $BG$ corresponding to $\rho$. For $g\in G$, let $(g)$ be its conjugacy class, $C(g)$ its centralizer, and $\mc T$ the set of conjugacy classes of $G$. Since $K_0(BG)\cong K(G)$, the pullback $\iota^*$ induces a homomorphism $\iota^*: K_0(\mc X)\to K(G)$, and for any coherent sheaf $E$ on $\mc X$, we write $\rho_E := \iota^*[E]$ for the associated virtual $G$-representation.

\subsection{Orbifold Chern character}

The inertia stack decomposes as
\begin{equation*}
	I\mc X = \mc X \,\coprod\, \coprod_{(g)\in\mc T\setminus\{(1)\}} BC(g),
\end{equation*}
and its cohomology splits accordingly:
\begin{equation}\label{equ decom of cohomgy}
	H^*(I\mc X,\mbb C) = H^*(\mc X,\mbb C) \,\oplus\, \bigoplus_{(g)\in\mc T\setminus\{(1)\}} H^*(BC(g),\mbb C).
\end{equation}
With respect to this decomposition, the orbifold Chern character of a coherent sheaf $E$ is
\begin{equation*}
	\wt{\ch}(E) = \big(\ch(E), (\ch^{(g)}(E))_{(g)\in\mc T\setminus\{(1)\}}\big),
\end{equation*}
where $\ch(E)\in H^{\rm even}(\mc X,\mbb C)$ and $\ch^{(g)}(E)\in\mbb C$.  

\begin{lemma}\label{lemm chern ch of vb}
	For $(g)\in\mc T\setminus\{(1)\}$, $\ch^{(g)}(E) = \chi_{\rho_E}(g)$.
	\qed
\end{lemma}

\begin{lemma}\label{lemm equiv koszul cp}
	Let $G$ act on $\mbb C[\![x,y]\!]$ via a 2-dimensional representation $\tau$. Then there is a $G$-equivariant exact sequence
	\begin{equation}\label{equ equiv koszul seq}
		0 \to \mbb C[\![x,y]\!]\otimes \det(\tau) \to \mbb C[\![x,y]\!]\otimes \tau \to \mbb C[\![x,y]\!] \to \mbb C \to 0.
	\end{equation}
	\qed
\end{lemma}

\begin{proposition}\label{pro ch of rho}
Let $\rho$ be a $G$-representation of degree $d$. Then 
\begin{gather*}
	\ch(\mc O_p\otimes\rho)=(0,0,d/|G|),\\
	\ch^{(g)}(\mc O_p\otimes\rho)=\dett(\id-\rho_{\Omega_\mc X}(g))\tr(\rho(g)),\quad (g)\in\mc T\setminus\{(1)\}.
\end{gather*}
\end{proposition}

\begin{proof}
By applying Lemma 1, the conclusion of this proposition follows by direct computation
\end{proof}


\subsection{Euler form and $K^{\rm num}(\mc X)$}\label{subsec euler form}
As in the case of schemes, the Euler characteristic of a pair $(E_1,E_2)$ of coherent sheaves is
\begin{equation*}
		\chi(E_1,E_2)=\sum_{0\leq i\leq 2}(-1)^i\dimm_\mbb C\ext^i(E_1,E_2).
\end{equation*}
By To\"{e}n-Riemann-Roch formula, we have
\begin{equation*}
		\chi(E_1,E_2)
		=\int_{I\mc X}\wt{\ch}(E_1^\vee)\wt{\ch}(E_2)\wt{\td}(T\mc X).
\end{equation*}
Here $(\cdot)^\vee : K_0(\mc X)\rightarrow K_0(\mc X)$ is the involution, which extends the operation ``taking the dual of a locally free sheaf'' linearly to whole $K_0(\mc X)$.
	
\begin{definition}
The Euler form $\bar{\chi} : K_0(\mc X)\times K_0(\mc X)\rightarrow\mathds Z$ is a biadditive integer-valued function determined by $\bar{\chi}([E_1],[E_2]):=\chi(E_1,E_2)$.
\end{definition}
	
By a straightforward computation, it is easy to prove the following formulas.
\begin{proposition}\label{pro comp of euler form}
\begin{enumerate}[$(1)$]
\item $\chi(E_1,E_2)=\chi(E_2,E_1\otimes K_\mc X)$.
\item $\chi(\mc O_\mc X,\mc O_\mc X)=\chi(\mc O_\mc X)$, where $\chi(\mc O_\mc X)$ is the Euler characteristic of $\mc O_\mc X$;
\item $\chi(\mc O_q,\mc O_q)=0$ and $\chi(\mc O_\mc X,\mc O_q)=\chi(\mc O_q,\mc O_\mc X)=1$;
\item $\chi(\mc O_p\otimes\rho_i,\mc O_p\otimes\rho_j)=\langle\chi_{\rho_j},\chi_{\rho_i}\rangle-\langle\chi_{\rho_j},\chi_{\rho_i\otimes\rho_{\Omega_\mc X}}\rangle+\langle\chi_{\rho_j},\chi_{\rho_i\otimes\rho_{K_\mc X}}\rangle$;
\item $\chi(\mc O_\mc X,\mc O_p\otimes\rho_i)=\langle\chi_{\rho_i},\chi_{\rho_0}\rangle$;
\item $\chi(\mc O_p\otimes\rho_i,\mc O_\mc X)=\langle\chi_{\rho_0},\chi_{\rho_i\otimes\rho_{K_{\mc X}}}\rangle$.
\end{enumerate}
\qed
\end{proposition}

\begin{definition}\label{def num K gp}
$K^{\rm num}(\mc X):=K_0(\mc X)/I$, where $I$ is the subgroup of $K_0(\mc X)$ consisting of those $\kappa_1$, which satisfy the condition: $\bar{\chi}(\kappa_1,\kappa_2)=\bar{\chi}(\kappa_2,\kappa_1)=0$ for all $\kappa_2\in K_0(\mc X)$. The element of $K^{\rm num}(\mc X)$ is called numerical K-class of $\mc X$.
\end{definition}
	
\begin{lemma}
The kernel of $\wt{\ch} : K_0(\mc X)\rightarrow H^{\rm even}(\mc X,\mbb C)$ is $I$.
\qed
\end{lemma}

\begin{corollary}\label{cor chern char of num k gp}
The orbifold chern character map $\wt{\ch} : K^{\rm num}(\mc X)\rightarrow H^{\rm even}(I\mc X,\mbb C)$ is an injective map.
\qed
\end{corollary}

\section{Hodge Index Theorem for orbifold surfaces}\label{subsec hodge indx}

We recall some basic facts on the de Rham and Hodge theory of projective orbifolds 
(i.e.\ smooth projective Deligne--Mumford stacks with trivial generic stabilizers); 
see \cite{baily56, baily57, steen76, beh04}. Let $\Theta$ be a Kähler metric on $\mc X^{\rm an}$ representing $c_1$ of an ample line bundle. Then

\begin{enumerate}[(i)]
	\item $H^m(\mc X,\ulin{\mbb C})=H^m(X,\ulin{\mbb C})$.
	\item $H^m(\mc X^{\rm an},\ulin{\mbb Z})\otimes\mbb C
	=H^m(\mc X^{\rm an},\ulin{\mbb C})
	=H^m_{DR}(\mc X^{\rm an})\otimes\mbb C$.
	\item $\dim_\mbb C H^{p,q}(\mc X^{\rm an})<\infty$.
	\item (Dolbeault--Kodaira) $H^q(\mc X^{\rm an},\Omega^p_\mc X)=H^{p,q}(\mc X^{\rm an})$.
	\item (Hodge decomposition) 
	\[
	H^m(\mc X^{\rm an},\ulin{\mbb C})
	=\bigoplus_{p+q=m} H^{p,q}(\mc X^{\rm an}),\qquad
	H^{p,q}(\mc X^{\rm an})=\ov{H^{q,p}(\mc X^{\rm an})}.
	\]
	\item (Hard Lefschetz) 
	$L^k:H^{n-k}(\mc X^{\rm an},\ulin{\mbb C})\to H^{n+k}(\mc X^{\rm an},\ulin{\mbb C})$ 
	is an isomorphism.  
	Define 
	\[
	P^{n-k}(\mc X^{\rm an})=\ker\!\left(
	L^{k+1}:H^{n-k}(\mc X^{\rm an},\ulin{\mbb C})
	\to H^{n+k+2}(\mc X^{\rm an},\ulin{\mbb C})\right),
	\]
	then
	\[
	H^m(\mc X^{\rm an},\ulin{\mbb C})
	=\bigoplus_k L^k P^{m-2k}(\mc X^{\rm an}).
	\]
	\item (Hodge--Riemann) The form
	\[
	Q(\xi,\eta)=\int_{\mc X}\xi\wedge\eta\wedge\Theta^k
	\]
	is positive definite on $P^{p,q}$ up to the usual sign convention.
\end{enumerate}

\begin{proposition}[Hodge Index Theorem]\label{pro hodge index tm}
	For a smooth projective orbifold surface $\mc X$,
	\[
	\ns_{\mbb R}(\mc X)
	=(P^{1,1}(\mc X^{\rm an})\cap H^2(\mc X^{\rm an},\ulin{\mbb R}))
	\oplus \mbb R\cdot\Theta,
	\]
	and $Q$ restricts to a nondegenerate form of signature $(1,\rho(\mc X)-1)$.
	\qed
\end{proposition}

\section{Base change theorem for relative Ext-sheaves on DM stacks}\label{app rela ext}

Let $\mc X$ be equipped with the small \'{e}tale site, and denote the category of $\mc O_{\mc X}$-modules by $\mathrm{Mod}(\mc X)$. Lemma \ref{lemm enough injec} is standard (see \cite{ol16}). The proofs of Lemmas \ref{lemm  local ext}, \ref{lemm free ext}, \ref{lemm exact se ext} and Corollaries \ref{cor coherence ext}, \ref{cor com ext} are analogous to the scheme case and are omitted.

\begin{lemma}\label{lemm enough injec}
Mod($\mc X$) has enough injective objects.
\end{lemma}

For a morphism of Deligne-Mumford stacks $f:\mc X\to\mc Y$, we have the adjoint functors
\[
f^*:\mathrm{Mod}(\mc Y)\to \mathrm{Mod}(\mc X),\quad
f_*:\mathrm{Mod}(\mc X)\to \mathrm{Mod}(\mc Y).
\]
For $F\in \mathrm{Mod}(\mc X)$, the functor $\Hom_{\mc O_{\mc X}}(F,-)$ is left exact, and so is $f_*\circ \Hom_{\mc O_{\mc X}}(F,-)$.

\begin{definition}
$\Ext_f^i(F,-)$ is defined to be the $i$-th right derived functor of $f_*\circ\Hom_{\mc O_\mc X}(F,-)$.
\end{definition}

\begin{lemma}\label{lemm  local ext}
For $F,G\in\text{Mod}(\mc X)$, $\Ext_f^i(F,G)$ is the sheaf associated to the presheaf:
\begin{equation*}
(U\overset{\alpha}{\rightarrow}\mc Y)\mapsto\ext^i(F|_{\mc X_U},G|_{\mc X_U}),
\end{equation*}
where $U\overset{\alpha}{\rightarrow}\mc Y$ is an \'{e}tale morphsim from a scheme $U$ to $\mc X$.
\qed
\end{lemma}

\begin{corollary}\label{cor coherence ext}
If $F,G$ are two quasi-coherent sheaves on $\mc X$, then $\Ext_f^i(F,G)$ are quasi-coherent on $\mc Y$.
\qed
\end{corollary}

\begin{lemma}\label{lemm free ext}
Assume that $L\in\text{Mod}(\mc X)$ and $N\in\text{Mod}(\mc Y)$ are locally free sheaves. Then, we have
\begin{equation*}
\Ext_f^i(L\otimes F,G)=\Ext_f^i(F,L^\vee\otimes G),\quad \Ext_f^i(F,f^*N\otimes G)=\Ext_f^i(F,G)\otimes N.
\end{equation*}
\qed
\end{lemma}

\begin{lemma}\label{lemm exact se ext}
If $\xymatrix@C=0.5cm{0 \ar[r] & F_1 \ar[r] & F \ar[r] & F_2 \ar[r] & 0 }$ is an exact sequence in Mod($\mc X$), then there is a long exact sequence
\begin{equation*}
\xymatrix@=0.5cm{
  \cdots \ar[r] & {\Ext}_f^i(F_2,G) \ar[r] & {\Ext}_f^i(F,G)  \ar[r] & {\Ext}_f^i(F_1,G) \ar[r] & {\Ext}_f^{i+1}(F_2,G) \ar[r] & \cdots}.
\end{equation*}
\qed
\end{lemma}

\begin{corollary}\label{cor com ext}
Let $\xymatrix@C=0.5cm{L^\bullet \ar[r] & F \ar[r] & 0 }$ be an exact sequence in Mod($\mc X$). If the complex $L^\bullet$ satisfies
\begin{equation*}
\Ext_f^i(L^e,G)=0 \quad\text{for $i\geq 1$},
\end{equation*}
then we get $\Ext_f^i(F,G)=\hh^i(f_*\Hom_{\mc O_{\mc X}}(L^\bullet,G))\quad\text{for all $i$}$.
\qed
\end{corollary}

\begin{definition}\label{def a family of dm}
A flat family of projective stacks over a $k$-scheme $S$ is a flat morphism $p : \mc X\rightarrow S$, where $\mc X$ is a tame separated $S$-global quotient stack (see \cite[Definition 2.9]{ehkv01}) and $p$ factorizes as $\pi : \mc X\rightarrow X$ composed with a projective morphism $q : X\rightarrow S$ i.e.
\begin{equation*}
\xymatrix@=0.5cm{
  \mc X \ar[rr]^{\pi} \ar[dr]_{p} &&    X \ar[dl]^{q}.    \\
                & S}
\end{equation*}
\end{definition}

\begin{lemma}\label{lemm proj}
Suppose that $p : \mc X\rightarrow\spec A$ is a flat family of projective stacks over $A$ with $A$ a finitely generated $k$-algebra. If $E$ is a $A$-flat coherent sheaf on $\mc X$ satisfying: $R^ip_*E=0$ for $i>0$, then for any $A$-algebra $A^\prime$ and $A^\prime$-module $M$
\begin{equation*}
\begin{split}
p_{A^\prime*}(E\otimes_AM)=p_*E\otimes_AM\quad\text{and}\quad
R^ip_{A^\prime*}(E\otimes_{A}M)=0\quad\text{for $i>0$},
\end{split}
\end{equation*}
where
\begin{equation*}
\xymatrix@=0.3cm{
  & X_{A^\prime} \ar[rr] \ar'[d]^{q_{A^\prime}}[dd]   &&    X   \ar[dd]^{q}        \\
  \mc X_{A^\prime}\ar[ur]^{\pi_{A^\prime}}\ar[rr]  \ar[dr]_{p_{A^\prime}}  && \mc X \ar[ur]^{\pi}\ar[dr]^{p} \\
  &  \spec A^\prime \ar[rr]  && \spec A    }.
\end{equation*}
\end{lemma}

\begin{proof}
Since $\pi_{A^\prime}$ is exact (\cite[Theorem 3.2]{aov08}) and $\pi_{A^\prime*}$ sends injective sheaves to flasque sheaves (\cite[Lemma 1.10]{ni08}), it suffice to show
\begin{equation}\label{equ lemm proj 1}
\begin{split}
q_{A^\prime*}(\pi_{A^\prime*}(E\otimes_AM))&=q_{A^\prime*}(\pi_*E\otimes_AM)=p_*E\otimes_AM,\\
R^iq_{A^\prime*}(\pi_{A^\prime*}(E\otimes_{A}M))&=R^iq_{A^\prime*}(\pi_*E\otimes_{A}M)=0
\end{split}
\end{equation}
by \cite[Prop.~1.5, Lem.~2.8]{ni08}. Since $q : X\rightarrow\spec(A)$ is projective, there exists a closed embedding $\iota : X\rightarrow\mds P^n_A$ i.e.
For the standard affine cover $\mfr U=(D_+(x_j))$ of $\proj(A[x_0,x_1,\cdots,x_n])$, the \v{C}ech complex of $E$ gives an exact sequence of $A$-modules
\begin{equation}\label{equ lemm proj 2}
\xymatrix@C=0.5cm{
  0 \ar[r] & H^0(\mds P^n_A,\iota_*\pi_*E) \ar[r] & C^0(\mfr U,\iota_*\pi_*E) \ar[r] & \cdots \ar[r] & C^n(\mfr U,\iota_*\pi_*E) \ar[r] & 0 },
\end{equation}
whose terms are flat $A$-modules. Hence, for any $A$-algebra $A^\prime$ and $A^\prime$-module $M$, $H^0(X_{A^\prime},\pi_*E\otimes_AM)=H^0(X,\pi_*E)\otimes_AM$ and $H^i(X_{A^\prime},\pi_*E\otimes_AM)=0$ for $i>0$.
\end{proof}

\begin{proposition}\label{pro comp ext}
Suppose that $p : \mc X\rightarrow S$ is a flat family of projective stacks over a $k$-scheme $S$ of finite type, and let $F,G$ be $S$-flat coherent sheaves on $\mc X$. Then there exists a locally free resolution $\xymatrix@=0.5cm{ L^\bullet \ar[r] & F \ar[r] & 0 }$ with the universal property: for any morphism of finite type $u : S^\prime\rightarrow S$, and any quasicoherent $\mc O_{S'}$-module $M$, one has
\begin{equation*}
\Ext_{p_{S'}}^i(F\otimes\mc O_{S'},\, G\otimes M)
=\hh^i\!\big(p_{S'*}\Hom(L^\bullet\otimes\mc O_{S'},\, G\otimes M)\big).
\end{equation*}
where
\begin{equation*}
\xymatrix@=0.3cm{
  & X_{S^\prime} \ar[rr] \ar'[d]^{q_{S^\prime}}[dd]   &&    X   \ar[dd]^{q}        \\
  \mc X_{S^\prime}\ar[ur]^{\pi_{S^\prime}}\ar[rr] \ar[dr]_{p_{S^\prime}}  && \mc X \ar[ur]^{\pi}\ar[dr]^{p} \\
  &  S^\prime \ar[rr]^{u} &&   S     }
\end{equation*}

\end{proposition}

\begin{proof}
For simplicity, we assume $S=\spec(A)$. Fix a generating sheaf $\mc E$ on $\mc X$. The natural morphism
\begin{equation}\label{equ pro comp ext 1}
	(\pi^*\pi_*(\mc E^\vee\otimes F))\otimes\mc E \to F
\end{equation}
is surjective \cite{os03}, and $\pi_*(\mc E^\vee\otimes F)$ is $S$-flat \cite[Cor.~3.3]{aov08}.  
By Serre vanishing, there exists $N$ such that $H^i(X,\pi_*(\mc E^\vee\otimes F)(l))=0 \quad (i>0,\, l\ge N)$, hence the Mumford-Castelnuovo regularities  ${\rm reg}(\pi_*(\mc E^\vee\otimes F)\otimes_Ak(s))\le N+n$ for all $s\in S$. Moreover, $q^*q_*(\pi_*(\mc E^\vee\otimes F)(l))$ is locally free and $q^*q_*(\pi_*(\mc E^\vee\otimes F)(l))\rightarrow\pi_*(\mc E^\vee\otimes F)(l)$ is surjective. Writing $V=q^*q_* (\pi_*(\mc E^\vee\otimes F)(l))$, we thus obtain a surjection
\begin{equation}\label{equ app gl}
	\pi^*V(-l)\otimes\mc E \twoheadrightarrow F.
\end{equation}
For any $G$, Lemma~\ref{lemm free ext} gives
\begin{equation}\label{equ pro comp ext 3}
	\Ext^i_p(\pi^*V(-l)\otimes\mc E,G)
	=R^iq_*(V^\vee(l)\otimes\pi_*(\mc E^\vee\otimes G)).
\end{equation}
As before, there exists $\wt N$ with ${\rm reg}(\pi_*(\mc E^\vee\otimes G)\otimes_Ak(s))\le\wt N$, so by \cite[Thm.~12.11]{ha77}, $R^iq_*(V^\vee(l)\otimes\pi_*(\mc E^\vee\otimes G))=0 \quad (i>0,\, l\ge \wt N)$.
Thus
\begin{equation}\label{equ pro comp ext 5}
	\Ext^i_p(\pi^*V(-l)\otimes\mc E,G)=0 \quad (i>0,\, l\ge \wt N).
\end{equation}
Iterating \eqref{equ app gl}, we construct a locally free resolution
\begin{equation}\label{equ pro compl ext 6}
	L^\bullet \to F \to 0
\end{equation}
with $\Ext^i_p(L^e,G)=0$ for $i>0$, $e\ge0$. Hence $\Ext^i_p(F,G)=\hh^i(p_*\Hom(L^\bullet,G))$. For $S'=\spec(A')$, the base change $L^\bullet\otimes_AA'$ is again a locally free resolution of $F\otimes_AA'$. By Lemma~\ref{lemm proj}, $\Ext^i_{p_{S'}}(L^e\otimes_AA',\, G\otimes_AM)=0 \quad (i>0,\, e\ge0)$, thus \eqref{equ pro compl ext 6} satisfies the required universal property.

\end{proof}

\begin{corollary}\label{cor pro comp ext}
Under the above hypothesis, there exists a bounded below complex $W^\bullet$ of locally free sheaves on $S$
such that $\Ext^i_{p_{S^\prime}}(F\otimes_{\mc O_S}\mc O_{S^\prime}, G\otimes_{\mc O_S}M)=\hh^i(W^\bullet\otimes_{\mc O_S}M)$.

\end{corollary}

\begin{proof}
By Proposition~\ref{pro comp ext}, there exists a bounded above complex $L^\bullet$ of locally free sheaves such that
\begin{equation}\label{equ cor comp ext 1}
	\Ext_{p_{S^\prime}}^i(F\otimes_{\mc O_S}\mc O_{S^\prime},G\otimes_{\mc O_S}M)
	=\hh^i(p_{S^\prime*}\Hom(L^\bullet\otimes_{\mc O_S}\mc O_{S^\prime},G\otimes_{\mc O_S}M)).
\end{equation}
Since $\Ext^i_p(L^e,G)=R^ip_*({L^e}^\vee\otimes G)=0$ for $i>0$, each $p_*({L^e}^\vee\otimes G)$ is locally free on $S$. Set $W^\bullet=p_*({L^\bullet}^\vee\otimes G)$. Then, by \eqref{equ cor comp ext 1} and Lemma~\ref{lemm proj}, $\Ext_{p_{S^\prime}}^i(F\otimes_{\mc O_S}\mc O_{S^\prime},G\otimes_{\mc O_S}M)
=\hh^i(W^\bullet\otimes_{\mc O_S}M)$.
\end{proof}

We give a simple proof of the base change theorem for relative $\ext$-sheaves on Deligne Mumford stacks which had been proved by Hall (see \cite[Theorem A]{hal14}).
\begin{theorem}\label{thm base ch for ext}
Let $y\in S$ be a point.
\begin{enumerate}[$(1)$]
\item If the natural map $\Ext_f^i(F,G)\otimes k(y)\rightarrow \ext^i(F_y,G_y)$ is surjective, then it is an isomorphism, and the same is true for all $y^\prime$ in a neighborhood of $y$.
\item Assume that the natural map $\Ext_f^i(F,G)\otimes k(y)\rightarrow \ext^i(F_y,G_y)$ is surjective. Then the following are equivalent:
    \begin{enumerate}[$(a)$]
    \item $\Ext_f^{i-1}(F,G)\otimes k(y)\rightarrow\ext^{i-1}(F_y,G_y)$ is surjective;
    \item $\Ext_f^i(F,G)$ is locally free in a neighborhood of $y$.
    \end{enumerate}
\end{enumerate}
\end{theorem}

\begin{proof}
Without loss of generality, let $S=\spec(A)$ with $A$ a finitely generated $k$-algebra.  
Define functors $T^i:\modd(A)\to\modd(A)$ by
\[
M \mapsto \ext^i(F,G\otimes_A M),\quad i\ge0.
\]
Then the proof of \cite[Thm.~12.11]{ha77} applies verbatim, yielding the base change theorem for relative Ext-sheaves.
\end{proof}

\end{appendices}

\section*{Acknowledgement}
{We are grateful to Professor Oleg Chalykh for drawing our attention to his work with Philip Argyres and Yongchao Lü on Poisson deformations of elliptic surfaces in the context of elliptic Cherednik algebras and SCFTs, which is closely related to the orbifold Hilbert schemes studied in this paper.


\end{document}